\documentclass[12pt,english]{amsart}
\usepackage[LGR,T1]{fontenc}
\usepackage[latin9]{inputenc}
\usepackage{geometry}
\usepackage{xcolor}
\usepackage{verbatim}
\usepackage{float}
\usepackage{calc}
\usepackage{enumitem}
\usepackage{amstext}
\usepackage{amsthm}
\usepackage{amssymb}
\usepackage{graphicx}
\usepackage[all]{xy}

\makeatletter

\DeclareRobustCommand{\greektext}{%
  \fontencoding{LGR}\selectfont\def\encodingdefault{LGR}}
\DeclareRobustCommand{\textgreek}[1]{\leavevmode{\greektext #1}}
\ProvideTextCommand{\~}{LGR}[1]{\char126#1}

\theoremstyle{plain}
\newtheorem{thm}{\protect\theoremname}[section]
\theoremstyle{remark}
\newtheorem{rem}[thm]{\protect\remarkname}
\theoremstyle{remark}
\newtheorem*{rem*}{\protect\remarkname}
\theoremstyle{definition}
\newtheorem{defn}[thm]{\protect\definitionname}
\theoremstyle{definition}
\newtheorem{example}[thm]{\protect\examplename}
\theoremstyle{plain}
\newtheorem{lem}[thm]{\protect\lemmaname}
\theoremstyle{plain}
\newtheorem{prop}[thm]{\protect\propositionname}
\theoremstyle{plain}
\newtheorem{cor}[thm]{\protect\corollaryname}

\usepackage{tikz}
\usepackage{tikz-cd}
\usepackage{xargs}
\usetikzlibrary{arrows.meta}

\usepackage{lscape}
\usepackage{adjustbox}

\usepackage{quiver}

\usepackage{inputenc}

\newcommand{\Yo}{\text{\usefont{U}{min}{m}{n}\symbol{'110}}}

\DeclareFontFamily{U}{min}{}
\DeclareFontShape{U}{min}{m}{n}{<-> dmjhira}{}

\newlength{\myline}           % line thickness
\newcommandx*{\triplearrow}[4][1=0, 2=1]{
  	\draw[line width=\myline,double distance=3\myline,#3] #4;
  	\draw[line width=\myline,shorten <=#1\myline,shorten >=#2\myline,#3] #4;
}

\tikzcdset{scale cd/.style={every label/.append style={scale=#1},
    cells={nodes={scale=#1}}}}

\newcommand{\xyR}[1]{
  \xydef@\xymatrixrowsep@{#1}}
\newcommand{\xyC}[1]{
  \xydef@\xymatrixcolsep@{#1}}

\newcommand{\pushoutcorner}[1][dr]{
  \save*!/#1+1.2pc/#1:(-1,1)@^{|-}\restore}

\newcommand{\tabld}[2]{\begin{pmatrix}#1^0_0 &\dots &#1^0_{#2-1}
		&#1^0_{#2}\cr\noalign{\vskip 3pt} #1^1_0 &\dots &#1^1_{#2-1}
		&#1^1_{#2}\end{pmatrix}}

\usepackage{tikzit}
\tikzstyle{adfsaf}=[fill={rgb,255: red,255; green,191; blue,191}, draw=none]
\tikzstyle{Tops}=[fill={rgb,255: red,255; green,234; blue,234}, draw=none]
\tikzstyle{Slanted Text}=[fill=none, draw=none, yslant=.25]
\tikzstyle{new style 0}=[fill=none, draw=black, shape=circle]

\tikzstyle{Dashed Edges No Fill}=[-, fill=none, dashed, shorten <=1cm, shorten >=1cm]
\tikzstyle{Solid arrow}=[->, shorten >=1cm, shorten <=1cm]
\tikzstyle{solid arrow}=[-, line width=0.75pt]
\tikzstyle{arrow}=[<-]
\tikzstyle{Fronts}=[-, fill={rgb,255: red,255; green,191; blue,191}, draw=black, fill opacity=.9]
\tikzstyle{Tops}=[-, fill={rgb,255: red,255; green,225; blue,225}, draw=black, fill opacity=0.75, line width=.75pt]
\tikzstyle{Dotted arrow}=[->, dotted, shorten <=1cm, shorten >=1cm, draw=none]
\tikzstyle{alignment lines}=[-, draw=black, dashdotted, fill=none, line width=0.25pt, shorten <=-50pt, shorten >=-50pt]
\tikzstyle{ghost lines}=[-, draw=none]
\tikzstyle{Left Insides}=[-, fill={rgb,255: red,237; green,178; blue,178}, draw=black]
\tikzstyle{obscured insides}=[-, fill={rgb,255: red,238; green,211; blue,211}, draw=black]
\tikzstyle{Thin Box}=[-, line width=0.25pt]
\tikzstyle{hidden edges}=[-, dashed]

\mathchardef\mhyphen="2D

\makeatother

\usepackage{babel}
\providecommand{\corollaryname}{Corollary}
\providecommand{\definitionname}{Definition}
\providecommand{\examplename}{Example}
\providecommand{\lemmaname}{Lemma}
\providecommand{\propositionname}{Proposition}
\providecommand{\remarkname}{Remark}
\providecommand{\theoremname}{Theorem}

\begin{document}
\begin{quotation}
%%%%%%%%%%%%%%%%%%%%%%%%%%%%%%%% Roman Characters %%%%%%%%%%%%%%%%%%%%%%%%%%%%%

\global\long\def\bA{\mathbf{A}}%

\global\long\def\sA{\mathscr{A}}%

\global\long\def\sB{\mathscr{B}}%

\global\long\def\C{\mathbf{C}}%

\global\long\def\bC{\mathbb{C}}%

\global\long\def\sC{\mathscr{C}}%

\global\long\def\sfC{\mathsf{C}}%

\global\long\def\sD{\mathscr{D}}%

\global\long\def\sfD{\mathsf{D}}%

\global\long\def\sE{\mathscr{E}}%

\global\long\def\bF{\mathbb{F}}%

\global\long\def\cF{\mathcal{F}}%

\global\long\def\sF{\mathscr{F}}%

\global\long\def\bG{\mathbf{G}}%

\global\long\def\bbG{\mathbb{G}}%

\global\long\def\cG{\mathcal{G}}%

\global\long\def\sG{\mathscr{G}}%

\global\long\def\sH{\mathscr{H}}%

\global\long\def\cI{\mathcal{I}}%

\global\long\def\sI{\mathcal{\mathscr{I}}}%

\global\long\def\fI{\mathfrak{I}}%

\global\long\def\bJ{\mathbf{J}}%

\global\long\def\cL{\mathcal{L}}%

\global\long\def\sL{\mathscr{L}}%

\global\long\def\fm{\mathfrak{m}}%

\global\long\def\sM{\mathscr{M}}%

\global\long\def\N{\mathbf{N}}%

\global\long\def\fN{\mathfrak{N}}%

\global\long\def\sN{\mathscr{N}}%

\global\long\def\cO{\mathcal{O}}%

\global\long\def\sO{\mathscr{O}}%

\global\long\def\bP{\mathbf{P}}%

\global\long\def\fp{\mathfrak{p}}%

\global\long\def\bQ{\mathbb{Q}}%

\global\long\def\fq{\mathfrak{q}}%

\global\long\def\sR{\mathscr{R}}%

\global\long\def\R{\mathbb{\mathbf{R}}}%

\global\long\def\RR{\overline{\sR}}%

\global\long\def\cS{\mathscr{\mathcal{S}}}%

\global\long\def\fS{\mathfrak{S}}%

\global\long\def\sS{\mathcal{\mathscr{S}}}%

\global\long\def\bU{\mathbf{U}}%

\global\long\def\sU{\mathfrak{\mathscr{U}}}%

\global\long\def\bV{\mathbf{V}}%

\global\long\def\sV{\mathscr{V}}%

\global\long\def\sfV{\mathsf{V}}%

\global\long\def\sW{\mathscr{W}}%

\global\long\def\sX{\mathcal{\mathscr{X}}}%

\global\long\def\bZ{\mathbb{\mathbf{Z}}}%

\global\long\def\Z{\mathbf{Z}}%

\global\long\def\frZ{\mathfrak{Z}}%

%%%%%%%%%%%%%%%%%%%%%%%%%%%%%%%% Greek Character Macros %%%%%%%%%%%%%%%%%%%%%%%%%%%%%%%%%%%%%%%

\global\long\def\e{\text{\ensuremath{\varepsilon}}}%

\global\long\def\G{\Gamma}%

%%%%%%%%%%%%%%%%%%%%%%%%%%%%%%%% Miscellaneous Character Macros %%%%%%%%%%%%%%%%%%%%%%%%%%%%%%%

\global\long\def\p{\mathcal{\prime}}%

\global\long\def\srn{\sqrt{-5}}%

%%%%%%%%%%%%%%%%%%%%%%%%%%%%%%%% Decoration Abbreviation %%%%%%%%%%%%%%%%%%%%%%%%%%%%%%%

\global\long\def\wh#1{\widehat{#1}}%

%%%%%%%%%%%%%%%%%%%%%%%%%%%%%%%% Miscellaneous Relational Macros %%%%%%%%%%%%%%%%%%%%%%%%%%%%%%

\global\long\def\nf#1#2{\nicefrac{#1}{#2}}%

%%%%%%%%%%%%%%%%%%%%%%%%%%%%%%%% Set Theoretic Macros %%%%%%%%%%%%%%%%%%%%%%%%%%%%%%%%%%%%%%%%%

\global\long\def\E{\mathsf{Ens}}%

\global\long\def\Fin{\mathsf{Fin}}%

\global\long\def\Ord{\mathsf{Ord}}%

\global\long\def\S{\mathsf{Set}}%

\global\long\def\sd#1{\left.#1\right|}%

%%%%%%%%%%%%%%%%%%%%%%%%%%%%%%%% Category Theoretic Macros %%%%%%%%%%%%%%%%%%%%%%%%%%%%%%%%%%%%

\global\long\def\Aut{\text{\ensuremath{\mathsf{Aut}}}}%

\global\long\def\Arr#1{\mathsf{Arr}\left(#1\right)}%

\global\long\def\Cat{\mathsf{Cat}}%

\global\long\def\coker{\mathrm{coker}}%

\global\long\def\colim{\underset{\longrightarrow}{\lim}}%

\global\long\def\lcolim#1{\underset{#1}{\colim}}%

\global\long\def\End{\mathsf{End}}%

\global\long\def\Ext{{\rm Ext}}%

\global\long\def\id{\text{id}}%

\global\long\def\im{\text{im}}%

\global\long\def\Im{\text{Im}}%

\global\long\def\iso{\overset{\sim}{\rightarrow}}%
\global\long\def\osi{\overset{\sim}{\leftarrow}}%

\global\long\def\liso{\overset{\sim}{\longrightarrow}}%
\global\long\def\losi{\overset{\sim}{\longleftarrow}}%

\global\long\def\Hom{\mathsf{Hom}}%

\global\long\def\BHom{\mathbf{Hom}}%

\global\long\def\Homc{\text{\text{Hom}}_{\mathscr{C}}}%

\global\long\def\Homd{\text{\text{Hom}}_{\mathscr{D}}}%

\global\long\def\Homs{\text{\text{Hom}}_{\mathsf{Set}}}%

\global\long\def\Homt{\text{\text{Hom}}_{\mathsf{Top}}}%

\global\long\def\llim{\underleftarrow{\mathsf{lim}}}%
\global\long\def\clim{\underrightarrow{\mathsf{lim}}}%

\global\long\def\cllim#1{\underset{#1}{\clim}}%

\global\long\def\Mono{\mathsf{Mono}}%

\global\long\def\Mor{\text{\ensuremath{\mathsf{Mor}}}}%

\global\long\def\Ob{\mathsf{Ob}}%

\global\long\def\one{\mathbf{1}}%

\global\long\def\op{\mathsf{op}}%

\global\long\def\Pull{\mathsf{Pull}}%

\global\long\def\Push{\mathsf{Push}}%

\global\long\def\Psh#1{\mathsf{Psh}\left(#1\right)}%

\global\long\def\pt{\text{pt.}}%

\global\long\def\Sh#1{\mathsf{Sh}\left(#1\right)}%

\global\long\def\sh{\text{sh}}%

\global\long\def\Tor{\text{\text{Tor}}}%

\global\long\def\XoY{\nicefrac{X}{Y}}%

\global\long\def\XtY{X\times Y}%

\global\long\def\Yon{\Yo}%

%%%%%%%%%%%%%%%%%%%%%%%%%%%%%%%% Algebra Macros %%%%%%%%%%%%%%%%%%%%%%%%%%%%%%%%%%%%%%%%%%%%%%

\global\long\def\Ab{\mathsf{Ab}}%

\global\long\def\AbGrp{\text{\ensuremath{\mathsf{AbGrp}}}}%

\global\long\def\alg{\mathsf{alg}}%

\global\long\def\ann{\text{ann}}%

\global\long\def\Ass{{\rm Ass}}%

\global\long\def\Ch{\mathsf{Ch}}%

\global\long\def\CR{\mathsf{ComRing}}%

\global\long\def\Gal{\text{Gal}}%

\global\long\def\Grp{\mathsf{Grp}}%

\global\long\def\Inn{\mathsf{Inn}}%

\global\long\def\Frac{\text{Frac}}%

\global\long\def\mod{\text{\ensuremath{\mathsf{mod}}}}%

\global\long\def\norm{\text{Norm}}%

\global\long\def\Norm{\mathsf{Norm}}%

\global\long\def\Orb{\mathsf{Orb}}%

\global\long\def\rad{\text{\text{rad}}}%

\global\long\def\Ring{\mathbb{\mathsf{Ring}}}%

\global\long\def\sgn{\text{sgn}}%

\global\long\def\Stab{\mathsf{Stab}}%

\global\long\def\Syl{\text{Syl}}%

\global\long\def\ZpZ{\nicefrac{\mathbb{\mathbf{Z}}}{p\bZ}}%

%%%%%%%%%%%%%%%%%%%%%%%%%%%%%%%% Homological Algebra Macros %%%%%%%%%%%%%%%%%%%%%%%%%%%%%%%%%%%

\global\long\def\CompR{_{R}\mathsf{Comp}}%

%%%%%%%%%%%%%%%%%%%%%%%%%%%%%%%% General Geometry %%%%%%%%%%%%%%%%%%%%%%%%%%%%%%%%%%%%%%%%%%%%%

\global\long\def\clm{\RR=\left(\sR,\mathbf{U},\cL,\bA\right)}%

\global\long\def\ebar{\overline{\sE}}%

\global\long\def\lrs{\mathsf{LocRngSpc}}%

\global\long\def\LRS{\mathsf{LRS}}%

\global\long\def\sHom{\mathscr{H}om}%

\global\long\def\RRing{\overline{\sR}\mathsf{-Ring}}%

%%%%%%%%%%%%%%%%%%%%%%%%%%%%%%%% Algebro-Geometric Macros %%%%%%%%%%%%%%%%%%%%%%%%%%%%%%%%%%%%

\global\long\def\Ae{\nicefrac{A\left[\varepsilon\right]}{\left(\e^{2}\right)}}%

\global\long\def\Aff{\mathsf{Aff}}%

\global\long\def\AX{\nicefrac{\Aff}{X}}%

\global\long\def\AY{\nicefrac{\Aff}{Y}}%

\global\long\def\Der{\text{Der}}%

\global\long\def\et{\text{\ensuremath{\acute{e}}t}}%

\global\long\def\etale{\acute{\text{e}}\text{tale}}%

\global\long\def\Et{\mathsf{\acute{E}t}}%

\global\long\def\ke{\nf{k\left[\varepsilon\right]}{\varepsilon^{2}}}%

\global\long\def\Pic{\text{Pic}}%

\global\long\def\proj{\text{Proj}}%

\global\long\def\Qcoh#1{\mathsf{QCoh}\left(#1\right)}%

\global\long\def\rad{\text{\text{rad}}}%

\global\long\def\red{\text{red.}}%

\global\long\def\aScheme#1{\left(\spec\left(#1\right),\cO_{\spec\left(#1\right)}\right)}%

\global\long\def\resScheme#1#2{\left(#2,\sd{\cO_{#1}}_{#2}\right)}%

\global\long\def\Sch{\mathbb{\mathsf{Sch}}}%

\global\long\def\scheme#1{\left(#1,\cO_{#1}\right)}%

\global\long\def\SCR{\S^{\CR}}%

\global\long\def\Schs{\mathsf{\nicefrac{\Sch}{S}}}%

\global\long\def\ShAb#1{\mathsf{Sh}_{\AbGrp}\left(#1\right)}%

\global\long\def\Spec{\text{Spec}}%

\global\long\def\Sym{\text{Sym}}%

\global\long\def\Zar{\mathsf{Zar}}%

%%%%%%%%%%%%%%%%%%%%%%%%%%%%%%%% Analysis Macros %%%%%%%%%%%%%%%%%%%%%%%%%%%%%%%%%%%%%%%%%%%%%%%

\global\long\def\co#1#2{\left[#1,#2\right)}%

\global\long\def\oc#1#2{\left(#1,#2\right]}%

\global\long\def\loc{\mathsf{\text{loc.}}}%

\global\long\def\nn#1{\left\Vert #1\right\Vert }%

\global\long\def\Re{\text{Re}}%

\global\long\def\supp{\text{\ensuremath{\mathsf{supp}}}}%

\global\long\def\XSM{\left(X,\cS,\mu\right)}%

%%%%%%%%%%%%%%%%%%%%%%%%%%%%%%%% Differential Macros %%%%%%%%%%%%%%%%%%%%%%%%%%%%%%%%%%%%%%%%%%%

\global\long\def\grad{\text{grad}}%

\global\long\def\Homrv{\text{\text{Hom}}_{\mathbb{R}-\mathsf{Vect}}}%

%%%%%%%%%%%%%%%%%%%%%%%%%%%%%%%% Differential Geometric Macros %%%%%%%%%%%%%%%%%%%%%%%%%%%%%%%%%

%%%%%%%%%%%%%%%%%%%%%%%%%%%%%%%% Topology Macros %%%%%%%%%%%%%%%%%%%%%%%%%%%%%%%%%%%%%%%%%%%%%%%

\global\long\def\Cech{{\rm \check{C}ech}}%
\global\long\def\cCC{\check{\mathcal{C}}}%
\global\long\def\CC{\check{C}}%

\global\long\def\CH{\mathsf{CHaus}}%

\global\long\def\Cov{\mathsf{Cov}}%

\global\long\def\CW{\mathsf{CW}}%

\global\long\def\HT{\mathsf{HTop}}%

\global\long\def\Homt{\text{\text{Hom}}_{\mathsf{Top}}}%

\global\long\def\Homrv{\text{\text{Hom}}_{\mathbb{R}-\mathsf{Vect}}}%

\global\long\def\MT{\text{\ensuremath{\mathsf{Mor}}}_{\T}}%

\global\long\def\Open{\text{\ensuremath{\mathsf{Open}}}}%

\global\long\def\PT{\mathsf{P-Top}}%

\global\long\def\T{\mathsf{Top}}%

% Homotopy Theory specific macros

\global\long\def\Ad{\mathsf{Ad}}%

\global\long\def\Cell{\mathsf{Cell}}%

\global\long\def\Cov{\mathsf{Cov}}%

\global\long\def\Sp{\mathsf{Sp}}%

\global\long\def\Spectra{\mathsf{Spectra}}%

\global\long\def\ss{\widehat{\triangle}}%

\global\long\def\Tn{\mathbb{T}^{n}}%

\global\long\def\Sk#1{\textrm{Sk}^{#1}}%

\global\long\def\smash{\wedge}%

\global\long\def\wp{\vee}%

% HoTT specific macros

\global\long\def\base{\mathsf{base}}%

\global\long\def\comp{\mathsf{comp}}%

\global\long\def\funext{\mathsf{funext}}%

\global\long\def\hfib{\text{\ensuremath{\mathsf{hfib}}}}%

\global\long\def\I{\mathbf{I}}%

\global\long\def\ind{\mathsf{ind}}%

% LOOP as \mathsf WHAT THE FUCK!!!!

\global\long\def\lp{\mathsf{loop}}%

\global\long\def\pair{\mathsf{pair}}%

\global\long\def\pr{\mathbf{\mathsf{pr}}}%

\global\long\def\rec{\mathsf{rec}}%

\global\long\def\refl{\mathsf{refl}}%

\global\long\def\transport{\mathsf{transport}}%

%%%%%%%%%%%%%%%%%%%%%%%%%%%%%%%% UNSORTED! %%%%%%%%%%%%%%%%%%%%%%%%%%%%%%%%%%%%%%%%%%%%%%%%%%%%%

\global\long\def\is{\triangle\raisebox{2mm}{\mbox{\ensuremath{\infty}}}}%

\global\long\def\Cof{\mathsf{Cof}}%

\global\long\def\sfW{\mathsf{W}}%

\global\long\def\Cyl{\mathsf{Cyl}}%

\global\long\def\Mono{\mathsf{Mono}}%

\global\long\def\t{\triangle}%

\global\long\def\tl{\triangleleft}%

\global\long\def\tr{\triangleright}%

\global\long\def\Shift{\mathrm{Shift}_{+1}}%

\global\long\def\Shiftd{\mathrm{Shift}_{-1}}%

\global\long\def\out{\mathrm{out}}%

\global\long\def\cN{\mathcal{N}}%

\global\long\def\fC{\mathcal{\mathfrak{C}}}%

\global\long\def\ev{\mathsf{ev}}%

\global\long\def\Map{\mathsf{Map}}%

\global\long\def\whp#1{\wh{#1}_{\bullet}}%

\global\long\def\bfTwo{\mathbf{2}}%

\global\long\def\bfL{\mathbf{L}}%

\global\long\def\bfR{\mathbf{R}}%

\global\long\def\sJ{\mathscr{J}}%

\global\long\def\Sing{\mathsf{Sing}}%

\global\long\def\Sph{\mathsf{Sph}}%

\global\long\def\whfin#1{\widehat{#1}_{\mathrm{fin}}}%

\global\long\def\whpfin#1{\widehat{#1}_{\mathrm{\bullet fin}}}%

\global\long\def\fin{\mathsf{fin}}%

\global\long\def\cT{\mathcal{T}}%

\global\long\def\Alg{\mathsf{Alg}}%

\global\long\def\st{\mathsf{st}}%

\global\long\def\IN{\mathsf{in}}%

\global\long\def\hr{\mathsf{hr}}%

\global\long\def\Fun{\mathsf{Fun}}%

\global\long\def\Th{\mathsf{Th}}%

\global\long\def\sT{\mathscr{T}}%

\global\long\def\Lex{\mathsf{Lex}}%

\global\long\def\FinSet{\mathsf{FinSet}}%

\global\long\def\Fib{\mathsf{Fib}}%

\global\long\def\FPS{\mathsf{FinPos}}%

\global\long\def\Mod{\mathsf{Mod}}%

\global\long\def\sfA{\mathsf{A}}%
\global\long\def\sfV{\mathsf{V}}%
\global\long\def\inner{\mathsf{inner}}%
\global\long\def\bfI{\mathbf{I}}%
\global\long\def\Kan{\mathsf{Kan}}%
\global\long\def\Berger{\mathsf{Berger}}%

\global\long\def\LFP{\mathsf{LFP}}%

\global\long\def\Mnd{\mathsf{Mnd}}%

\global\long\def\sfS{\mathsf{S}}%

\global\long\def\Adj{\mathsf{Adj}}%

\global\long\def\RAdj{\mathsf{RAdj}}%

\global\long\def\LAdj{\mathsf{LAdj}}%

\global\long\def\conj{\mathsf{conj}}%

\global\long\def\exact{\mathsf{exact}}%

\global\long\def\CwA{\mathsf{CwA}}%

\begin{comment}
DE-FRAKTURification
\end{comment}

\global\long\def\sf#1{\mathsf{#1}}%

\global\long\def\fA{\sf A}%

\global\long\def\fB{\sf B}%

\global\long\def\fE{\sf E}%

\global\long\def\fF{\sf F}%

\global\long\def\fG{\sf G}%

\global\long\def\fD{\sf D}%

\global\long\def\fJ{\sf I}%

\global\long\def\fJ{\sf J}%
\global\long\def\sfJ{\mathsf{J}}%

\global\long\def\fX{\mathfrak{\sf X}}%

\global\long\def\fY{\sf Y}%

\global\long\def\fZ{\sf Z}%

\global\long\def\bf#1{\mathsf{#1}}%

\global\long\def\cc{\mathsf{cc}}%
\global\long\def\Arity{\sf{Arity}}%
 
\global\long\def\conj{\sf{conj}}%
\global\long\def\Ins{\sf{Ins}}%
\global\long\def\bfOne{\mathbf{1}}%
\global\long\def\oplaxlim{\underleftarrow{\sf{oplaxlim}}}%
\global\long\def\ff{\mathsf{ff}}%
\global\long\def\cc{\mathsf{cc}}%
\global\long\def\Prof{\mathsf{Prof}}%
\global\long\def\sfA{\mathsf{A}}%
\global\long\def\sfB{\mathsf{B}}%
\global\long\def\sfC{\mathsf{C}}%
\global\long\def\sfD{\mathsf{D}}%
\global\long\def\sfE{\mathsf{E}}%
\global\long\def\sfF{\mathsf{F}}%
\global\long\def\sfG{\mathsf{G}}%
\global\long\def\ls{\mathsf{ls}}%
\global\long\def\CAT{\mathsf{CAT}}%
\global\long\def\Equ{\mathsf{Equ}}%
\global\long\def\Par{\mathsf{Par}}%
\global\long\def\Ran{\mathsf{Ran}}%
\global\long\def\Lan{\mathsf{Lan}}%
 
\global\long\def\LG{\underset{\mathsf{LG}}{\otimes}}%
\global\long\def\PR{\mathsf{P.R.}}%
\global\long\def\CDA{\sfC_{\sf{DA}}}%
\global\long\def\StrCat{\omega\mhyphen\Cat}%
\global\long\def\omegaCat{\omega\mhyphen\Cat}%
 
\global\long\def\gray#1{\textcolor{gray}{#1}}%
 
\global\long\def\lightgray#1{\textcolor{lightgray}{#1}}%
\end{quotation}
\title{On the Combinatorics of the Gray Cylinder}
\author{Paul Roy Lessard}
\begin{abstract}
This paper develops some combinatorics of the lax Gray cylinder on
the cells of $\Theta$ understood as a full subcategory of the category
of strict $\omega$-categories. More, we construct a span relating
the Cartesian cylinder, the Gray cylinder, and the shift functor.
\end{abstract}

\maketitle
\tableofcontents{}

\section*{Acknowledgements}

The author is indebted to both Dominic Verity and Yuki Maehara for
their consultation and guidance during the development and preparation
of this work. The author was supported by Australian Research Council
Discovery Project grant: DP190102432 during the development of this
work.

\section*{Overview}

\subsubsection*{Berger's categorical wreath product and the cell categories $\Theta_{n}$
and $\Theta$}

As a preliminary we recall how the wreath product of \cite{Berger2}
provides an elegant and explicit description of the categories $\Theta_{n}$
and $\Theta$ in terms of the simplex category. We also recall Berger's
description of strict $n$-categories (resp. strict $\omega$-categories)
as an orthogonal subcategory of the category of $n$-cellular sets
$\wh{\Theta}$ (resp. cellular sets $\wh{\Theta}$).
\begin{rem}
Throughout this work we will be concerned only with \emph{strict }versions
of higher categories. As such, by $n$-category we will mean the strict
notion.
\end{rem}

\subsubsection*{The Gray cylinder}

The Gray tensor product for $2$-categories, first defined in \cite{Gray},
is the monoidal structure $\otimes$ for the bi-closed structure 
\[
\left(\sf{oplax\left(\_,\_\right)},\otimes,\sf{lax}\left(\_,\_\right)\right)
\]
where, for two $2$-categories $X$ and $Y$, $\sf{\sf{oplax}}\left(X,Y\right)$
is the $2$-category:
\begin{itemize}
\item whose objects are functors;
\item whose $1$-morphisms are oplax-natural transformations; and
\item whose $2$-morphisms are modifications
\end{itemize}
and similarly $\sf{lax}\left(X,Y\right)$ is essentially the same
but with lax-natural transformations instead of oplax ones. While
the Gray tensor product is not well known outside of the category
theory community, within it, it is a cornerstone of formal category
theory. The Gray (a.k.a. generalized Gray, Crans-Gray, etc.) tensor
product for $\omega$-categories fills the analogous role for those
entities. In this work we construct a pair of explicit combinatorial
descriptions of the lax Gray cylinder over cells $T$ of $\Theta$,
that is to say the tensor products $\left[1\right]\otimes T$.

\subsubsection*{The lax shuffle decomposition}

As first observed in \cite{Berger1}, the cartesian product of cells
$S$ and $T$ of $\Theta$, taken either in $\wh{\Theta}$ or $\StrCat$,
is covered by all of their ``shuffled bouquets'', which restrict
on the subcategory $\t$ in $\Theta$ to the usual shuffles of simplices.
For the case of the cartesian cylinder $\left[1\right]\times T$ this
yields a description of that object as a wide pushout. Indeed, given
a cell
\[
T=\left[n\right];\left(T_{1},\dots,T_{n}\right)
\]
of $\Theta$, the cellular set $\left[1\right]\times T$ enjoys the
universal property of the colimit below left. The lax shuffle decomposition
makes the gluing along copies of $\left[n\right];\left(T_{1},\dots,T_{n}\right)$
\emph{lax} by fattening the faces $\left[n\right];\left(T_{i}\right)$
along which we glue, replacing them with spans
\[
\left[n\right];\left(T_{i}\right)\longrightarrow\left[n\right];\left(T_{<k},\left[1\right]\otimes T_{k},T_{>k}\right)\longleftarrow\left[n\right];\left(T_{i}\right)
\]
 as in the colimit below right.

\begin{minipage}[c][1\totalheight][t]{0.45\textwidth}%
\begin{center}
\[
\clim\left\{
\vcenter{\vbox{
\xyR{1.5pc}\xyC{6pc}\xymatrix@!0{ & {\scriptstyle \left[n+1\right];\left(\left[0\right],T_{1},\dots,T_{n}\right)}\\\\{\scriptstyle \left[n\right];\left(T_{1},\dots,T_{n}\right)}\ar[uur]\ar[ddr]\\\\ & {\scriptstyle \left[n+1\right];\left(T_{1},\left[0\right],T_{2},\dots,T_{n}\right)}\\{\scriptstyle \vdots}\ar[ur]\ar[dr]\\ & {\scriptstyle \left[n+1\right];\left(T_{1},\dots,T_{n-1},\left[0\right],T_{n}\right)}\\\\{\scriptstyle \left[n\right];\left(T_{1},\dots,T_{n}\right)}\ar[uur]\ar[ddr]\\\\ & {\scriptstyle \left[n+1\right];\left(T_{1},\dots,T_{n},\left[0\right]\right)}}
}}
\right\}
\]
\par\end{center}%
\end{minipage}\hfill{}%
\begin{minipage}[c][1\totalheight][t]{0.45\textwidth}%
\[
\clim\left\{
\vcenter{\vbox{
\xyR{1.5pc}\xyC{6pc}\xymatrix@!0{ & {\scriptstyle \left[n+1\right];\left(\left[0\right],T_{1},\dots,T_{n}\right)}\\{\scriptstyle \left[n\right];\left(T_{i}\right)}\ar[ur]\ar[dr]\\ & {\scriptstyle \left[n\right];\left(\left[1\right]\otimes T_{1},T_{2},\dots,T_{n}\right)}\\{\scriptstyle \left[n\right];\left(T_{i}\right)}\ar[ur]\ar[dr]\\ & {\scriptstyle \left[n+1\right];\left(T_{1},\left[0\right],\dots,T_{n}\right)}\\{\scriptstyle \vdots}\ar[ur]\ar[dr]\\ & {\scriptstyle \left[n+1\right];\left(T_{1},,\dots T_{n-1},\left[0\right],T_{n}\right)}\\{\scriptstyle \left[n\right];\left(T_{i}\right)}\ar[ur]\ar[dr]\\ & {\scriptstyle \left[n\right];\left(T_{1},\dots,\left[1\right]\otimes T_{n}\right)}\\{\scriptstyle \left[n\right];\left(T_{i}\right)}\ar[ur]\ar[dr]\\ & {\scriptstyle \left[n+1\right];\left(T_{1},,\dots,T_{n},\left[0\right]\right)}}
}}
\right\}
\]%
\end{minipage}
\begin{rem}
Importantly the colimit computing the Gray cylinder needs to be performed
in $\StrCat$ and not on $\wh{\Theta}$ for reasons we discuss in
the text.
\end{rem}

\begin{comment}
For an illustration of the connection between these formulae, for
the case $\left[1\right]\times\left[2\right]$ vs. $\left[1\right]\otimes\left[2\right]$
, see Figure \ref{fig: Comparing the  cells of the shuffle and lax-shuffle}.
\end{comment}

\begin{figure}
\noindent\fbox{\begin{minipage}[t]{1\columnwidth - 2\fboxsep - 2\fboxrule}%
\bigskip{}
\bigskip{}
\bigskip{}

\centering
\resizebox{0.75\textwidth}{!}{
\begin{tikzpicture}
	\begin{pgfonlayer}{nodelayer}
		\node [style=none] (38) at (2, 5.5) {};
		\node [style=none] (40) at (8, 12.5) {};
		\node [style=none] (26) at (2, -2.5) {};
		\node [style=none] (27) at (8, -0.5) {};
		\node [style=none] (28) at (8, 4.5) {};
		\node [style=none] (29) at (14, 0.5) {};
		\node [style=none] (52) at (2, -2.5) {};
		\node [style=none] (54) at (8, 4.5) {};
		\node [style=none] (55) at (14, 0.5) {};
		\node [style=none] (56) at (2, -4.5) {};
		\node [style=none] (57) at (8, -2.5) {};
		\node [style=none] (58) at (14, -1.5) {};
		\node [style=none] (59) at (2, -4.5) {};
		\node [style=none] (60) at (8, -2.5) {};
		\node [style=none] (61) at (14, -1.5) {};
		\node [style=none] (62) at (2, -6.5) {};
		\node [style=none] (64) at (14, -3.5) {};
		\node [style=none] (65) at (2, -6.5) {};
		\node [style=none] (66) at (8, -4.5) {};
		\node [style=none] (67) at (14, -3.5) {};
		\node [style=none] (68) at (2, -6.5) {};
		\node [style=none] (69) at (8, -4.5) {};
		\node [style=none] (70) at (14, -3.5) {};
		\node [style=none] (71) at (2, -10.5) {};
		\node [style=none] (72) at (8, -8.5) {};
		\node [style=none] (73) at (14, -11.5) {};
		\node [style=none] (74) at (14, -7.5) {};
		\node [style=none] (22) at (2, -10.5) {};
		\node [style=none] (23) at (8, -8.5) {};
		\node [style=none] (25) at (14, -7.5) {};
		\node [style=none] (79) at (2, -0.5) {};
		\node [style=none] (80) at (8, 6.5) {};
		\node [style=none] (81) at (14, 2.5) {};
		\node [style=none] (46) at (2, 1.5) {};
		\node [style=none] (47) at (8, 8.5) {};
		\node [style=none] (48) at (14, 4.5) {};
		\node [style=none] (82) at (2, -0.5) {};
		\node [style=none] (83) at (8, 6.5) {};
		\node [style=none] (84) at (14, 2.5) {};
		\node [style=none] (85) at (2, 1.5) {};
		\node [style=none] (86) at (8, 8.5) {};
		\node [style=none] (87) at (14, 4.5) {};
		\node [style=none] (10) at (2, 1.5) {};
		\node [style=none] (13) at (14, 4.5) {};
		\node [style=none] (49) at (2, 1.5) {};
		\node [style=none] (50) at (8, 8.5) {};
		\node [style=none] (51) at (14, 4.5) {};
		\node [style=none] (91) at (2, 5.5) {};
		\node [style=none] (92) at (2, 9.5) {};
		\node [style=none] (93) at (8, 12.5) {};
		\node [style=none] (94) at (14, 8.5) {};
		\node [style=none] (42) at (2, 5.5) {};
		\node [style=none] (43) at (2, 9.5) {};
		\node [style=none] (45) at (14, 8.5) {};
		\node [style=none] (128) at (-14, 5.5) {};
		\node [style=none] (129) at (-8, 12.5) {};
		\node [style=none] (130) at (-14, -2.5) {};
		\node [style=none] (131) at (-8, -0.5) {};
		\node [style=none] (133) at (-2, 0.5) {};
		\node [style=none] (134) at (-14, -2.5) {};
		\node [style=none] (135) at (-8, 4.5) {};
		\node [style=none] (136) at (-2, 0.5) {};
		\node [style=none] (137) at (-14, -6.5) {};
		\node [style=none] (138) at (-8, -4.5) {};
		\node [style=none] (139) at (-2, -3.5) {};
		\node [style=none] (140) at (-14, -6.5) {};
		\node [style=none] (141) at (-8, -4.5) {};
		\node [style=none] (142) at (-2, -3.5) {};
		\node [style=none] (151) at (-14, -10.5) {};
		\node [style=none] (152) at (-2, -11.5) {};
		\node [style=none] (153) at (-2, -7.5) {};
		\node [style=none] (154) at (-14, -10.5) {};
		\node [style=none] (155) at (-8, -8.5) {};
		\node [style=none] (156) at (-2, -11.5) {};
		\node [style=none] (157) at (-2, -7.5) {};
		\node [style=none] (158) at (-14, -10.5) {};
		\node [style=none] (159) at (-8, -8.5) {};
		\node [style=none] (160) at (-2, -7.5) {};
		\node [style=none] (161) at (-14, 1.5) {};
		\node [style=none] (167) at (-14, 1.5) {};
		\node [style=none] (168) at (-8, 8.5) {};
		\node [style=none] (169) at (-2, 4.5) {};
		\node [style=none] (178) at (-14, 5.5) {};
		\node [style=none] (179) at (-14, 9.5) {};
		\node [style=none] (180) at (-8, 12.5) {};
		\node [style=none] (181) at (-2, 8.5) {};
		\node [style=none] (182) at (-14, 5.5) {};
		\node [style=none] (183) at (-14, 9.5) {};
		\node [style=none] (184) at (-2, 8.5) {};
		\node [style=none] (185) at (-14, 5.5) {};
		\node [style=none] (186) at (-14, 9.5) {};
		\node [style=none] (187) at (-2, 8.5) {};
		\node [style=none] (188) at (-14, 9.5) {};
		\node [style=none] (189) at (-2, 8.5) {};
		\node [style=none] (190) at (-8, 12.5) {};
		\node [style=none] (238) at (-8, 12.5) {};
		\node [style=none] (239) at (-8, -11.5) {};
		\node [style=none] (240) at (8, 12.5) {};
		\node [style=none] (241) at (8, -11.5) {};
		\node [style=none] (244) at (-8, 7.25) {};
		\node [style=none] (245) at (8, 7.25) {};
		\node [style=none] (246) at (-8, -7.5) {};
		\node [style=none] (247) at (8, -7.5) {};
		\node [style=none] (248) at (8, 5) {};
		\node [style=none] (250) at (8, -5) {};
		\node [style=none] (252) at (-8, 1) {};
		\node [style=none] (253) at (8, 1) {};
		\node [style=none] (254) at (8, 0) {};
		\node [style=none] (256) at (2, 3.5) {};
		\node [style=none] (257) at (8, 10.5) {};
		\node [style=none] (258) at (14, 6.5) {};
		\node [style=none] (35) at (2, 9.5) {};
		\node [style=none] (36) at (14, 8.5) {};
		\node [style=none] (37) at (8, 12.5) {};
		\node [style=none] (30) at (2, 5.5) {};
		\node [style=none] (31) at (2, 9.5) {};
		\node [style=none] (33) at (14, 8.5) {};
		\node [style=none] (259) at (2, -8.5) {};
		\node [style=none] (260) at (14, -5.5) {};
		\node [style=none] (261) at (2, -8.5) {};
		\node [style=none] (262) at (14, -5.5) {};
		\node [style=none] (263) at (2, -8.5) {};
		\node [style=none] (264) at (8, -6.5) {};
		\node [style=none] (265) at (14, -5.5) {};
		\node [style=none] (75) at (2, -10.5) {};
		\node [style=none] (77) at (14, -11.5) {};
		\node [style=none] (78) at (14, -7.5) {};
	\end{pgfonlayer}
	\begin{pgfonlayer}{edgelayer}
		\draw [style=hidden edges] (130.center) to (131.center);
		\draw [style=alignment lines] (238.center) to (159.center);
		\draw [style=alignment lines] (135.center) to (54.center);
		\draw [style=hidden edges] (135.center) to (131.center);
		\draw [style=hidden edges, bend left=15] (38.center) to (40.center);
		\draw [style=solid arrow] (42.center) to (43.center);
		\draw [style=solid arrow] (43.center) to (45.center);
		\draw [style=solid arrow, bend left=15] (42.center) to (45.center);
		\draw [style=hidden edges] (72.center) to (73.center);
		\draw [style=solid arrow] (68.center) to (69.center);
		\draw [style=solid arrow, bend right=15] (68.center) to (70.center);
		\draw [style=solid arrow] (69.center) to (70.center);
		\draw [style=solid arrow] (68.center) to (70.center);
		\draw [style=hidden edges, bend right=15, looseness=1.25] (69.center) to (70.center);
		\draw [style=Tops] (25.center)
			 to [bend left=15] (22.center)
			 to (23.center)
			 to [bend right=15, looseness=1.25] cycle;
		\draw [style=solid arrow] (28.center) to (29.center);
		\draw [style=hidden edges] (26.center) to (27.center);
		\draw [style=Dashed Edges No Fill] (27.center) to (28.center);
		\draw [style=hidden edges] (27.center) to (29.center);
		\draw [style=solid arrow] (26.center) to (28.center);
		\draw [style=solid arrow] (26.center) to (29.center);
		\draw [style=hidden edges] (46.center) to (47.center);
		\draw [style=Tops] (55.center)
			 to (54.center)
			 to (52.center)
			 to cycle;
		\draw [style=solid arrow] (56.center) to (57.center);
		\draw [style=solid arrow] (57.center) to (58.center);
		\draw [style=solid arrow] (56.center) to (58.center);
		\draw [style=Tops] (60.center)
			 to (61.center)
			 to (59.center)
			 to cycle;
		\draw [style=Fronts] (64.center)
			 to (62.center)
			 to [bend right=15] cycle;
		\draw [style=Tops] (66.center)
			 to (67.center)
			 to (65.center)
			 to cycle;
		\draw [style=solid arrow] (73.center) to (74.center);
		\draw [style=solid arrow] (71.center) to (73.center);
		\draw [style=solid arrow] (71.center) to (72.center);
		\draw [style=solid arrow, bend right=15] (71.center) to (74.center);
		\draw [style=solid arrow, bend right=15, looseness=1.25] (72.center) to (74.center);
		\draw [style=solid arrow] (80.center)
			 to (81.center)
			 to (79.center)
			 to cycle;
		\draw [style=Tops] (84.center)
			 to (83.center)
			 to (82.center)
			 to cycle;
		\draw [style=solid arrow] (85.center)
			 to [bend left=15] (87.center)
			 to (86.center)
			 to [bend right=15] cycle
			 to (87.center);
		\draw [style=Fronts] (13.center)
			 to (10.center)
			 to [bend left=15] cycle;
		\draw [style=Tops] (50.center)
			 to [bend right=15] (49.center)
			 to [bend left=15] (51.center)
			 to cycle;
		\draw [style=solid arrow] (92.center) to (91.center);
		\draw [style=solid arrow] (94.center) to (92.center);
		\draw [style=solid arrow] (92.center) to (93.center);
		\draw (37.center) to (36.center);
		\draw [style=hidden edges] (128.center) to (129.center);
		\draw [style=solid arrow] (185.center) to (186.center);
		\draw [style=solid arrow] (186.center) to (187.center);
		\draw [style=hidden edges] (155.center) to (156.center);
		\draw [style=Tops] (160.center)
			 to (158.center)
			 to (159.center)
			 to cycle;
		\draw [style=hidden edges] (131.center) to (133.center);
		\draw [style=solid arrow] (130.center) to (133.center);
		\draw [style=Tops] (136.center)
			 to (135.center)
			 to (134.center)
			 to cycle;
		\draw [style=solid arrow] (137.center) to (138.center);
		\draw [style=solid arrow] (138.center) to (139.center);
		\draw [style=solid arrow] (137.center) to (139.center);
		\draw [style=Tops] (141.center)
			 to (142.center)
			 to (140.center)
			 to cycle;
		\draw [style=solid arrow] (156.center) to (157.center);
		\draw [style=solid arrow] (154.center) to (156.center);
		\draw [style=solid arrow] (154.center) to (155.center);
		\draw [style=solid arrow] (154.center) to (157.center);
		\draw [style=solid arrow] (155.center) to (157.center);
		\draw [style=Fronts] (151.center)
			 to (152.center)
			 to (153.center)
			 to cycle;
		\draw [style=Tops] (169.center)
			 to (168.center)
			 to (167.center)
			 to cycle;
		\draw [style=solid arrow] (179.center) to (178.center);
		\draw [style=solid arrow] (178.center) to (181.center);
		\draw [style=solid arrow] (181.center) to (179.center);
		\draw [style=solid arrow] (179.center) to (180.center);
		\draw [style=Fronts] (184.center)
			 to (183.center)
			 to (182.center)
			 to cycle;
		\draw [style=Tops] (190.center)
			 to (188.center)
			 to (189.center)
			 to cycle;
		\draw (190.center) to (189.center);
		\draw [style=Tops] (257.center)
			 to [bend right=15] (256.center)
			 to [bend left=15] (258.center)
			 to cycle;
		\draw [style=Tops] (37.center)
			 to (35.center)
			 to (36.center)
			 to cycle;
		\draw [style=Fronts] (33.center)
			 to (31.center)
			 to (30.center)
			 to [bend left=15] cycle;
		\draw [style=solid arrow, bend left=15] (91.center) to (94.center);
		\draw [style=Tops] (265.center)
			 to [bend left=15] (263.center)
			 to (264.center)
			 to [bend right=15, looseness=1.25] cycle;
		\draw [style=solid arrow, bend right=15] (261.center) to (262.center);
		\draw [style=alignment lines] (156.center) to (73.center);
		\draw [style=alignment lines] (238.center) to (37.center);
		\draw [style=Fronts] (75.center)
			 to (77.center)
			 to (78.center)
			 to [bend left=15] cycle;
		\draw [style=alignment lines] (185.center) to (30.center);
		\draw [style=alignment lines] (134.center) to (52.center);
		\draw [style=alignment lines] (160.center) to (78.center);
		\draw [style=alignment lines] (37.center) to (23.center);
		\draw [style=alignment lines] (33.center) to (77.center);
		\draw [style=alignment lines] (31.center) to (75.center);
		\draw [style=alignment lines] (189.center) to (156.center);
		\draw [style=alignment lines] (188.center) to (158.center);
		\draw [style=alignment lines] (136.center) to (55.center);
		\draw [style=alignment lines] (142.center) to (70.center);
		\draw [style=alignment lines] (141.center) to (69.center);
		\draw [style=alignment lines] (140.center) to (68.center);
		\draw [style=alignment lines] (159.center) to (23.center);
		\draw [style=alignment lines] (158.center) to (75.center);
		\draw [style=alignment lines] (168.center) to (50.center);
	\end{pgfonlayer}
\end{tikzpicture}
}\bigskip{}
\bigskip{}
\caption{\label{fig: Comparing the  cells of the shuffle and lax-shuffle}Comparison
of the cells of the shuffle and lax-shuffle decompositions for $\left[1\right]\times\left[2\right]$
and $\left[1\right]\otimes\left[2\right]$ respectively.}
\bigskip{}
\end{minipage}}
\end{figure}

\subsubsection*{Steiner's $\omega$-categories and the correctness of the formula}

While the lax shuffle decomposition is a well formed combinatorial
construction it remains to shown that the construction describes the
Gray cylinder $\left[1\right]\otimes\left(\_\right)$. To do so we
show that the lax shuffle decomposition preserve globular sums - certain
wide pushouts of globes in $\StrCat$ - whence to prove that formula
correct for the Gray cylinder it suffices to prove that it correctly
computes the Gray cylinder over a single globe. To prove this in turn
we appeal to Steiner's theory.

Steiner, in \cite{Steiner1}\footnote{Although this author learned the material from \cite{AraMaltsiniotis}},
develops a treatment of strict treatment of $\omega$-categories as
directed augmented complexes. A directed augmented complex is a chain
complex of abelian groups $A_{\bullet}$, in the homological (positive
degree) convention, together with further data. These further data
rigidify the intuitive translation of chain complex $A_{\bullet}$
into a $\omega$-graph into providing an $\omega$-category. Specifically
the extra data are used to encode the orientation of cells as ``positivity''
- this is the directed in directed augmented complexes - and to pick
out which elements of the $0^{th}$ group are ``real'' objects as
opposed to merely formal linear combinations - this is the datum of
the augmentation. The strength of Steiner's theory on which we lean
is that the lax Gray tensor product of Steiner's $\omega$-categories
is easily written in terms of the tensor product of the underlying
chain complexes.

\subsubsection*{Shifted product rule}

We also develop a formula for $\left[1\right]\otimes T$ as an $\omega$-category
enriched category.%
\begin{comment}
\footnote{which in $\Z$-CATS 2 will used to more easily compute certain model
category theoretic properties of $\left[1\right]\boxtimes\left(\_\right)$}
\end{comment}
{} We define, for any $n\in\N$ and objects $S_{1},\dots,S_{n}\in\Ob\left(\Theta\right)$,
the $\omega$-category $\mathsf{P.R.}\left(S_{1},\dots,S_{n}\right)$
as the colimit 
\[
\colim\left\{ \vcenter{\vbox{\xyR{1.5pc}\xyC{3pc}\xymatrix{\left(\left[1\right]\otimes S_{1}\right)\times S_{2}\times\cdots\times S_{n}\\
S_{1}\times S_{2}\times\cdots\times S_{n}\ar[u]^{\left\{ 1\right\} \otimes S_{1}\times\id\times\cdots\times\id}\ar[d]\\
\vdots\\
S_{1}\times\cdots\times S_{n-1}\times S_{n}\ar[d]_{\id\times\cdots\times\id\times\left\{ 0\right\} \otimes S_{n}}\ar[u]\\
S_{1}\times\cdots\times S_{n-1}\times\left(\left[1\right]\otimes S_{n}\right)
}
}}\right\} 
\]
taken in $\omega$-categories. We observe that
\[
\mathsf{P.R.}\left(\left(S_{i}\right)_{i\in\left\langle n\right\rangle }\right)=\left(\left(\left[1\right]\otimes S_{1}\right)\times\cdots\times S_{n}\right)\underset{\prod S_{i}}{\bigoplus}\cdots\underset{\prod S_{i}}{\bigoplus}\left(S_{1}\times\cdots\times\left(\left[1\right]\otimes S_{n}\right)\right)
\]
justifying the name of product rule - and use these formula to describe
the $\Hom$-categories of $\left[1\right]\otimes\left[n\right];\left(S_{1},S_{2},\dots,S_{n}\right)$
- hence the (dimensionally) shifted part of the name.

\subsubsection*{The Cartesian$\leftarrow$Gray$\rightarrow$Shift span}

Lastly we describe how the Gray cylinder serves as the apex of a span
\[
\left[1\right]\times\left(\_\right)\Longleftarrow\left[1\right]\otimes\left(\_\right)\Longrightarrow\left[1\right];\left(\_\right)
\]
We construct this span explicitly, by way of the shuffle decomposition
and shifted product rule. We illustrate this span at the $2$-simplex
$\left[2\right]$ in Figure \ref{fig:CartesianGrayShift-span-example-diagram}.

\begin{figure}[H]
\centering{}%
\noindent\fbox{\begin{minipage}[t]{1\columnwidth - 2\fboxsep - 2\fboxrule}%
\begin{center}
\bigskip{}
\adjustbox{scale=.25,center}{
\begin{tikzpicture}
	\begin{pgfonlayer}{nodelayer}
		\node [style=none] (62) at (2, -16.5) {};
		\node [style=none] (64) at (14, -13.5) {};
		\node [style=none] (151) at (-14, -32.5) {};
		\node [style=none] (152) at (-2, -33.5) {};
		\node [style=none] (153) at (-2, -29.5) {};
		\node [style=none] (75) at (2, -32.5) {};
		\node [style=none] (77) at (14, -33.5) {};
		\node [style=none] (78) at (14, -29.5) {};
		\node [style=none] (38) at (2, 29.5) {};
		\node [style=none] (40) at (8, 36.5) {};
		\node [style=none] (26) at (2, -2.5) {};
		\node [style=none] (27) at (8, -0.5) {};
		\node [style=none] (28) at (8, 4.5) {};
		\node [style=none] (29) at (14, 0.5) {};
		\node [style=none] (52) at (2, -2.5) {};
		\node [style=none] (54) at (8, 4.5) {};
		\node [style=none] (55) at (14, 0.5) {};
		\node [style=none] (56) at (2, -10.5) {};
		\node [style=none] (57) at (8, -8.5) {};
		\node [style=none] (58) at (14, -7.5) {};
		\node [style=none] (59) at (2, -10.5) {};
		\node [style=none] (60) at (8, -8.5) {};
		\node [style=none] (61) at (14, -7.5) {};
		\node [style=none] (65) at (2, -16.5) {};
		\node [style=none] (66) at (8, -14.5) {};
		\node [style=none] (67) at (14, -13.5) {};
		\node [style=none] (68) at (2, -16.5) {};
		\node [style=none] (69) at (8, -14.5) {};
		\node [style=none] (70) at (14, -13.5) {};
		\node [style=none] (71) at (2, -32.5) {};
		\node [style=none] (72) at (8, -30.5) {};
		\node [style=none] (73) at (14, -33.5) {};
		\node [style=none] (74) at (14, -29.5) {};
		\node [style=none] (22) at (2, -32.5) {};
		\node [style=none] (23) at (8, -30.5) {};
		\node [style=none] (25) at (14, -29.5) {};
		\node [style=none] (79) at (2, 5.5) {};
		\node [style=none] (80) at (8, 12.5) {};
		\node [style=none] (81) at (14, 8.5) {};
		\node [style=none] (46) at (2, 13.5) {};
		\node [style=none] (47) at (8, 20.5) {};
		\node [style=none] (48) at (14, 16.5) {};
		\node [style=none] (82) at (2, 5.5) {};
		\node [style=none] (83) at (8, 12.5) {};
		\node [style=none] (84) at (14, 8.5) {};
		\node [style=none] (85) at (2, 13.5) {};
		\node [style=none] (86) at (8, 20.5) {};
		\node [style=none] (87) at (14, 16.5) {};
		\node [style=none] (10) at (2, 13.5) {};
		\node [style=none] (13) at (14, 16.5) {};
		\node [style=none] (49) at (2, 13.5) {};
		\node [style=none] (50) at (8, 20.5) {};
		\node [style=none] (51) at (14, 16.5) {};
		\node [style=none] (91) at (2, 29.5) {};
		\node [style=none] (92) at (2, 33.5) {};
		\node [style=none] (93) at (8, 36.5) {};
		\node [style=none] (94) at (14, 32.5) {};
		\node [style=none] (42) at (2, 29.5) {};
		\node [style=none] (43) at (2, 33.5) {};
		\node [style=none] (45) at (14, 32.5) {};
		\node [style=none] (128) at (-14, 29.5) {};
		\node [style=none] (129) at (-8, 36.5) {};
		\node [style=none] (130) at (-14, -2.5) {};
		\node [style=none] (131) at (-8, -0.5) {};
		\node [style=none] (133) at (-2, 0.5) {};
		\node [style=none] (134) at (-14, -2.5) {};
		\node [style=none] (135) at (-8, 4.5) {};
		\node [style=none] (136) at (-2, 0.5) {};
		\node [style=none] (137) at (-14, -16.5) {};
		\node [style=none] (138) at (-8, -14.5) {};
		\node [style=none] (139) at (-2, -13.5) {};
		\node [style=none] (140) at (-14, -16.5) {};
		\node [style=none] (141) at (-8, -14.5) {};
		\node [style=none] (142) at (-2, -13.5) {};
		\node [style=none] (154) at (-14, -32.5) {};
		\node [style=none] (155) at (-8, -30.5) {};
		\node [style=none] (156) at (-2, -33.5) {};
		\node [style=none] (157) at (-2, -29.5) {};
		\node [style=none] (158) at (-14, -32.5) {};
		\node [style=none] (159) at (-8, -30.5) {};
		\node [style=none] (160) at (-2, -29.5) {};
		\node [style=none] (161) at (-14, 13.5) {};
		\node [style=none] (167) at (-14, 13.5) {};
		\node [style=none] (168) at (-8, 20.5) {};
		\node [style=none] (169) at (-2, 16.5) {};
		\node [style=none] (178) at (-14, 29.5) {};
		\node [style=none] (179) at (-14, 33.5) {};
		\node [style=none] (180) at (-8, 36.5) {};
		\node [style=none] (181) at (-2, 32.5) {};
		\node [style=none] (182) at (-14, 29.5) {};
		\node [style=none] (183) at (-14, 33.5) {};
		\node [style=none] (184) at (-2, 32.5) {};
		\node [style=none] (185) at (-14, 29.5) {};
		\node [style=none] (186) at (-14, 33.5) {};
		\node [style=none] (187) at (-2, 32.5) {};
		\node [style=none] (188) at (-14, 33.5) {};
		\node [style=none] (189) at (-2, 32.5) {};
		\node [style=none] (190) at (-8, 36.5) {};
		\node [style=none] (238) at (-8, 36.5) {};
		\node [style=none] (239) at (-8, -33.5) {};
		\node [style=none] (240) at (8, 36.5) {};
		\node [style=none] (241) at (8, -33.5) {};
		\node [style=none] (244) at (-8, 15.25) {};
		\node [style=none] (245) at (8, 13) {};
		\node [style=none] (246) at (-8, -29.75) {};
		\node [style=none] (247) at (8, -23.5) {};
		\node [style=none] (248) at (8, 11) {};
		\node [style=none] (250) at (8, -15) {};
		\node [style=none] (252) at (-8, 1) {};
		\node [style=none] (253) at (8, 1) {};
		\node [style=none] (254) at (8, 0) {};
		\node [style=none] (256) at (2, 21.5) {};
		\node [style=none] (257) at (8, 28.5) {};
		\node [style=none] (258) at (14, 24.5) {};
		\node [style=none] (35) at (2, 33.5) {};
		\node [style=none] (36) at (14, 32.5) {};
		\node [style=none] (37) at (8, 36.5) {};
		\node [style=none] (30) at (2, 29.5) {};
		\node [style=none] (31) at (2, 33.5) {};
		\node [style=none] (33) at (14, 32.5) {};
		\node [style=none] (259) at (2, -24.5) {};
		\node [style=none] (260) at (14, -21.5) {};
		\node [style=none] (261) at (2, -24.5) {};
		\node [style=none] (262) at (14, -21.5) {};
		\node [style=none] (263) at (2, -24.5) {};
		\node [style=none] (264) at (8, -22.5) {};
		\node [style=none] (265) at (14, -21.5) {};
		\node [style=none] (266) at (18, -2.5) {};
		\node [style=none] (267) at (30, 0.5) {};
		\node [style=none] (268) at (18, -2.5) {};
		\node [style=none] (269) at (30, 0.5) {};
		\node [style=none] (270) at (18, -16.5) {};
		\node [style=none] (271) at (30, -13.5) {};
		\node [style=none] (272) at (18, -16.5) {};
		\node [style=none] (273) at (30, -13.5) {};
		\node [style=none] (274) at (18, -16.5) {};
		\node [style=none] (275) at (30, -13.5) {};
		\node [style=none] (276) at (18, 13.5) {};
		\node [style=none] (277) at (30, 16.5) {};
		\node [style=none] (278) at (18, 13.5) {};
		\node [style=none] (279) at (30, 16.5) {};
		\node [style=none] (280) at (18, 13.5) {};
		\node [style=none] (281) at (30, 16.5) {};
		\node [style=none] (282) at (24, -23.5) {};
		\node [style=none] (283) at (24, 17) {};
		\node [style=none] (284) at (24, -15) {};
		\node [style=none] (285) at (24, 0) {};
		\node [style=none] (286) at (1.5, 33) {};
		\node [style=none] (287) at (-1.75, 33) {};
		\node [style=none] (288) at (-1.75, 17) {};
		\node [style=none] (289) at (3.5, 17) {};
		\node [style=none] (290) at (-1.75, 1) {};
		\node [style=none] (291) at (4.5, 1) {};
		\node [style=none] (292) at (-6.5, -15) {};
		\node [style=none] (293) at (5.75, -15) {};
		\node [style=none] (294) at (-1.5, -31.5) {};
		\node [style=none] (295) at (4, -31.5) {};
		\node [style=none] (296) at (12, -15) {};
		\node [style=none] (297) at (22.5, -15) {};
		\node [style=none] (298) at (9.75, -1) {};
		\node [style=none] (299) at (22, -1) {};
		\node [style=none] (300) at (10, 15) {};
		\node [style=none] (301) at (20, 15) {};
		\node [style=none] (302) at (-8, 21.25) {};
		\node [style=none] (303) at (-8, 30.5) {};
		\node [style=none] (304) at (-8, 14.5) {};
		\node [style=none] (305) at (-8, 5.25) {};
		\node [style=none] (306) at (-8, -1.5) {};
		\node [style=none] (307) at (-8, -13.75) {};
		\node [style=none] (308) at (8, -9.5) {};
		\node [style=none] (309) at (8, -13.75) {};
		\node [style=none] (310) at (8, -7.75) {};
		\node [style=none] (311) at (8, -1.5) {};
		\node [style=none] (312) at (-8, -15.75) {};
		\node [style=none] (313) at (8, -16.5) {};
		\node [style=none] (314) at (8, -21.75) {};
		\node [style=none] (315) at (8, -24.5) {};
		\node [style=none] (316) at (8, -29.75) {};
		\node [style=none] (317) at (8, 31.5) {};
		\node [style=none] (318) at (8, 29) {};
		\node [style=none] (319) at (8, 23.5) {};
		\node [style=none] (320) at (8, 21) {};
		\node [style=none] (321) at (8, 14.5) {};
		\node [style=none] (322) at (8, 6.5) {};
		\node [style=none] (323) at (8, 5.25) {};
		\node [style=none] (324) at (24, 14.5) {};
		\node [style=none] (325) at (24, -1.75) {};
		\node [style=none] (326) at (24, -14.25) {};
	\end{pgfonlayer}
	\begin{pgfonlayer}{edgelayer}
		\draw [style=hidden edges, bend right=15, looseness=1.25] (69.center) to (70.center);
		\draw [style=Fronts] (64.center)
			 to (62.center)
			 to [bend right=15] cycle;
		\draw [style=hidden edges] (155.center) to (156.center);
		\draw [style=Fronts] (151.center)
			 to (152.center)
			 to (153.center)
			 to cycle;
		\draw [style=hidden edges] (72.center) to (73.center);
		\draw [style=Fronts] (75.center)
			 to (77.center)
			 to (78.center)
			 to [bend left=15] cycle;
		\draw [style=ghost lines] (135.center) to (54.center);
		\draw [style=hidden edges] (135.center) to (131.center);
		\draw [style=hidden edges, bend left=15] (38.center) to (40.center);
		\draw [style=solid arrow] (42.center) to (43.center);
		\draw [style=solid arrow] (43.center) to (45.center);
		\draw [style=solid arrow, bend left=15] (42.center) to (45.center);
		\draw [style=solid arrow] (68.center) to (69.center);
		\draw [style=solid arrow, bend right=15] (68.center) to (70.center);
		\draw [style=solid arrow] (69.center) to (70.center);
		\draw [style=solid arrow] (68.center) to (70.center);
		\draw [style=Tops] (25.center)
			 to [bend left=15] (22.center)
			 to (23.center)
			 to [bend right=15, looseness=1.25] cycle;
		\draw [style=solid arrow] (28.center) to (29.center);
		\draw [style=hidden edges] (26.center) to (27.center);
		\draw [style=hidden edges] (27.center) to (28.center);
		\draw [style=hidden edges] (27.center) to (29.center);
		\draw [style=solid arrow] (26.center) to (28.center);
		\draw [style=solid arrow] (26.center) to (29.center);
		\draw [style=hidden edges] (46.center) to (47.center);
		\draw [style=Tops] (55.center)
			 to (54.center)
			 to (52.center)
			 to cycle;
		\draw [style=solid arrow] (56.center) to (57.center);
		\draw [style=solid arrow] (57.center) to (58.center);
		\draw [style=solid arrow] (56.center) to (58.center);
		\draw [style=Tops] (60.center)
			 to (61.center)
			 to (59.center)
			 to cycle;
		\draw [style=Tops] (66.center)
			 to (67.center)
			 to (65.center)
			 to cycle;
		\draw [style=solid arrow] (73.center) to (74.center);
		\draw [style=solid arrow] (71.center) to (73.center);
		\draw [style=solid arrow] (71.center) to (72.center);
		\draw [style=solid arrow, bend right=15] (71.center) to (74.center);
		\draw [style=solid arrow, bend right=15, looseness=1.25] (72.center) to (74.center);
		\draw [style=solid arrow] (80.center)
			 to (81.center)
			 to (79.center)
			 to cycle;
		\draw [style=Tops] (84.center)
			 to (83.center)
			 to (82.center)
			 to cycle;
		\draw [style=solid arrow] (85.center)
			 to [bend left=15] (87.center)
			 to (86.center)
			 to [bend right=15] cycle
			 to (87.center);
		\draw [style=Fronts] (13.center)
			 to (10.center)
			 to [bend left=15] cycle;
		\draw [style=Tops] (50.center)
			 to [bend right=15] (49.center)
			 to [bend left=15] (51.center)
			 to cycle;
		\draw [style=solid arrow] (92.center) to (91.center);
		\draw [style=solid arrow] (94.center) to (92.center);
		\draw [style=solid arrow] (92.center) to (93.center);
		\draw (37.center) to (36.center);
		\draw [style=hidden edges] (128.center) to (129.center);
		\draw [style=solid arrow] (185.center) to (186.center);
		\draw [style=solid arrow] (186.center) to (187.center);
		\draw [style=Tops] (160.center)
			 to (158.center)
			 to (159.center)
			 to cycle;
		\draw [style=hidden edges] (130.center) to (131.center);
		\draw [style=hidden edges] (131.center) to (133.center);
		\draw [style=solid arrow] (130.center) to (133.center);
		\draw [style=Tops] (136.center)
			 to (135.center)
			 to (134.center)
			 to cycle;
		\draw [style=solid arrow] (137.center) to (138.center);
		\draw [style=solid arrow] (138.center) to (139.center);
		\draw [style=solid arrow] (137.center) to (139.center);
		\draw [style=Tops] (141.center)
			 to (142.center)
			 to (140.center)
			 to cycle;
		\draw [style=solid arrow] (156.center) to (157.center);
		\draw [style=solid arrow] (154.center) to (156.center);
		\draw [style=solid arrow] (154.center) to (155.center);
		\draw [style=solid arrow] (154.center) to (157.center);
		\draw [style=solid arrow] (155.center) to (157.center);
		\draw [style=Tops] (169.center)
			 to (168.center)
			 to (167.center)
			 to cycle;
		\draw [style=solid arrow] (179.center) to (178.center);
		\draw [style=solid arrow] (178.center) to (181.center);
		\draw [style=solid arrow] (181.center) to (179.center);
		\draw [style=solid arrow] (179.center) to (180.center);
		\draw [style=Fronts] (184.center)
			 to (183.center)
			 to (182.center)
			 to cycle;
		\draw [style=Tops] (190.center)
			 to (188.center)
			 to (189.center)
			 to cycle;
		\draw (190.center) to (189.center);
		\draw [style=Tops] (257.center)
			 to [bend right=15] (256.center)
			 to [bend left=15] (258.center)
			 to cycle;
		\draw [style=Tops] (37.center)
			 to (35.center)
			 to (36.center)
			 to cycle;
		\draw [style=Fronts] (33.center)
			 to (31.center)
			 to (30.center)
			 to [bend left=15] cycle;
		\draw [style=solid arrow, bend left=15] (91.center) to (94.center);
		\draw [style=Tops] (265.center)
			 to [bend left=15] (263.center)
			 to (264.center)
			 to [bend right=15, looseness=1.25] cycle;
		\draw [style=solid arrow, bend right=15] (261.center) to (262.center);
		\draw [style=ghost lines] (156.center) to (73.center);
		\draw [style=ghost lines] (238.center) to (37.center);
		\draw [style=ghost lines] (185.center) to (30.center);
		\draw [style=ghost lines] (134.center) to (52.center);
		\draw [style=ghost lines] (160.center) to (78.center);
		\draw [style=ghost lines] (33.center) to (77.center);
		\draw [style=ghost lines] (31.center) to (75.center);
		\draw [style=ghost lines] (136.center) to (55.center);
		\draw [style=ghost lines] (142.center) to (70.center);
		\draw [style=ghost lines] (141.center) to (69.center);
		\draw [style=ghost lines] (140.center) to (68.center);
		\draw [style=ghost lines] (158.center) to (75.center);
		\draw [style=ghost lines] (168.center) to (50.center);
		\draw [style=solid arrow, bend right=15] (274.center) to (275.center);
		\draw [style=solid arrow] (274.center) to (275.center);
		\draw [style=solid arrow] (266.center) to (267.center);
		\draw [style=Tops] (268.center) to (269.center);
		\draw [style=Fronts] (271.center)
			 to (270.center)
			 to [bend right=15] cycle;
		\draw [style=Tops] (273.center) to (272.center);
		\draw [style=solid arrow, bend left=15] (278.center) to (279.center);
		\draw [style=solid arrow] (278.center) to (279.center);
		\draw [style=Fronts] (281.center)
			 to (280.center)
			 to [bend left=15] cycle;
		\draw [style=ghost lines] (49.center) to (280.center);
		\draw [style=ghost lines] (51.center) to (281.center);
		\draw [style=ghost lines] (52.center) to (268.center);
		\draw [style=ghost lines] (55.center) to (269.center);
		\draw [style=ghost lines] (68.center) to (274.center);
		\draw [style=ghost lines] (70.center) to (275.center);
		\draw [style=arrow] (287.center) to (286.center);
		\draw [style=arrow] (288.center) to (289.center);
		\draw [style=arrow] (290.center) to (291.center);
		\draw [style=arrow] (292.center) to (293.center);
		\draw [style=arrow] (294.center) to (295.center);
		\draw [style=arrow] (297.center) to (296.center);
		\draw [style=arrow] (299.center) to (298.center);
		\draw [style=arrow] (301.center) to (300.center);
		\draw [style=arrow] (303.center) to (302.center);
		\draw [style=arrow] (305.center) to (304.center);
		\draw [style=arrow] (307.center) to (306.center);
		\draw [style=arrow] (309.center) to (308.center);
		\draw [style=arrow] (311.center) to (310.center);
		\draw [style=arrow] (246.center) to (312.center);
		\draw [style=arrow] (313.center) to (314.center);
		\draw [style=arrow] (316.center) to (315.center);
		\draw [style=arrow] (317.center) to (318.center);
		\draw [style=arrow] (320.center) to (319.center);
		\draw [style=arrow] (321.center) to (245.center);
		\draw [style=arrow] (323.center) to (322.center);
		\draw [style=arrow] (324.center) to (285.center);
		\draw [style=arrow] (326.center) to (325.center);
	\end{pgfonlayer}
\end{tikzpicture}
}\caption{\label{fig:CartesianGrayShift-span-example-diagram}Cartesian$\leftarrow$Gray$\rightarrow$Shift
span at $\left[2\right]$}
\bigskip{}
\par\end{center}%
\end{minipage}}
\end{figure}
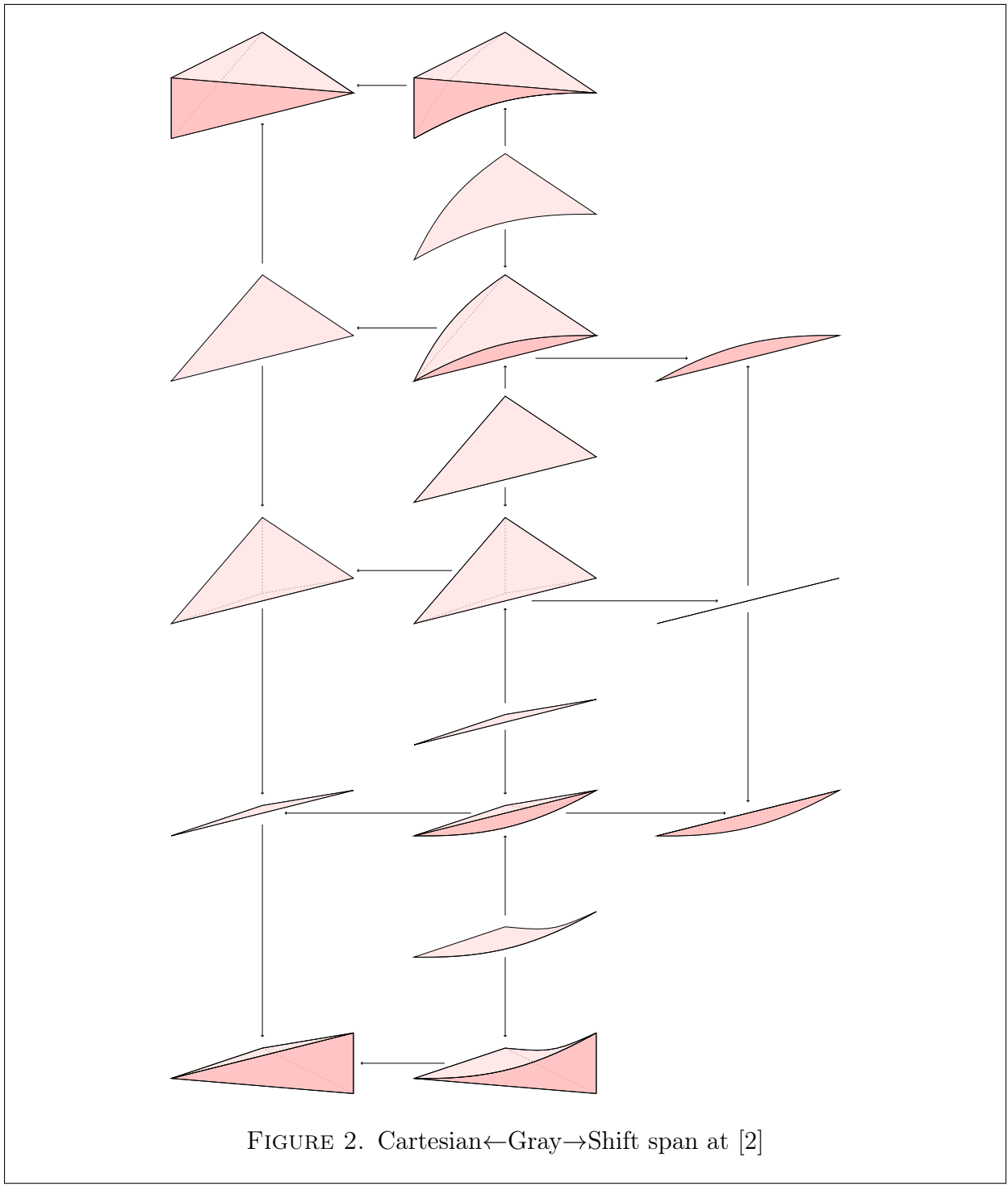

\begin{rem*}
While a pleasing construction in its own right, the Cartesian-Gray-shift
span, in the forthcoming $\Z$-cats II this span is seen to be and
or induce a span of natural weak equivalences for various model structures.

\pagebreak{}
\end{rem*}

\section{Berger's Categorical Wreath Product and the Cell Categories $\Theta$
and $\Theta_{n}$}

What follows was developed first in \cite{Berger2}.\footnote{However our presentation borrows more from \cite{CisinskiMaltsiniotis}.}

\subsection{Segal's category $\Gamma$}

Segal's category $\Gamma$ is a skeleton, in the sense of a small
category of chosen representatives for each isomorphism class, for
the opposite category of the category of finite pointed sets.
\begin{defn}
Let $\Gamma$, \textbf{Segal's gamma category}, be the category specified
thus: let
\[
\Ob\left(\Gamma\right)=\left\{ \sd{\left\langle k\right\rangle =\left\{ 1,\dots,k\right\} }k\geq1\right\} \cup\left\{ \left\langle 0\right\rangle =\varnothing\right\} ,
\]
 and let $\Gamma\left(\left\langle n\right\rangle ,\left\langle m\right\rangle \right)$
be defined by the expression
\[
\Gamma\left(\left\langle n\right\rangle ,\left\langle m\right\rangle \right)=\left\{ \sd{\varphi:\left\langle n\right\rangle \longrightarrow\mathsf{Sub}_{\S}\left(\left\langle m\right\rangle \right)}\forall i\neq j\in\left\langle m\right\rangle ,\ \varphi\left(i\right)\cap\varphi\left(j\right)=\varnothing\right\} 
\]
where, for any category $\sf A$ and object $a$ thereof, $\mathsf{Sub}_{\sf A}\left(a\right)$
is the category of subobjects of $a$. Define the composition of morphisms
in $\Gamma$ by setting
\[
\left\langle \ell\right\rangle \overset{\varphi}{\longrightarrow}\left\langle m\right\rangle \overset{\sigma}{\longrightarrow}\left\langle n\right\rangle 
\]
to be the map
\[
\sigma\circ\varphi:i\longmapsto\bigcup_{j\in\varphi\left(i\right)}\sigma\left(j\right).
\]
\end{defn}

\begin{rem}
The equivalence of categories between $\Gamma$ and $\FinSet_{\bullet}^{\op}$
is a particularly truncated analogue of the Grothendieck construction
- a map of finite pointed sets is replaced with the data of the fibres
it parameterizes (see Figure \ref{fig:Illustrating-the-equivalence- between-FinSet and Gamma}). 

\begin{figure}
\noindent\fbox{\begin{minipage}[t]{1\columnwidth - 2\fboxsep - 2\fboxrule}%
\begin{center}
\includegraphics[scale=0.75]{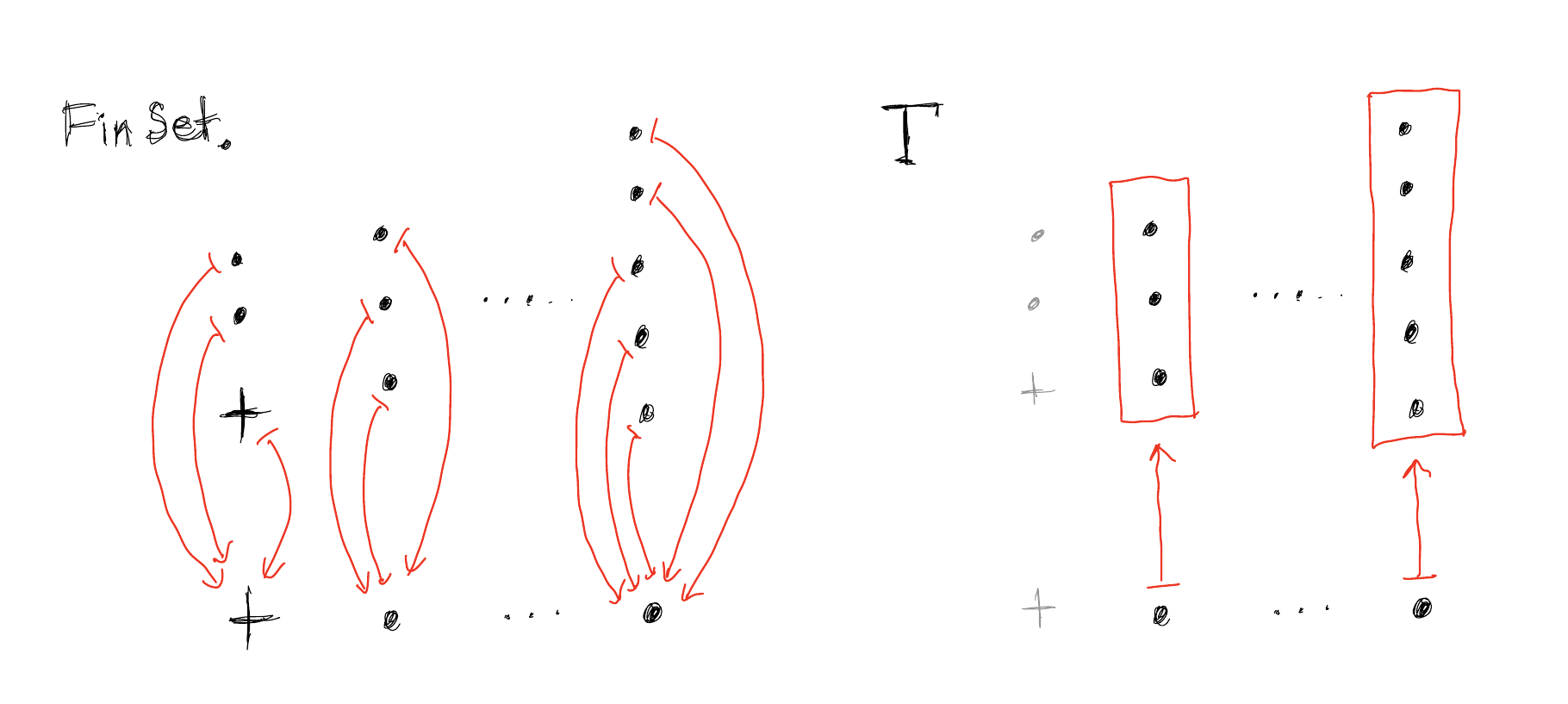}
\par\end{center}
\caption{\label{fig:Illustrating-the-equivalence- between-FinSet and Gamma}Illustrating
the equivalence between $\protect\sf{FinSet}_{\bullet}^{\protect\op}$
and $\Gamma$}

\bigskip{}
\end{minipage}}
\end{figure}
\end{rem}

\subsection{\label{subsec:The-Categorical-wreath}Berger's categorical wreath
product}
\begin{defn}
Let $A$ and $B$ be small categories. Given a functor $G:\sf B\longrightarrow\Gamma$,
we define $\sf B\int_{G}\sf A=\sf B\int\sf A$ (with the second notation
suppressing the functor $G$ when the meaning is clear) be the category
the objects of which are indexed tuples
\[
b;\left(a_{1},\dots,a_{m}\right)
\]
where:
\begin{itemize}
\item $b$ is an object of $\sf B$, $G\left(b\right)=\left\langle m\right\rangle $;
and
\item $\left(a_{1},\dots,a_{m}\right)$ describes a function $G\left(b\right)\longrightarrow\Ob\left(\sf A\right)$.
\end{itemize}
\end{defn}

The morphisms of the wreath product $\sf B\int\sf A$, denoted 
\[
g;\mathbf{f}:b;\left(a_{i}\right)_{i\in G\left(b\right)}\longrightarrow d;\left(c_{i}\right)_{i\in G\left(d\right)}
\]
 are comprised of a morphism
\[
g:b\longrightarrow d
\]
 of $\sf B$ and a morphism of $\wh A$, 
\[
\mathbf{f}=\left(\left(f_{ji}:a_{i}\rightarrow c_{j}\right)_{j\in G\left(g\right)\left(i\right)}\right)_{i\in G\left(b\right)}:\prod_{i\in G\left(b\right)}\left(\sf A\left[a_{i}\right]\longrightarrow\prod_{j\in G\left(g\right)\left(i\right)}\sf A\left[c_{j}\right]\right).
\]
 The composition
\[
b;\left(a_{i}\right)_{i\in G\left(b\right)}\overset{g;\mathbf{f}}{\longrightarrow}d;\left(c_{i}\right)_{i\in G\left(d\right)}\overset{r;\mathbf{q}}{\longrightarrow}\ell;\left(k_{i}\right)_{i\in G\left(\ell\right)}
\]
is denoted $r\circ g;\mathbf{q}\circ\mathbf{f}$ where the meaning
of $r\circ g$ is clear and
\[
\mathbf{q}\circ\mathbf{f}=\left(\left(q_{jk}\circ f_{ki}\right)_{j\in G\left(r\circ g\right)\left(i\right)}\right)_{i\in G\left(b\right)}
\]
 with the values for $k\in G\left(d\right)$ being those unique $k$
in $G\left(g\right)\left(i\right)$ such that $j\in G\left(r\right)\left(k\right)$.
\begin{rem}
We'll also use the more explicit notation $g;\left(f_{1},f_{2},\dots,f_{n}\right)$
where doing so simplifies the exposition.
\end{rem}

\begin{example}
\label{exa:from-simplex-category-to-segal's-gamma}

Define the functor $F:\t\longrightarrow\Gamma$ by setting 
\[
F\left(\left[n\right]\right)=\left\langle n\right\rangle 
\]
 and setting for each $\varphi:\left[m\right]\longrightarrow\left[n\right]$,
\[
F\left(\varphi\right):\left\langle m\right\rangle \longrightarrow\left\langle n\right\rangle 
\]
to be the function 
\[
F\left(\varphi\right):\left\langle m\right\rangle \longrightarrow\mathsf{Sub}_{\S}\left(\left\langle n\right\rangle \right)
\]
 given thus:
\[
F\left(\varphi\right)\left(i\right)=\left\{ \sd j\varphi\left(i-1\right)<j\leq\varphi\left(i\right)\right\} .
\]
\end{example}

\begin{figure}[H]
\noindent\fbox{\begin{minipage}[t]{1\columnwidth - 2\fboxsep - 2\fboxrule}%
\bigskip{}
\[
\xyR{3pc}\xyC{1pc}\xymatrix{0\ar@{|->}[d]\ar@{-}[r] & \lightgray 1\ar@[lightgray]@{|->}[d]\ar@[lightgray]@{|->}[drr]\ar[r] & 1\ar@{|->}[drr]\ar@{-}[r] & \lightgray{2\ar@[lightgray]@{|->}[drr]}\ar[r] & 2\ar@{|->}[drr]\\
0\ar@{-}[r] & \lightgray 1\ar[r] & 1\ar@{-}[r] & \lightgray 2\ar[r] & 2\ar@{-}[r] & \lightgray 3\ar[r] & 3
}
\]

\caption{Overlay of the $\protect\t$-map $d^{1}:\left[2\right]\rightarrow\left[3\right]$
and the $\Gamma$-map $F\left(d^{1}\right):F\left(\left[2\right]\right)\rightarrow F\left(\left[3\right]\right)$}
\bigskip{}
\end{minipage}}
\end{figure}

\begin{rem}
It is not hard to see that the wreath product defines a functor 
\[
\text{\ensuremath{\left(\_\right)}\ensuremath{\ensuremath{\int\left(\_\right)}}:}\left(\Cat\downarrow\Gamma\right)\times\Cat\longrightarrow\Cat.
\]
\end{rem}

\subsection{The Categories $\Theta$ and $\Theta_{n}$}
\begin{defn}
Let 
\[
\gamma:\t\longrightarrow\t\int\t
\]
 be the obvious functor extending the assignment $\gamma\left(\left[n\right]\right)=\left[n\right];\left(\left[0\right],\dots,\left[0\right]\right)$.
We define the categories $\Theta_{n}$ to be the $n^{th}$ wreath
product of $\t$ with itself, i.e. we set
\[
\Theta_{n}=\underbrace{\t\int\left(\cdots\int\t\right)}_{n-\mathrm{times}}.
\]
 We set $\Theta$ to be the conical colimit\footnote{An elementary argument about the projective-canonical and Reedy-canonical
model structures on $\CAT\left(\N,\Cat\right)$ provides that this
conical colimit enjoys the universal property of the pseudo-colimit.}
\[
\underset{\longrightarrow}{\lim}\left\{ \t\overset{\gamma}{\longrightarrow}\t\int\t\overset{\id\int\gamma}{\longrightarrow}\cdots\right\} .
\]
\end{defn}

\begin{rem}
It should be noted that 
\[
\Theta\iso\t\int\Theta\iso\t\int\t\int\Theta\iso\cdots
\]
 so we may denote cells - where cells are the objects of $\Theta$
- in many compatible ways. For example for any $T$ a cell of $\Theta$
we may also write $T=\left[n\right];\left(T_{1},\dots,T_{n}\right)$
for some unique $n\in\N$ and unique $T_{1},\dots,T_{n}$ cells of
$\Theta$.
\end{rem}

\subsection{Globular Sums}

The decompositions of the preceding remark terminate in the observation
that every cell $T$ of $\Theta$ admits a unique description as a
wide pushout of globes. Indeed, for every $T$ of $\Theta$, there
exist unique integers $n_{0}\geq m_{1}\leq n_{1}\geq\cdots\leq n_{\ell-1}\geq m_{\ell-1}\leq n_{\ell}$
for which we have a canonical comparison isomorphism
\[
\colim\left\{ \vcenter{\vbox{\xyR{1.5pc}\xyC{1.5pc}\xymatrix{\overline{n_{0}} &  & \overline{n_{1}} &  & \overline{n_{\ell-1}} &  & \overline{n_{\ell}}\\
 & \overline{m_{1}}\ar[ul]|-{t^{n_{0}-m_{1}}}\ar[ur]|-{s^{n_{1}-m_{1}}} &  & \cdots\ar[ur]\ar[ul] &  & \overline{m_{\ell-1}}\ar[ul]|-{t^{n_{\ell-1}-m_{\ell-1}}}\ar[ur]|-{s^{n_{\ell}-m_{\ell-1}}}
}
}}\right\} \liso T
\]
Codifying this relation between $T$ and integers $n_{0}\geq m_{1}\leq n_{1}\geq\cdots\leq n_{\ell-1}\geq m_{\ell-1}\leq n_{\ell}$
we write $T=\overline{n_{0}}\underset{m_{1}}{\oplus}\overline{n_{1}}\underset{\cdots}{\oplus}\cdots\underset{\cdots}{\oplus}\overline{n_{\ell-1}}\underset{m_{\ell-1}}{\oplus}\overline{n_{\ell}}$.
We will refer to the right-hand expression as a globular sum and the
equality as the globular-sum decomposition.

\pagebreak{}

\section{\label{sec:The-Gray-Cylinder-section}The Gray Cylinder $\left[1\right]\otimes\left(\_\right):\protect\wh{\Theta}\protect\longrightarrow\protect\wh{\Theta}$}

The lax Gray tensor product $\otimes$ for $\omega$-categories is
defined by its enjoyment of the following universal property.\footnote{For specific treatments in other models see \cite{Verity},\cite{CampionKapulkinMaehara},
or \cite{AraMaltsiniotis}}
\[
\StrCat\left(\sfB,\sf{oplax}\left(\sfA,\sfC\right)\right)\liso\omega\mhyphen\Cat\left(\sf A\otimes\sf B,\sf C\right)\liso\omega\mhyphen\Cat\left(\sf A,\text{\ensuremath{\mathsf{Lax}\left(\sf B,\sf C\right)}}\right)
\]
Since for any object $X\in\wh{\Theta}$, the functor
\[
X\otimes\left(\_\right):\Theta\longrightarrow\omega\mhyphen\Cat
\]
 is the restriction of a left adjoint endofunctor on $\StrCat$ to
the subcategory $\Theta$, and cells admit a canonical globular decomposition,
it follows that for any functor $F$, to prove that $F\liso X\otimes\left(\_\right)$
it suffices to prove that:
\begin{itemize}
\item $F\left(\overline{n}\right)\liso X\otimes\overline{n}$ for all $n\in\N$;
and
\item that $F$ preserves globular sums.
\end{itemize}
We'll use this observation, with $X=\left[1\right]$, as follows:
\begin{itemize}
\item in Section \ref{subsec:The-assignment-on-objects-LG} we provide a
recursive formula - a lax version of the shuffle decomposition - assigning
to every cell $T$ of $\Theta$, a wide push-out - which we'll denote
(in a subtle and local (in scope) abuse of notation) by $\left[1\right]\LG T$;
\item in Section \ref{subsec:The-assignment-LG-preserves-globular-sums}
we prove that our formula preserves globular sums;
\item in Section \ref{subsec:Comparison-to-Steiner} we prove - by way of
comparison to Steiner's treatment of $\omega$-categories - that $\left[1\right]\LG\overline{n}$
enjoys the universal property of the lax Gray cylinder on $\overline{n}$
for all $n\in\N$.
\end{itemize}
From these results it then follows that the lax shuffle decomposition
presents the lax Gray cylinder.

\subsection{\label{subsec:The-assignment-on-objects-LG} The Lax shuffle decomposition }

Recall that the cartesian cylinder admits the so-called shuffle decomposition.
The lax Gray tensor product enjoys a similar expression, but the gluing
along copies of $\left[n\right];\left(T_{i}\right)$ is made \emph{lax}.
\begin{rem}
As explained above in what follows we mean, by $\left[1\right]\LG T$,
a particular colimit and only later will we prove that the colimit
$\left[1\right]\LG T$ presents the Gray cylinder.
\end{rem}

\begin{defn}
\label{def:LG-on-objects}We define the function
\[
\left[1\right]\LG\left(\_\right):\Ob\left(\Theta\right)\longrightarrow\Ob\left(\StrCat\right)
\]
and a family of maps:

\[
\left\{ 0\right\} \LG T:T\longrightarrow\left[1\right]\LG T
\]
and 
\[
\left\{ 1\right\} \LG T:T\longrightarrow\left[1\right]\LG T
\]
recursively as follows.

Set
\[
\left[1\right]\LG\left[0\right]=\left[1\right]
\]
and set
\[
\left\{ 0\right\} \LG\left[0\right]=\left\{ 0\right\} :\left[0\right]\longrightarrow\left[1\right]=\left[1\right]\LG\left[0\right]
\]
and 
\[
\left\{ 1\right\} \LG\left[0\right]=\left\{ 0\right\} :\left[0\right]\longrightarrow\left[1\right]=\left[1\right]\LG\left[0\right].
\]
We then define $\left[1\right]\LG\left[\ell\right];\left(A_{i}\right)$
to be the colimit of the diagram below left, taken in $\mathsf{Str}\mhyphen\omega\mhyphen\Cat$,
with the morphisms of that diagram being given in the diagram below
right.

\begin{minipage}[c][1\totalheight][t]{0.45\textwidth}%
\begin{center}
\[\clim\left\{ \vcenter{\vbox{\xyR{1.5pc}\xyC{6pc}\xymatrix@!0{ & {\scriptstyle \left[\ell+1\right];\left(\left[0\right],A_{1},\dots,A_{n}\right)}\\{\scriptstyle \left[\ell\right];\left(A_{i}\right)}\ar[ur]\ar[dr]\\ & {\scriptstyle \left[\ell\right];\left(\left[1\right]\LG A_{1},A_{2},\dots,A_{n}\right)}\\{\scriptstyle \left[\ell\right];\left(A_{i}\right)}\ar[ur]\ar[dr]\\ & {\scriptstyle \left[\ell+1\right];\left(A_{1},\left[0\right],\dots,A_{n}\right)}\\{\scriptstyle \vdots}\ar[ur]\ar[dr]\\ & {\scriptstyle \left[\ell+1\right];\left(A_{1},,\dots A_{\ell-1},\left[0\right],A_{n}\right)}\\{\scriptstyle \left[\ell\right];\left(A_{i}\right)}\ar[ur]\ar[dr]\\ & {\scriptstyle \left[\ell\right];\left(A_{1},\dots,\left[1\right]\LG A_{n}\right)}\\{\scriptstyle \left[\ell\right];\left(A_{i}\right)}\ar[ur]\ar[dr]\\ & {\scriptstyle \left[\ell+1\right];\left(A_{1},,\dots,A_{n},\left[0\right]\right)}}}}\right\} \]
\par\end{center}%
\end{minipage}\hfill{}%
\begin{minipage}[c][1\totalheight][t]{0.45\textwidth}%
$\vcenter{\vbox{\xyR{1.5pc}\xyC{0pc}\xymatrix{{\scriptstyle \vdots} &  & {\scriptscriptstyle \vdots} &  &  &  & {\scriptscriptstyle \vdots} &  & \vdots\\
{\scriptstyle \left[\ell+1\right]} & {\scriptstyle ;} & {\scriptscriptstyle A_{1}} & {\scriptscriptstyle \cdots} & {\scriptscriptstyle \left[0\right]} &  & {\scriptscriptstyle A_{j}} & {\scriptscriptstyle \cdots} & A_{\ell}\\
{\scriptstyle \left[\ell\right]}\ar[d]_{\id}\ar[u]^{d^{j}} & {\scriptstyle ;} & {\scriptscriptstyle A_{1}}\ar[u]^{\id}\ar[d]^{\id} & {\scriptscriptstyle \cdots} &  & {\scriptscriptstyle A_{j}}\ar[ul]^{!}\ar[ur]^{\id}\ar[d]_{\id\LG\left\{ 0\right\} } &  & {\scriptscriptstyle \cdots} & A_{\ell}\ar[u]\ar[d]^{\id}\\
{\scriptstyle \left[\ell\right]} & {\scriptstyle ;} & {\scriptscriptstyle A_{1}} & {\scriptscriptstyle \cdots} &  & {\scriptscriptstyle A_{j}\LG\left[1\right]} &  & {\scriptscriptstyle \cdots} & A_{\ell}\\
{\scriptstyle \left[\ell\right]}\ar[u]^{\id}\ar[d]_{d^{j}} & {\scriptstyle ;} & {\scriptscriptstyle A_{1}}\ar[u]^{\id}\ar[d]^{\id} & {\scriptscriptstyle \cdots} &  & {\scriptscriptstyle A_{j}}\ar[u]^{\id\LG\left\{ 1\right\} }\ar[dl]_{\id}\ar[dr]_{!} &  & {\scriptscriptstyle \cdots} & A_{\ell}\ar[u]^{\id}\ar[d]^{\id}\\
{\scriptstyle \left[\ell+1\right]} & {\scriptstyle ;} & {\scriptscriptstyle A_{1}} & {\scriptscriptstyle \cdots} & {\scriptscriptstyle A_{j}} &  & {\scriptscriptstyle \left[0\right]} & {\scriptscriptstyle \cdots} & A_{\ell}\\
{\scriptstyle \vdots} &  & {\scriptscriptstyle \vdots} &  & {\scriptscriptstyle \vdots} &  &  &  & \vdots
}
}}$%
\end{minipage}

We then set 
\[
\left[\ell\right];\left(A_{i}\right)\xrightarrow{\left\{ 0\right\} \LG\left[\ell\right];\left(A_{i}\right)}\left[1\right]\LG\left[\ell\right];\left(A_{i}\right)
\]
 to be the inclusion 
\[
\left[\ell\right];\left(A_{i}\right)\xrightarrow{d^{0};\left(\id,\dots,\id\right)}\left[\ell+1\right];\left(\left[0\right],A_{1},\dots,A_{n}\right)\longrightarrow\left[1\right]\LG\left[\ell\right];\left(A_{i}\right)
\]
and similarly set
\[
\left[\ell\right];\left(A_{i}\right)\xrightarrow{\left\{ 1\right\} \LG\left[\ell\right];\left(A_{i}\right)}\left[1\right]\LG\left[\ell\right];\left(A_{i}\right)
\]
to be the inclusion
\[
\left[\ell\right];\left(A_{i}\right)\xrightarrow{d^{\ell+1};\left(\id,\dots,\id\right)}\left[\ell+1\right];\left(A_{1},\dots,A_{n},\left[0\right]\right)\longrightarrow\left[1\right]\LG\left[\ell\right];\left(A_{i}\right).
\]

\begin{figure}
\noindent\fbox{\begin{minipage}[t]{1\columnwidth - 2\fboxsep - 2\fboxrule}%
\bigskip{}

\centering
\resizebox{0.75\textwidth}{!}{
\begin{tikzpicture}
	\begin{pgfonlayer}{nodelayer}
		\node [style=none] (38) at (2, 5.5) {};
		\node [style=none] (40) at (8, 12.5) {};
		\node [style=none] (26) at (2, -2.5) {};
		\node [style=none] (27) at (8, -0.5) {};
		\node [style=none] (28) at (8, 4.5) {};
		\node [style=none] (29) at (14, 0.5) {};
		\node [style=none] (52) at (2, -2.5) {};
		\node [style=none] (54) at (8, 4.5) {};
		\node [style=none] (55) at (14, 0.5) {};
		\node [style=none] (56) at (2, -4.5) {};
		\node [style=none] (57) at (8, -2.5) {};
		\node [style=none] (58) at (14, -1.5) {};
		\node [style=none] (59) at (2, -4.5) {};
		\node [style=none] (60) at (8, -2.5) {};
		\node [style=none] (61) at (14, -1.5) {};
		\node [style=none] (62) at (2, -6.5) {};
		\node [style=none] (64) at (14, -3.5) {};
		\node [style=none] (65) at (2, -6.5) {};
		\node [style=none] (66) at (8, -4.5) {};
		\node [style=none] (67) at (14, -3.5) {};
		\node [style=none] (68) at (2, -6.5) {};
		\node [style=none] (69) at (8, -4.5) {};
		\node [style=none] (70) at (14, -3.5) {};
		\node [style=none] (71) at (2, -10.5) {};
		\node [style=none] (72) at (8, -8.5) {};
		\node [style=none] (73) at (14, -11.5) {};
		\node [style=none] (74) at (14, -7.5) {};
		\node [style=none] (22) at (2, -10.5) {};
		\node [style=none] (23) at (8, -8.5) {};
		\node [style=none] (25) at (14, -7.5) {};
		\node [style=none] (79) at (2, -0.5) {};
		\node [style=none] (80) at (8, 6.5) {};
		\node [style=none] (81) at (14, 2.5) {};
		\node [style=none] (46) at (2, 1.5) {};
		\node [style=none] (47) at (8, 8.5) {};
		\node [style=none] (48) at (14, 4.5) {};
		\node [style=none] (82) at (2, -0.5) {};
		\node [style=none] (83) at (8, 6.5) {};
		\node [style=none] (84) at (14, 2.5) {};
		\node [style=none] (85) at (2, 1.5) {};
		\node [style=none] (86) at (8, 8.5) {};
		\node [style=none] (87) at (14, 4.5) {};
		\node [style=none] (10) at (2, 1.5) {};
		\node [style=none] (13) at (14, 4.5) {};
		\node [style=none] (49) at (2, 1.5) {};
		\node [style=none] (50) at (8, 8.5) {};
		\node [style=none] (51) at (14, 4.5) {};
		\node [style=none] (91) at (2, 5.5) {};
		\node [style=none] (92) at (2, 9.5) {};
		\node [style=none] (93) at (8, 12.5) {};
		\node [style=none] (94) at (14, 8.5) {};
		\node [style=none] (42) at (2, 5.5) {};
		\node [style=none] (43) at (2, 9.5) {};
		\node [style=none] (45) at (14, 8.5) {};
		\node [style=none] (129) at (-8, 12.5) {};
		\node [style=none] (151) at (-14, -10.5) {};
		\node [style=none] (152) at (-2, -11.5) {};
		\node [style=none] (154) at (-14, -10.5) {};
		\node [style=none] (155) at (-8, -8.5) {};
		\node [style=none] (156) at (-2, -11.5) {};
		\node [style=none] (158) at (-14, -10.5) {};
		\node [style=none] (159) at (-8, -8.5) {};
		\node [style=none] (179) at (-14, 9.5) {};
		\node [style=none] (180) at (-8, 12.5) {};
		\node [style=none] (181) at (-2, 8.5) {};
		\node [style=none] (186) at (-14, 9.5) {};
		\node [style=none] (187) at (-2, 8.5) {};
		\node [style=none] (188) at (-14, 9.5) {};
		\node [style=none] (189) at (-2, 8.5) {};
		\node [style=none] (190) at (-8, 12.5) {};
		\node [style=none] (238) at (-8, 12.5) {};
		\node [style=none] (240) at (8, 12.5) {};
		\node [style=none] (241) at (8, -11.5) {};
		\node [style=none] (245) at (8, 7.25) {};
		\node [style=none] (247) at (8, -7.5) {};
		\node [style=none] (248) at (8, 5) {};
		\node [style=none] (250) at (8, -5) {};
		\node [style=none] (253) at (8, 1) {};
		\node [style=none] (254) at (8, 0) {};
		\node [style=none] (256) at (2, 3.5) {};
		\node [style=none] (257) at (8, 10.5) {};
		\node [style=none] (258) at (14, 6.5) {};
		\node [style=none] (35) at (2, 9.5) {};
		\node [style=none] (36) at (14, 8.5) {};
		\node [style=none] (37) at (8, 12.5) {};
		\node [style=none] (30) at (2, 5.5) {};
		\node [style=none] (31) at (2, 9.5) {};
		\node [style=none] (33) at (14, 8.5) {};
		\node [style=none] (259) at (2, -8.5) {};
		\node [style=none] (260) at (14, -5.5) {};
		\node [style=none] (261) at (2, -8.5) {};
		\node [style=none] (262) at (14, -5.5) {};
		\node [style=none] (263) at (2, -8.5) {};
		\node [style=none] (264) at (8, -6.5) {};
		\node [style=none] (265) at (14, -5.5) {};
		\node [style=none] (75) at (2, -10.5) {};
		\node [style=none] (77) at (14, -11.5) {};
		\node [style=none] (78) at (14, -7.5) {};
		\node [style=none] (272) at (-14, -10.5) {};
		\node [style=none] (273) at (-2, -11.5) {};
		\node [style=none] (274) at (-8, -7.5) {};
		\node [style=none] (275) at (-4.75, 11) {};
		\node [style=none] (276) at (3.75, 11) {};
		\node [style=none] (277) at (-8, -7.5) {};
		\node [style=none] (278) at (-14, -10.5) {};
		\node [style=none] (279) at (-8, -7.5) {};
		\node [style=none] (280) at (-2, -11.5) {};
		\node [style=none] (281) at (-14, -10.5) {};
		\node [style=none] (282) at (-2, -11.5) {};
		\node [style=none] (283) at (-14, -10.5) {};
		\node [style=none] (284) at (-2, -11.5) {};
		\node [style=none] (285) at (-8, -7.5) {};
		\node [style=none] (286) at (-8, -7.5) {};
		\node [style=none] (287) at (-4.75, -9.25) {};
		\node [style=none] (288) at (3.75, -9.25) {};
	\end{pgfonlayer}
	\begin{pgfonlayer}{edgelayer}
		\draw [style=hidden edges, bend left=15] (38.center) to (40.center);
		\draw [style=solid arrow] (42.center) to (43.center);
		\draw [style=solid arrow] (43.center) to (45.center);
		\draw [style=solid arrow, bend left=15] (42.center) to (45.center);
		\draw [style=hidden edges] (72.center) to (73.center);
		\draw [style=solid arrow] (68.center) to (69.center);
		\draw [style=solid arrow, bend right=15] (68.center) to (70.center);
		\draw [style=solid arrow] (69.center) to (70.center);
		\draw [style=solid arrow] (68.center) to (70.center);
		\draw [style=hidden edges, bend right=15, looseness=1.25] (69.center) to (70.center);
		\draw [style=Tops] (25.center)
			 to [bend left=15] (22.center)
			 to (23.center)
			 to [bend right=15, looseness=1.25] cycle;
		\draw [style=solid arrow] (28.center) to (29.center);
		\draw [style=hidden edges] (26.center) to (27.center);
		\draw [style=hidden edges] (27.center) to (28.center);
		\draw [style=hidden edges] (27.center) to (29.center);
		\draw [style=solid arrow] (26.center) to (28.center);
		\draw [style=solid arrow] (26.center) to (29.center);
		\draw [style=hidden edges] (46.center) to (47.center);
		\draw [style=Tops] (55.center)
			 to (54.center)
			 to (52.center)
			 to cycle;
		\draw [style=solid arrow] (56.center) to (57.center);
		\draw [style=solid arrow] (57.center) to (58.center);
		\draw [style=solid arrow] (56.center) to (58.center);
		\draw [style=Tops] (60.center)
			 to (61.center)
			 to (59.center)
			 to cycle;
		\draw [style=Fronts] (64.center)
			 to (62.center)
			 to [bend right=15] cycle;
		\draw [style=Tops] (66.center)
			 to (67.center)
			 to (65.center)
			 to cycle;
		\draw [style=solid arrow] (73.center) to (74.center);
		\draw [style=solid arrow] (71.center) to (73.center);
		\draw [style=solid arrow] (71.center) to (72.center);
		\draw [style=solid arrow, bend right=15] (71.center) to (74.center);
		\draw [style=solid arrow, bend right=15, looseness=1.25] (72.center) to (74.center);
		\draw [style=solid arrow] (80.center)
			 to (81.center)
			 to (79.center)
			 to cycle;
		\draw [style=Tops] (84.center)
			 to (83.center)
			 to (82.center)
			 to cycle;
		\draw [style=solid arrow] (85.center)
			 to [bend left=15] (87.center)
			 to (86.center)
			 to [bend right=15] cycle
			 to (87.center);
		\draw [style=Fronts] (13.center)
			 to (10.center)
			 to [bend left=15] cycle;
		\draw [style=Tops] (50.center)
			 to [bend right=15] (49.center)
			 to [bend left=15] (51.center)
			 to cycle;
		\draw [style=solid arrow] (92.center) to (91.center);
		\draw [style=solid arrow] (94.center) to (92.center);
		\draw [style=solid arrow] (92.center) to (93.center);
		\draw (37.center) to (36.center);
		\draw [style=solid arrow] (186.center) to (187.center);
		\draw [style=solid arrow] (181.center) to (179.center);
		\draw [style=solid arrow] (179.center) to (180.center);
		\draw [style=Tops] (190.center)
			 to (188.center)
			 to (189.center)
			 to cycle;
		\draw (190.center) to (189.center);
		\draw [style=Tops] (257.center)
			 to [bend right=15] (256.center)
			 to [bend left=15] (258.center)
			 to cycle;
		\draw [style=Tops] (37.center)
			 to (35.center)
			 to (36.center)
			 to cycle;
		\draw [style=Fronts] (33.center)
			 to (31.center)
			 to (30.center)
			 to [bend left=15] cycle;
		\draw [style=solid arrow, bend left=15] (91.center) to (94.center);
		\draw [style=Tops] (265.center)
			 to [bend left=15] (263.center)
			 to (264.center)
			 to [bend right=15, looseness=1.25] cycle;
		\draw [style=solid arrow, bend right=15] (261.center) to (262.center);
		\draw [style=Fronts] (75.center)
			 to (77.center)
			 to (78.center)
			 to [bend left=15] cycle;
		\draw [style=Tops] (274.center)
			 to (272.center)
			 to (273.center)
			 to cycle;
		\draw (274.center) to (273.center);
		\draw [style=arrow] (276.center) to (275.center);
		\draw [style=solid arrow] (281.center) to (282.center);
		\draw [style=solid arrow] (280.center) to (278.center);
		\draw [style=solid arrow] (278.center) to (279.center);
		\draw [style=Tops] (285.center)
			 to (283.center)
			 to (284.center)
			 to cycle;
		\draw (285.center) to (284.center);
		\draw [style=arrow] (288.center) to (287.center);
	\end{pgfonlayer}
\end{tikzpicture}
}\bigskip{}
\caption{\label{fig:Illustration-of-the-endpoints}Illustration of the endpoint
inclusions $\left\{ 0\right\} \otimes\left[2\right]$ and $\left\{ 1\right\} \otimes\left[2\right]$}
\bigskip{}
\end{minipage}}
\end{figure}
\begin{comment}
Consider that $\left[1\right]\LG\left[2\right]$ is the colimit
\[
\colim\left\{ \vcenter{\vbox{\xyR{2pc}\xyC{.5pc}\xymatrix{ & \left[3\right];\left(\underline{\left[0\right]},\left[0\right]\left[0\right]\right)\\
\left[2\right];\left(\left[0\right],\left[0\right]\right)\ar[ur]^{d^{1};\left(!\times!,!\right)}\ar[dr]_{\id;\left(\left\{ 0\right\} \LG!,!\right)}\\
 & \left[2\right];\left(\left[1\right],\left[0\right]\right)\\
\left[2\right];\left(\left[0\right],\left[0\right]\right)\ar[ur]^{\id;\left(\left\{ 1\right\} \LG!,!\right)}\ar[dr]_{d^{1};\left(!\times!,!\right)}\\
 & \left[3\right];\left(\left[0\right],\underline{\left[0\right]},\left[0\right]\right)\\
\left[2\right];\left(\left[0\right],\left[0\right]\right)\ar[ur]^{d^{2};\left(!,!\times!\right)}\ar[dr]_{\id;\left(!,\left\{ 0\right\} \LG!\right)}\\
 & \left[2\right];\left(\left[0\right],\left[1\right]\right)\\
\left[2\right];\left(\left[0\right],\left[0\right]\right)\ar[ur]^{\id;\left(!,\left\{ 1\right\} \LG!\right)}\ar[dr]_{d^{2};\left(!,!\times!\right)}\\
 & \left[3\right];\left(\left[0\right],\left[0\right],\underline{\left[0\right]}\right)
}
}}\right\} 
\]
where the underscored $\underline{\left[0\right]}$ is just a bookkeeping
device and is just the $0$-simplex.
\end{comment}
\end{defn}

\begin{rem}
An example of the maps $\left\{ 0\right\} \LG\left(\_\right)$ and
$\left\{ 1\right\} \LG\left(\_\right)$ is illustrated in Figure \ref{fig:Illustration-of-the-endpoints}.
The connection between these formulae and the pasting diagrams for
$\left[1\right]\times\left[n\right];\left(T_{1},T_{2},\dots,T_{n}\right)$
and $\left[1\right]\otimes\left[n\right];\left(T_{1},T_{2},\dots,T_{n}\right)$
can be seen in Figure \ref{fig:justifying the shuffle} 
\begin{figure}
\noindent\fbox{\begin{minipage}[t]{1\columnwidth - 2\fboxsep - 2\fboxrule}%
\bigskip{}
\adjustbox{scale=.75,center}{
\begin{tikzcd}[column sep=large, row sep=large, ampersand replacement=\&]
	% the bullets
	\bullet \&\& \bullet \&\& \bullet \&\& \bullet \&\& \bullet \\
	\\
	\bullet \&\& \bullet \&\& \bullet \&\& \bullet \&\& \bullet
	% the arrows
	\arrow["{[0]}"{description}, from=3-1, to=1-1] 	\arrow["{T_1}"{description}, from=1-1, to=1-3] 	\arrow["{T_1}"{description}, from=3-1, to=3-3] 	\arrow["{[0]}"{description}, from=3-3, to=1-3] 	\arrow["{T_2}"{description}, from=1-3, to=1-5] 	\arrow["{T_2}"{description}, from=3-3, to=3-5] 	\arrow["{T_1}"{description}, from=3-1, to=1-3] 	\arrow["{T_2}"{description}, from=3-3, to=1-5] 	\arrow["\cdots"{description}, from=3-5, to=3-7] 	\arrow["{[0]}"{description}, from=3-5, to=1-5] 	\arrow["\cdots"{description}, from=1-5, to=1-7] 	\arrow["{[0]}"{description}, from=3-7, to=1-7] 	\arrow["{T_n}"{description}, from=3-7, to=1-9] 	\arrow["{T_n}"{description}, from=1-7, to=1-9] 	\arrow["{T_n}"{description}, from=3-7, to=3-9] 	\arrow["{[0]}"{description}, from=3-9, to=1-9] 	\arrow["\ddots"{description}, from=3-5, to=1-7]
 \end{tikzcd}
}\caption{\label{fig:justifying the shuffle}$\Theta$-graph pasting diagram
for $\left[1\right]\times\left[n\right];\left(T_{1},T_{2},\dots,T_{n}\right)$}
\bigskip{}
\end{minipage}}
\end{figure}
\begin{figure}
\noindent\fbox{\begin{minipage}[t]{1\columnwidth - 2\fboxsep - 2\fboxrule}%
\bigskip{}
\adjustbox{scale=.75,center}{
\begin{tikzcd}[column sep=large, row sep=large, ampersand replacement=\&]
	\bullet \&\& \bullet \&\& \bullet \&\& \bullet \&\& \bullet \\
	\\
	\bullet \&\& \bullet \&\& \bullet \&\& \bullet \&\& \bullet
	% lower arrows
	\arrow["{[0]}"{description}, from=3-1, to=1-1] 	\arrow["{T_1}"{description}, from=1-1, to=1-3]
	\arrow[""{name=0, anchor=center, inner sep=0}, "{T_1}"{description}, curve={height=-24pt}, from=3-1, to=1-3]
	\arrow[""{name=1, anchor=center, inner sep=0}, "{T_1}"{description}, curve={height=24pt}, from=3-1, to=1-3]
	\arrow["{T_1}"{description}, from=3-1, to=3-3]
	\arrow["{[0]}"{description}, from=3-3, to=1-3]
	\arrow["{T_2}"{description}, from=1-3, to=1-5]
	\arrow["{T_2}"{description}, from=3-3, to=3-5]
	\arrow["{[0]}"{description}, from=3-5, to=1-5]
	\arrow[""{name=2, anchor=center, inner sep=0}, "{T_2}"{description}, curve={height=-24pt}, from=3-3, to=1-5] 	\arrow[""{name=3, anchor=center, inner sep=0}, "{T_2}"{description}, curve={height=24pt}, from=3-3, to=1-5] 	\arrow["\cdots"{description}, from=1-5, to=1-7] 	\arrow["\cdots"{description}, from=3-5, to=3-7] 	\arrow["{[0]}"{description}, from=3-7, to=1-7] 	\arrow["{T_n}"{description}, from=1-7, to=1-9] 	\arrow["{T_n}"{description}, from=3-7, to=3-9] 	\arrow["{[0]}"{description}, from=3-9, to=1-9] 	\arrow["\ddots"{description}, curve={height=-24pt}, from=3-5, to=1-7] 	\arrow["\ddots"{description}, curve={height=24pt}, from=3-5, to=1-7] 	\arrow[""{name=4, anchor=center, inner sep=0}, "{T_n}"{description}, curve={height=-24pt}, from=3-7, to=1-9]
	\arrow[""{name=5, anchor=center, inner sep=0}, "{T_n}"{description}, curve={height=24pt}, from=3-7, to=1-9]
	\arrow["{[1]\otimes T_1}"', shorten <=5pt, shorten >=5pt, Rightarrow, from=1, to=0]
	\arrow["{[1]\otimes T_2}"', shorten <=5pt, shorten >=5pt, Rightarrow, from=3, to=2]
	\arrow["{[1]\otimes T_n}"', shorten <=5pt, shorten >=5pt, Rightarrow, from=5, to=4]
\end{tikzcd}
}\caption{\label{fig:justifying the shuffle-1}$\Theta$-$2$-graph pasting
diagram for $\left[1\right]\otimes\left[n\right];\left(T_{1},T_{2},\dots,T_{n}\right)$}
\bigskip{}
\end{minipage}}
\end{figure}

\pagebreak{}
\end{rem}

\subsection{\label{subsec:The-assignment-LG-preserves-globular-sums}The Lax
shuffle decomposition preserves globular sums}
\begin{thm}
The lax shuffle decomposition, the assignment on objects 
\[
\left[1\right]\LG\left(\_\right):\Ob\left(\Theta\right)\longrightarrow\Ob\left(\mathsf{Str}\mhyphen\omega\mhyphen\Cat\right)
\]
 preserves globular sums in the sense that for any $\ell,n_{0},\dots,n_{\ell},m_{1},\dots,m_{\ell}$
with $n_{i-1}\geq m_{i}\leq n_{i}$ for all $1\leq i\leq\ell$, the
canonical comparison map between the $\omega$-categories
\[
\left[1\right]\LG\text{\ensuremath{\overline{n_{0}}}}\underset{\left[1\right]\LG\overline{m_{1}}}{\bigoplus}\left[1\right]\LG\overline{n_{1}}\bigoplus\cdots\bigoplus\left[1\right]\LG\overline{n_{\ell-1}}\underset{\left[1\right]\LG\overline{m_{\ell}}}{\bigoplus}\left[1\right]\LG\overline{n_{\ell}}
\]
 and 
\[
\left[1\right]\LG\left(\overline{n_{0}}\underset{\overline{m_{1}}}{\bigoplus}\overline{n_{1}}\cdots\overline{n_{\ell-1}}\underset{\overline{m_{\ell}}}{\bigoplus}\overline{n_{\ell}}\right)
\]
is an isomorphism.
\end{thm}

The proof proceeds by induction. The base case of $n=0$ is treated
as Lemma \ref{lem:globular-sum-lax-gray-proof!}, whereas the induction
is treated as Lemma \ref{cor:lax-gray-shuffle-preserves-all-binary-globular-sums}.
The proofs themselves are long because of numerous large diagrams,
though they are not difficult. Indeed the picture which underlies
the formalism is rather intuitive. For example the $0$-globular case
follows from seeing the consideration of two possible decompositions
of the pasting diagram for $\left[1\right]\otimes\left[2\right];\left(A,B\right)$
(See Figure \ref{Pasting Diagram for -----}).

\begin{figure}
\noindent\fbox{\begin{minipage}[t]{1\columnwidth - 2\fboxsep - 2\fboxrule}%
\bigskip{}
\adjustbox{scale=.65,center}{
\begin{tikzcd}[ampersand replacement=\&]
	{\;} \& \bullet \&\& \bullet \&\& \bullet \\
	\&\&\& \bullet \&\& \bullet \\
	\& \bullet \&\& \bullet \&\& \bullet \&\& \bullet \&\& \bullet \&\& \bullet \&\& \bullet \&\& \bullet \& \bullet \&\& \bullet \\
	\& \bullet \&\&\&\& \bullet \\
	\& \bullet \&\& \bullet \&\& \bullet \&\& \bullet \&\& \bullet \&\& \bullet \&\& \bullet \&\& \bullet \& \bullet \&\& \bullet \\
	\& \bullet \&\& \bullet \\
	\& \bullet \&\& \bullet \&\& \bullet
	\arrow[from=3-2, to=1-2]
	\arrow["A"{description}, from=1-2, to=1-4]
	\arrow[from=1-4, to=1-6]
	\arrow["A"{description}, curve={height=-12pt}, from=3-2, to=1-4] 	\arrow[from=2-4, to=2-6]
	\arrow["B"{description}, from=3-4, to=3-6]
	\arrow["A"{description}, curve={height=12pt}, from=5-2, to=3-4]
	\arrow[""{name=0, anchor=center, inner sep=0}, "A"{description}, curve={height=-12pt}, from=4-2, to=2-4]
	\arrow["A"{description}, from=5-2, to=5-4]
	\arrow[from=5-4, to=3-4]
	\arrow["B"{description}, curve={height=-12pt}, from=5-4, to=3-6]
	\arrow["A"{description}, from=6-2, to=6-4]
	\arrow[""{name=1, anchor=center, inner sep=0}, "B"{description}, curve={height=-12pt}, from=6-4, to=4-6]
	\arrow[""{name=2, anchor=center, inner sep=0}, "B"{description}, curve={height=12pt}, from=6-4, to=4-6]
	\arrow[""{name=3, anchor=center, inner sep=0}, "A"{description}, curve={height=12pt}, from=4-2, to=2-4]
	\arrow["A"{description}, from=7-2, to=7-4]
	\arrow["B"{description}, curve={height=12pt}, from=7-4, to=5-6] 	\arrow["B"{description}, from=7-4, to=7-6]
	\arrow[from=7-6, to=5-6]
	\arrow[from=5-8, to=3-8]
	\arrow["A"{description}, from=3-8, to=3-10]
	\arrow[""{name=4, anchor=center, inner sep=0}, "A"{description}, curve={height=-12pt}, from=5-8, to=3-10]
	\arrow[""{name=5, anchor=center, inner sep=0}, "A"{description}, curve={height=12pt}, from=5-8, to=3-10]
	\arrow["A"{description}, from=5-8, to=5-10]
	\arrow[from=5-10, to=3-10]
	\arrow["B"{description}, from=3-10, to=3-12]
	\arrow["B"{description}, from=5-10, to=5-12]
	\arrow[from=5-12, to=3-12]
	\arrow[""{name=6, anchor=center, inner sep=0}, "B"{description}, curve={height=-12pt}, from=5-10, to=3-12]
	\arrow[""{name=7, anchor=center, inner sep=0}, "B"{description}, curve={height=12pt}, from=5-10, to=3-12]
	\arrow[from=5-14, to=3-14]
	\arrow["A"{description}, from=3-14, to=3-16]
	\arrow[""{name=8, anchor=center, inner sep=0}, "A"{description}, curve={height=-12pt}, from=5-14, to=3-16]
	\arrow[""{name=9, anchor=center, inner sep=0}, "A"{description}, curve={height=12pt}, from=5-14, to=3-16]
	\arrow["A"{description}, from=5-14, to=5-16]
	\arrow[from=5-16, to=3-16]
	\arrow[from=5-17, to=3-17]
	\arrow["B"{description}, from=3-17, to=3-19]
	\arrow["B"{description}, from=5-17, to=5-19]
	\arrow[from=5-19, to=3-19]
	\arrow[""{name=10, anchor=center, inner sep=0}, "B"{description}, curve={height=-12pt}, from=5-17, to=3-19]
	\arrow[""{name=11, anchor=center, inner sep=0}, "B"{description}, curve={height=12pt}, from=5-17, to=3-19]
	\arrow["{[1] \otimes A}"', shorten <=4pt, shorten >=4pt, Rightarrow, from=3, to=0]
	\arrow["{[1] \otimes B}"', shorten <=4pt, shorten >=4pt, Rightarrow, from=2, to=1]
	\arrow["{[1] \otimes A}"', shorten <=4pt, shorten >=4pt, Rightarrow, from=5, to=4]
	\arrow["{[1] \otimes B}"', shorten <=4pt, shorten >=4pt, Rightarrow, from=7, to=6]
	\arrow["{[1] \otimes A}"', shorten <=4pt, shorten >=4pt, Rightarrow, from=9, to=8]
	\arrow["{[1] \otimes B}"', shorten <=4pt, shorten >=4pt, Rightarrow, from=11, to=10]
\end{tikzcd}
}

\caption{\label{Pasting Diagram for -----}Two pasting diagrams for $\left[1\right]\protect\LG\left[2\right];\left(A,B\right)$.
The ``vertical'' pasting diagram for $\left[1\right]\otimes\left[2\right];\left(A,B\right)$
and the pasting of $\left[1\right]\protect\LG\left[1\right];\left(A\right)$
with $\left[1\right]\protect\LG\left[1\right];\left(B\right)$.}
\bigskip{}
\end{minipage}}
\end{figure}

We now attend to the promised Lemmata and their proofs.
\begin{lem}
\label{lem:globular-sum-lax-gray-proof!} The functor $\left[1\right]\LG\left(\_\right)$
preserves $0$-globular sums in the following sense: for 
\[
\left[n\right];\left(A_{i}\right)=\left[1\right];A_{1}\underset{0}{\oplus}\cdots\underset{0}{\oplus}\left[1\right];A_{n}
\]
we have that the canonical comparison map
\[
\left(\left[1\right]\LG\left[1\right];\left(A_{1}\right)\right)\underset{1}{\oplus}\cdots\underset{1}{\oplus}\left(\left[1\right]\LG\left[1\right];\left(A_{n}\right)\right)\liso\left[1\right]\LG\left[n\right];\left(A_{i}\right)
\]
 is an isomorphism as indicated.
\end{lem}

\begin{proof}
We'll explicitly show that we've a canonical isomorphism
\[
\left[1\right]\LG\left[2\right];\left(A,B\right)\liso\left[1\right]\LG\left[1\right];\left(A\right)\underset{1}{\bigoplus}\left[1\right]\LG\left[1\right];\left(B\right)
\]
A nearly identical argument provides that canonical map between
\[
\left[1\right]\LG\left[n\right];\left(A_{1},\dots,A_{n}\right)\underset{1}{\bigoplus}\left[1\right]\LG\left[m\right];\left(B_{1},\dots,B_{m}\right)
\]
 and 
\[
\left[1\right]\LG\left[n+m\right];\left(A_{1},\dots,A_{n},B_{1},\dots,B_{m}\right)
\]
 is an isomorphism, and the general case $k$-ary, instead of binary
case, follows therefrom.

We begin with $\left[1\right]\LG\left[1\right];\left(A\right)\underset{1}{\bigoplus}\left[1\right]\LG\left[1\right];\left(A\right)$.
This $\omega$-category is but the colimit of the diagram below left.
We observe that the colimit of the central span in the diagram is
but $\left[3\right];\left(A,0,B\right)$ so the diagram above has
the very-same colimit as the diagram below right.%
\[
\vcenter{\vbox{\xyR{1pc}\xyC{1pc}\xymatrix{{\scriptstyle \left[2\right];\left(0,A\right)} &  &  &  & {\scriptstyle \left[2\right];\left(0,A\right)}\\
{\scriptstyle \left[1\right];\left(A\right)}\ar[u]\ar[d] &  &  &  & {\scriptstyle \left[1\right];\left(A\right)}\ar[u]\ar[d]\\
{\scriptstyle \left[1\right];\left(\left[1\right]\LG A\right)} &  &  &  & {\scriptstyle \left[1\right];\left(\left[1\right]\LG A\right)}\\
{\scriptstyle \left[1\right];\left(A\right)}\ar[u]\ar[d] &  &  &  & {\scriptstyle \left[1\right];\left(A\right)}\ar[u]\ar[ddr]\\
{\scriptstyle \left[2\right];\left(A,0\right)}\\
 & {\scriptstyle \left[1\right];\left(0\right)}\ar[ul]\ar[dr] &  &  &  & {\scriptstyle \left[3\right];\left(A,0,B\right)}\\
 &  & {\scriptstyle \left[2\right];\left(0,B\right)}\\
 &  & {\scriptstyle \left[1\right];\left(B\right)}\ar[u]\ar[d] &  &  &  & {\scriptstyle \left[1\right];\left(B\right)}\ar[uul]\ar[d]\\
 &  & {\scriptstyle \left[1\right];\left(\left[1\right]\LG B\right)} &  &  &  & {\scriptstyle \left[1\right];\left(\left[1\right]\LG B\right)}\\
 &  & {\scriptstyle \left[1\right];\left(B\right)}\ar[u]\ar[d] &  &  &  & {\scriptstyle \left[1\right];\left(B\right)}\ar[u]\ar[d]\\
 &  & {\scriptstyle \left[2\right];\left(B,0\right)} &  &  &  & {\scriptstyle \left[2\right];\left(B,0\right)}
}
}}
\]
But this diagram may be partially completed (dotted arrows) into the
commutative diagram below left which, by the universal property of
the push-out can be further completed (dashed arrows). But we then
find that the diagram below left shares it's colimit with the diagram
(solid arrows) below right.

\[
\xyR{1pc}\xyC{1pc}\xymatrix{{\scriptstyle \left[2\right];\left(0,A\right)} &  & {\scriptstyle \left[0\right]}\ar@/_{2pc}/@{..>}[dddll]\ar@{..>}[ddd] & {\scriptstyle \left[2\right];\left(0,A\right)}\ar@{-->}[r] & {\scriptstyle \left[3\right];\left(0,A,B\right)}\pushoutcorner & {\scriptstyle \left[0\right]}\ar@/_{3pc}/@{..>}[dll]\ar@{..>}[d]\\
{\scriptstyle \left[1\right];\left(A\right)}\ar[u]\ar[d] &  &  & {\scriptstyle \left[1\right];\left(A\right)}\ar[u]\ar[dr]\ar@{-->}[r] & {\scriptstyle \left[2\right];\left(A,B\right)}\ar@{-->}[u]\ar@{-->}[d]\pushoutcorner & {\scriptstyle \left[1\right];\left(B\right)}\ar@{..>}[dl]\ar@{-->}[l]\\
{\scriptstyle \left[1\right];\left(\left[1\right]\LG A\right)}\ar@{-->}[r] & {\scriptstyle \left[2\right];\left(\left[1\right]\LG A,B\right)}\pushoutcorner &  &  & {\scriptstyle \left[2\right];\left(\left[1\right]\LG A,B\right)}\\
{\scriptstyle \left[1\right];\left(A\right)}\ar[u]\ar[dr]\ar@{-->}[r] & {\scriptstyle \left[2\right];\left(A,B\right)}\ar@{-->}[u]\ar@{-->}[d]\pushoutcorner & {\scriptstyle \left[1\right];\left(B\right)}\ar@{=}[dd]\ar[dl]\ar@{-->}[l] &  & {\scriptstyle \left[2\right];\left(A,B\right)}\ar[u]\ar[d]\\
 & {\scriptstyle \left[3\right];\left(A,0,B\right)} &  &  & {\scriptstyle \left[3\right];\left(A,0,B\right)}\\
{\scriptstyle \left[1\right];\left(A\right)}\ar@{=}[uu]\ar[ur]\ar@{-->}[r] & {\scriptstyle \left[2\right];\left(A,B\right)}\ar@{-->}[u]\ar@{-->}[d]\pushoutcorner & {\scriptstyle \left[1\right];\left(B\right)}\ar[ul]\ar[d]\ar@{-->}[l] &  & {\scriptstyle \left[2\right];\left(A,B\right)}\ar[u]\ar[d]\\
 & {\scriptstyle \left[2\right];\left(A,\left[1\right]\LG B\right)}\pushoutcorner & {\scriptstyle \left[1\right];\left(\left[1\right]\LG B\right)}\ar@{-->}[l] &  & {\scriptstyle \left[2\right];\left(A,\left[1\right]\LG B\right)}\\
 &  & {\scriptstyle \left[1\right];\left(B\right)}\ar[u]\ar[d] & {\scriptstyle \left[1\right];\left(A\right)}\ar@{..>}[ur]\ar@{-->}[r] & {\scriptstyle \left[2\right];\left(A,B\right)}\ar@{-->}[u]\ar@{-->}[d]\pushoutcorner & {\scriptstyle \left[1\right];\left(B\right)}\ar[ul]\ar[d]\ar@{-->}[l]\\
{\scriptstyle \left[0\right]}\ar@/_{1pc}/@{..>}[uuurr]\ar@{..>}[uuu] &  & {\scriptstyle \left[2\right];\left(B,0\right)} & {\scriptstyle \left[0\right]}\ar@/_{2pc}/@{..>}[urr]\ar@{..>}[u] & {\scriptstyle \left[3\right];\left(A,B,0\right)}\pushoutcorner & {\scriptstyle \left[2\right];\left(B,0\right)}\ar@{-->}[l]
}
\]

But that above right diagram can be similarly completed (dotted arrows),
and partially computed (dashed arrows), thereby sharing a colimit
with the diagram

\[
\xyR{1pc}\xyC{1pc}\xymatrix{{\scriptstyle \left[3\right];\left(0,A,B\right)}\\
{\scriptstyle \left[2\right];\left(A,B\right)}\ar[u]\ar[d]\\
{\scriptstyle \left[2\right];\left(\left[1\right]\LG A,B\right)}\\
{\scriptstyle \left[2\right];\left(A,B\right)}\ar[u]\ar[d]\\
{\scriptstyle \left[3\right];\left(A,0,B\right)}\\
{\scriptstyle \left[2\right];\left(A,B\right)}\ar[d]\ar[u]\\
{\scriptstyle \left[2\right];\left(A,\left[1\right]\LG B\right)}\\
{\scriptstyle \left[2\right];\left(B\right)}\ar[u]\ar[d]\\
{\scriptstyle \left[3\right];\left(A,B,0\right)}
}
\]

which is but $\left[1\right]\LG\left[2\right];\left(A,B\right)$.
\end{proof}
Prior to proving that $\left[1\right]\LG\left(\_\right)$ preserves
$1,2,\dots$-globular sum, we introduce some simplifying notation.
\begin{defn}
Given an cell $T$ of $\Theta$ and some $n\in\N$, let $\overline{n};T$
denote the cell 
\[
\underbrace{\left[1\right];\left[1\right];\cdots;\left[1\right]}_{n-\mathrm{many}};T
\]
\end{defn}

\begin{lem}
\label{cor:lax-gray-shuffle-preserves-all-binary-globular-sums}\emph{The
assignment 
\[
\left[1\right]\LG\left(\_\right):\Ob\left(\Theta\right)\longrightarrow\Ob\left(\mathsf{Str}\mhyphen\omega\mhyphen\Cat\right)
\]
 preserves $n$-globular sums for all $n\geq1$ in the sense that
for any such $n$ and any cells $X,Y,\dots,Z$ of $\Theta$, the canonical
comparison map between 
\[
\left[1\right]\LG\left(\overline{n};X\right)\underset{\left[1\right]\LG\overline{n}}{\bigoplus}\left[1\right]\LG\left(\overline{n};Y\right)\underset{\left[1\right]\LG\overline{n}}{\bigoplus}\cdots\underset{\left[1\right]\LG\overline{n}}{\bigoplus}\left[1\right]\LG\left(\overline{n};Z\right)
\]
and
\[
\left[1\right]\LG\left(\overline{n};X\underset{\overline{n}}{\bigoplus}\overline{n};Y\underset{\overline{n}}{\bigoplus}\cdots\underset{\overline{n}}{\bigoplus}\overline{n};Z\right)
\]
 is an isomorphism.}
\end{lem}

\begin{rem}
The sophisticated reader may wonder why there is much left to prove
here. Indeed as $\overline{n};\left(\_\right)$ takes $0$-globular
sums to $n$-globular ones and preserves colimits and $\left[1\right]\LG\left(\_\right)$
is defined by way of colimits. We cannot however naively commute one
past the other as the choice of diagram over which $\left[1\right]\LG\left(\_\right)$
is a colimit is \emph{not }functorial over all of $\Theta$, instead
it is only functorial for the non-full subcategory $\left[1\right];\Theta\hookrightarrow\Theta$.
It should be possible for one to continue down this path and arrive
at another proof of the lemma, but it is likely more trouble than
it is worth.
\end{rem}

\begin{proof}
As with the previous lemma, we'll prove the claim for binary sums.
The general case of $k$-ary sums merely requires larger diagrams,
whence we leave it to the reader.

Let $n\geq1$ be given and assume that for any cells $X$ and $Y$of
$\Theta$ and any $m\leq n$ that we have canonical isomorphisms 
\[
\left[1\right]\LG\left(\overline{m};X\underset{\overline{m}}{\bigoplus}\overline{m};Y\right)\osi\left[1\right]\LG\left(\overline{m};X\right)\underset{\left[1\right]\LG\overline{m}}{\bigoplus}\left[1\right]\LG\left(\overline{m};Y\right)
\]
we will prove that, for any cells $Z$ and $W$ of $\Theta$, there
is a canonical isomorphism
\[
\left[1\right]\LG\left(\overline{n+1};Z\underset{\overline{n+1}}{\bigoplus}\overline{n+1};W\right)\osi\left[1\right]\LG\left(\overline{n+1};Z\right)\underset{\left[1\right]\LG\overline{n+1}}{\bigoplus}\left[1\right]\LG\left(\overline{n+1};W\right)
\]
To that end see that the colimit of the diagram below left is $\left[1\right]\LG\left(\overline{n+1};Z\underset{\overline{n+1}}{\bigoplus}\overline{n+1};W\right)$,
since $\overline{n+1};Z\underset{\overline{n+1}}{\bigoplus}\overline{n+1};W=\left[1\right];\left(\overline{n};\left(Z\underset{0}{\oplus}W\right)\right)$,
and the colimit of the diagram below right is $\left[1\right]\LG\left(\overline{n+1};Z\right)\underset{\left[1\right]\LG\overline{n+1}}{\bigoplus}\left[1\right]\LG\left(\overline{n+1};W\right)$.
\[
\vcenter{\vbox{\xyR{1.5pc}\xyC{1.5pc}\xymatrix{{\scriptstyle \left[2\right];\left(0,\left(\overline{n};Z\underset{0}{\bigoplus}W\right)\right)} & {\scriptstyle \left[2\right];\left(0,\left(\overline{n};Z\right)\right)} & {\scriptstyle \left[2\right];\left(0,\overline{n}\right)}\ar[l]\ar[r] & {\scriptstyle \left[2\right];\left(0,\left(\overline{n};W\right)\right)}\\
{\scriptstyle \left[1\right];\left(\overline{n};Z\underset{0}{\bigoplus}W\right)}\ar[d]\ar[u] & {\scriptstyle \left[1\right];\left(\overline{n};Z\right)}\ar[d]\ar[u] & {\scriptstyle \left[1\right];\left(\overline{n}\right)}\ar[l]\ar[r]\ar[d]\ar[u] & {\scriptstyle \left[1\right];\left(\overline{n};W\right)}\ar[d]\ar[u]\\
{\scriptstyle \left[1\right];\left(\left[1\right]\LG\left(\overline{n};Z\underset{0}{\bigoplus}W\right)\right)} & {\scriptstyle \left[1\right];\left(\left[1\right]\LG\left(\overline{n};Z\right)\right)} & {\scriptstyle \left[1\right];\left(\left[1\right]\LG\overline{n}\right)}\ar[l]\ar[r] & {\scriptstyle \left[1\right];\left(\left[1\right]\LG\left(\overline{n};W\right)\right)}\\
{\scriptstyle \left[1\right];\left(\overline{n};Z\underset{0}{\bigoplus}W\right)}\ar[d]\ar[u] & {\scriptstyle \left[1\right];\left(\overline{n};Z\right)}\ar[d]\ar[u] & {\scriptstyle \left[1\right];\left(\overline{n}\right)}\ar[l]\ar[r]\ar[d]\ar[u] & {\scriptstyle \left[1\right];\left(\overline{n};W\right)}\ar[d]\ar[u]\\
{\scriptstyle \left[2\right];\left(\left(\overline{n};Z\underset{0}{\bigoplus}W\right),0\right)} & {\scriptstyle \left[2\right];\left(\left(\overline{n};Z\right),0\right)} & {\scriptstyle \left[2\right];\left(\overline{n},0\right)}\ar[l]\ar[r] & {\scriptstyle \left[2\right];\left(\left(\overline{n};W\right),0\right)}
}
}}
\]
The induction step is then completed by observing that the colimits
of the rows of the diagram on the right are the entries in the diagram
on the left. This is clear for the first, second, fourth, and fifth
rows so it suffice to prove it for the third row. Since $\left[1\right];\left(\_\right)$
preserves connected colimits, such as spans, the colimit of the span
\[
\vcenter{\vbox{\xyR{1.5pc}\xyC{1.5pc}\xymatrix{\left[1\right];\left(\left[1\right]\LG\left(\overline{n};Z\right)\right) & \left[1\right];\left(\left[1\right]\LG\overline{n}\right)\ar[l]\ar[r] & \left[1\right];\left(\left[1\right]\LG\left(\overline{n};W\right)\right)}
}}
\]
is but $\left[1\right];\left(\_\right)$ of the colimit of the span
\[
\vcenter{\vbox{\xyR{1.5pc}\xyC{1.5pc}\xymatrix{\left[1\right]\LG\left(\overline{n};Z\right) & \left[1\right]\LG\overline{n}\ar[l]\ar[r] & \left[1\right]\LG\left(\overline{n};W\right)}
}}
\]
which is to say it is $\left[1\right]\LG\left(\overline{n};Z\right)\underset{\left[1\right]\LG\overline{n}}{\oplus}\left[1\right]\LG\left(\overline{n};W\right)$
which is but $\left[1\right]\LG\left(\overline{n};Z\underset{0}{\oplus}W\right)$
by hypothesis.
\end{proof}

\subsection{\label{subsec:Comparison-to-Steiner}Steiner's theory }

We now develop enough of Steiner's theory to use its elegant description
of the Gray tensor product to prove the lax Gray shuffle decomposition
correct.

\subsubsection{Steiner Complexes and their relation to $\omega$-categories}

In \cite{Steiner1} (and nicely recovered in \cite{AraMaltsiniotis}
whose exposition we largely follow here) we find developed a treatment
of $\omega$-categories as chain complexes of abelian groups, in the
homological (positive degree) convention, together with further data
required to encode the orientation of cells as ``positivity''. Of
particular utility to us here is that the lax gray tensor product
of $\omega$-categories in this treatment is easily written in terms
of the tensor product of the underlying chain complexes.
\begin{defn}
A \textbf{directed augmented complex $\left(K,K^{*},e\right)$ }is
comprised of:
\begin{itemize}
\item a chain complex of abelian groups $K$, 
\[
\cdots\xrightarrow{d_{n+1}}K_{n}\xrightarrow{d_{n}}K_{n-1}\xrightarrow{d_{n-1}}\cdots\xrightarrow{d_{2}}K_{1}\xrightarrow{d_{1}}K_{0}
\]
\item a set of submonoids 
\[
\left\{ K_{n}^{*}\subset K_{n}\right\} _{n\in\N}
\]
 (no compatibility between these submonoids and the differentials
of $K$ is assumed); and
\item a morphism of groups $e:K_{0}\rightarrow\Z$ such that $e\circ d_{1}=0$.
\end{itemize}
\begin{rem}
The submonoids, which will define ``positivity'', encode the direction
of cells/group-elements, whereas the map $e$, commonly known as an
``augmentation'', identify the objects, as oppose to the formal
sums of objects, as we will see.
\end{rem}

A \textbf{morphism of directed} \textbf{augmented complexes
\[
a:\left(K,K^{*},e\right)\longrightarrow\left(L,L^{*},f\right)
\]
}is a morphism of augmented chain complexes $a:\left(K,e\right)\longrightarrow\left(L,f\right)$
which respects the positivity sub-monoids, i.e. $a_{n}\left(K_{n}^{*}\right)\subset L_{n}^{*}$
for each $n\in\N$. Let $\mathsf{C_{DA}}$ denote the category of
directed augmented complexes. 
\end{defn}

Steiner further defines functors
\[
\omega\mhyphen\Cat\longrightarrow\sfC_{\sf{DA}}
\]
and 
\[
\CDA\longrightarrow\StrCat
\]

\begin{defn}
\label{def:Steiner's-lambda-and-nu}Let 
\[
\lambda:\omega\mhyphen\Cat\longrightarrow\sfC_{\sf{DA}}
\]
be the functor which sends a $\omega$-category $X$, to the directed
augmented chain complex 
\[
\left(\lambda\left(X\right),\lambda^{*}\left(X^{*}\right),e_{X}\right)
\]
where:
\begin{itemize}
\item the abelian groups $\lambda\left(X\right)_{n}$ are:
\begin{itemize}
\item generated by elements $\left[x\right]$ for each $n$-cells $x:\overline{n}\longrightarrow X$; 
\item subject to the minimal relation such that, for any permissible composition
of cells $x,y:\overline{n}\longrightarrow X$, 
\[
x\bigoplus_{m}y:\overline{n}\bigoplus_{m}\overline{n}\longrightarrow X
\]
with composition cell 
\[
x\underset{m}{\star}y:\overline{n}\longrightarrow\overline{n}\bigoplus_{m}\overline{n}\longrightarrow X
\]
 we have 
\[
\left[x\underset{m}{\star}y\right]=\left[x\right]+\left[y\right]
\]
\end{itemize}
\item the differentials are generated by setting, for a generating element
$\left[x\right]$, 
\[
d\left(\left[x\right]\right)=\left[t\left(x\right)\right]-\left[s\left(x\right)\right]
\]
\item the positivity sub-monoids are the sub-monoids generated by those
generating elements $\left[x\right]$; and
\item the augmentation $e_{X}$ is the unique map $\lambda\left(X\right)_{0}\longrightarrow\Z$
which sends the generating elements $\left[x\right]\in\lambda\left(X\right)_{0}$,
corresponding to objects of $X$, to $1\in\Z$
\end{itemize}
\begin{rem}
note that in our notation the composition cell $x\underset{m}{\star}y$
is to be read left to right - not in the usual composition order -
it is the composition
\[
\vcenter{\vbox{\xyR{.0pc}\xyC{1.5pc}\xymatrix{s\left(x\right)\ar@{=>}[r]|-{x} & t\left(x\right)=s\left(y\right)\ar@{=>}[r]|-{y} & t\left(y\right)}
}}
\]
\end{rem}

The action of $\lambda$ on morphisms is precisely what one would
expect:
\begin{itemize}
\item given a functor $a:X\longrightarrow Y$, an element $\left[x\right]\in\lambda\left(X\right)_{n}$
is sent by $\lambda\left(a\right)_{n}$ to the element $\left[a\left(x\right)\right]\in\lambda\left(Y\right)_{n}$.
\end{itemize}
We define the functor 
\[
\nu:\CDA\longrightarrow\StrCat
\]
as follows. Given a directed augmented complex $\left(K,K^{*},e\right)$:
\begin{itemize}
\item the $n$-morphisms of 
\[
\nu\left(K,K^{*},e\right)
\]
 are tables of graded group elements
\[
\left(\begin{array}{cccc}
x_{0}^{0} & \cdots & x_{i-1}^{0} & x_{i}^{0}\\
x_{0}^{1} & \cdots & x_{i-1}^{1} & x_{i}^{1}
\end{array}\right)
\]
where:
\end{itemize}
\begin{enumerate}
\item $x_{k}^{\e}\in K_{k}^{*}\subset K_{k}$ for $\e=0,1$ and $0\leq k\leq i$;
\item $d\left(x_{k}^{\e}\right)=x_{k-1}^{1}-x_{k-1}^{0}$ for $\e=0,1$
and $0\leq k\leq i$;
\item $e\left(x_{0}^{\e}\right)=1$ for $\e=0,1$; and
\item $x_{i}^{1}=x_{i}^{0}$
\end{enumerate}
\begin{itemize}
\item the targets and sources of an $i$-morphism 
\[
\left(\begin{array}{cccc}
x_{0}^{0} & \cdots & x_{i-1}^{0} & x_{i}^{0}\\
x_{0}^{1} & \cdots & x_{i-1}^{1} & x_{i}^{1}
\end{array}\right)
\]
 are given:
\begin{itemize}
\item 
\[
t\left(\begin{array}{cccc}
x_{0}^{0} & \cdots & x_{i-1}^{0} & x_{i}^{0}\\
x_{0}^{1} & \cdots & x_{i-1}^{1} & x_{i}^{1}
\end{array}\right)=\left(\begin{array}{cccc}
x_{0}^{0} & \cdots & x_{i-2}^{0} & x_{i-1}^{1}\\
x_{0}^{1} & \cdots & x_{i-2}^{1} & x_{i-1}^{1}
\end{array}\right)
\]
\item 
\[
s\left(\begin{array}{cccc}
x_{0}^{0} & \cdots & x_{i-1}^{0} & x_{i}^{0}\\
x_{0}^{1} & \cdots & x_{i-1}^{1} & x_{i}^{1}
\end{array}\right)=\left(\begin{array}{cccc}
x_{0}^{0} & \cdots & x_{i-2}^{0} & x_{i-1}^{0}\\
x_{0}^{1} & \cdots & x_{i-2}^{1} & x_{i-1}^{0}
\end{array}\right)
\]
\end{itemize}
\item the identity $\left(i+1\right)$-cell for a table 
\[
\left(\begin{array}{cccc}
x_{0}^{0} & \cdots & x_{i-1}^{0} & x_{i}^{0}\\
x_{0}^{1} & \cdots & x_{i-1}^{1} & x_{i}^{1}
\end{array}\right)
\]
is the table
\[
\left(\begin{array}{ccccc}
x_{0}^{0} & \cdots & x_{i-1}^{0} & x_{i}^{0} & 0\\
x_{0}^{1} & \cdots & x_{i-1}^{1} & x_{i}^{1} & 0
\end{array}\right)
\]
\item the $j\leq i$-composition of $i$-cells 
\[
\left(\begin{array}{cccc}
x_{0}^{0} & \cdots & x_{i-1}^{0} & x_{i}^{0}\\
x_{0}^{1} & \cdots & x_{i-1}^{1} & x_{i}^{1}
\end{array}\right)\mathrm{\ and\ }\left(\begin{array}{cccc}
y_{0}^{0} & \cdots & y_{i-1}^{0} & y_{i}^{0}\\
y_{0}^{1} & \cdots & y_{i-1}^{1} & y_{i}^{1}
\end{array}\right)
\]
is the table
\[
\left(\begin{array}{ccccccc}
x_{0}^{0} & \cdots & x_{j-1}^{0} & x_{j}^{0} & x_{j+1}^{0}+y_{j+1}^{0} & \cdots & x_{i}^{0}+y_{i}^{0}\\
y_{0}^{1} & \cdots & y_{j-1}^{1} & y_{j}^{1} & x_{j+1}^{1}+y_{j+1}^{1} & \cdots & x_{i}^{1}+y_{i}^{1}
\end{array}\right)
\]
\end{itemize}
\end{defn}

\begin{thm}
\label{thm:steiner's-adjunction}(Theorem 2.11 of \cite{Steiner1})
The functor 
\[
\lambda:\StrCat\longrightarrow\CDA
\]
 is left adjoint to the functor 
\[
\StrCat\longleftarrow\CDA:\nu
\]
\end{thm}

\begin{example}
\label{exa:steiner-ness of globes}Consider the $n$-globe $\overline{n}$
as an $\omega$-category. The directed augmented complex 
\[
\left(\lambda\left(\overline{n}\right),\lambda^{*}\left(\overline{n}\right),e_{\overline{n}}\right)
\]
is comprised of:
\begin{itemize}
\item the chain complex
\[
\vcenter{\vbox{\xyR{0pc}\xyC{2pc}\xymatrix{\cdots\ar[r] & 0\ar[r] & \Z\ar[r] & \Z\oplus\Z\ar[r]^{{\scriptscriptstyle }} & \Z\oplus\Z\ar[r]^{{\scriptscriptstyle }} & \cdots\ar[r]^{{\scriptscriptstyle }} & \Z\oplus\Z\\
 & n+1 & n & n-1 & n-2 &  & 0
}
}}
\]
with the constituent morphisms
\[
\xyR{0pc}\xyC{2pc}\xymatrix{{\scriptscriptstyle \left[\begin{array}{cc}
1 & -1\end{array}\right]}:\Z\ar[r] & \Z\oplus\Z}
\]
and 
\[
\xyR{0pc}\xyC{2pc}{\scriptscriptstyle \left[\begin{array}{cc}
1 & -1\\
1 & -1
\end{array}\right]}:\xymatrix{\Z\oplus\Z\ar[r] & \Z\oplus\Z}
\]
\item the obvious sub-monoids $\N\subset\Z$ and $\N\oplus\N\subset\Z\oplus\Z$;
and
\item the augmentation 
\[
\vcenter{\vbox{\xyR{0pc}\xyC{2pc}\xymatrix{{\scriptscriptstyle \left[\begin{array}{c}
1\\
1
\end{array}\right]}:\Z\oplus\Z\ar[r] & \Z\\
0 & -1
}
}}
\]
\end{itemize}
A particularly nice fact about $\overline{n}$ is that $\overline{n}=\nu\circ\lambda\left(\overline{n}\right)$.
Indeed, consider that a non-degenerate $n$-cell of $\nu\circ\lambda\left(\overline{n}\right)$
is a table
\[
\left(\begin{array}{cccc}
\left(x_{0}^{0},y_{0}^{0}\right) & \cdots & \left(x_{n-1}^{0},y_{n-1}^{0}\right) & z_{n}\\
\left(x_{0}^{1},y_{0}^{1}\right) & \cdots & \left(x_{n-1}^{1},y_{n-1}^{1}\right) & z_{n}
\end{array}\right)
\]
where (following Definition \ref{def:Steiner's-lambda-and-nu}):
\begin{enumerate}
\item where:
\begin{enumerate}
\item $z_{n}^{0}=z_{n}^{1}\geq0$ and
\item $\left(x_{n-1}^{0},y_{n-1}^{0}\right),\left(x_{n-1}^{1},y_{n-1}^{1}\right),\dots,\left(x_{0}^{0},y_{0}^{0}\right),\left(x_{0}^{1},y_{0}^{1}\right)\in\N^{2}$
\end{enumerate}
\item \label{enu:differential thing}$d\left(z_{n}\right)=\left(x_{n-1}^{1},y_{n-1}^{1}\right)-\left(x_{n-1}^{0},y_{n-1}^{0}\right)$
and $d\left(x_{k}^{\e},y_{k}^{\e}\right)=x_{k-1}^{1}-x_{k-1}^{0}$
for $\e=0,1$ and $0\leq k<n$;
\item \label{enu:positive-1-0-things}$e\left(x_{0}^{\e},y_{0}^{\e}\right)=1$
for $\e=0,1$.
\end{enumerate}
In light of (\ref{enu:positive-1-0-things}) we find that $\left(x_{0}^{\e},y_{0}^{\e}\right)$
is either $\left(1,0\right)$ or $\left(0,1\right)$ for $\e=0,1$,
and then in light of (\ref{enu:differential thing}) and the hypothesis
on non-degeneracy we find that $\left(x_{0}^{0},y_{0}^{0}\right)=\left(1,0\right)$
and $\left(x_{0}^{1},y_{0}^{1}\right)=\left(0,1\right)$. Likewise
from (\ref{enu:differential thing}), for $1\leq k<n$ we may deduce
that $\left(x_{k}^{e},y_{k}^{\e}\right)$ is either $\left(0,1\right)$
or $\left(1,0\right)$, and if the table is to correspond to a non-degenerate
cell, $\left(x_{k}^{1},y_{k}^{1}\right)=\left(0,1\right)$ and $\left(x_{k}^{0},y_{k}^{0}\right)=\left(1,0\right)$.
Lastly, since $d_{n}$ is the map 
\[
\left[\begin{array}{cc}
1 & -1\end{array}\right]:\Z\rightarrow\Z\oplus\Z
\]
it follows that $z=1$. To wit, the unique non-degenerate $n$-cell
of $\nu\circ\lambda\left(\overline{n}\right)$ is the table
\[
\left(\begin{array}{cccc}
\left(1,0\right) & \cdots & \left(1,0\right) & 1\\
\left(0,1\right) & \cdots & \left(0,1\right) & 1
\end{array}\right)
\]
Similar arguments demonstrate that the only two non-degenerate $k$-cells,
for $0<k<n$ are
\[
\left(\begin{array}{cccc}
\left(1,0\right) & \cdots & \left(1,0\right) & \left(1,0\right)\\
\left(0,1\right) & \cdots & \left(0,1\right) & \left(1,0\right)
\end{array}\right)
\]
and 
\[
\left(\begin{array}{cccc}
\left(1,0\right) & \cdots & \left(1,0\right) & \left(0,1\right)\\
\left(0,1\right) & \cdots & \left(0,1\right) & \left(0,1\right)
\end{array}\right)
\]
and that there are two $0$-cells, 
\[
\left(\begin{array}{c}
\left(1,0\right)\\
\left(1,0\right)
\end{array}\right)
\]
and 
\[
\left(\begin{array}{c}
\left(0,1\right)\\
\left(0,1\right)
\end{array}\right)
\]
\end{example}

\subsubsection{(Strong) Steiner complexes, (strong) Steiner $\omega$-categories,
Berger's Wreath product, and the tensor product}

The co-unit $\eta:\lambda\circ\nu\Longrightarrow\id$ of the adjunction
$\lambda\dashv\nu$ (Theorem \ref{thm:steiner's-adjunction}) is not
in general invertible, i.e. $\nu$ is not full-and-faithful. We can
however identify a sub-category of $\sf{St}\CDA\longrightarrow\CDA$,
the so-called \emph{Steiner complexes, }on which $\nu$ restricts
to a full-and-faithful functor.
\begin{thm}
(Steiner\footnote{as cited in \cite{AraMaltsiniotis}, see Theorem 5.6 of \cite{Steiner1}
and Paragraph 2.15 of \cite{AraMaltsiniotis}}) For all Steiner complexes $K$, the co-unit 
\[
\eta_{K}:\lambda\circ\nu\left(K\right)\longrightarrow K
\]
is an isomorphism. In particular, $\nu$ restricted to the full subcategory
$\mathsf{StC}_{\sf{DA}}$ of $\CDA$, subtended by the Steiner complexes,
is full-and-faithful.

More, on the subcategory of \emph{strong Steiner complexes}, $\sf{StrSt}\CDA$,
the adjunction $\lambda\dashv\nu$ restricts to an equivalence of
categories 
\[
\sf{StrSt}\CDA\liso\sf{StrSt}\StrCat
\]
 between the subcategory $\sf{StrSt}\CDA\longrightarrow\sf{St}\CDA\longrightarrow\CDA$
of strong Steiner complexes and the subcategory
\[
\sf{StrSt}\StrCat\longrightarrow\sf{St}\StrCat\longrightarrow\StrCat
\]
 of strong Steiner $\omega$-categories.
\end{thm}

Furthermore the tensor product of chain complexes, together with the
obvious choice of positivity sub-monoids, defines a bi-closed monoidal
product on $\sf{StrSt}\CDA$ which, under $\nu$, passes to the lax
Gray tensor product of $\StrCat$. Indeed the adjunction
\[
\lambda\dashv\nu:\sf{StrSt}\CDA\liso\sf{StrSt}\StrCat
\]
 is a monoidal equivalence of monoidal categories.
\begin{defn}
\label{def:tensor-product-of-directed-augmented-complexes}Let $\left(K,K^{*},e\right)$
and $\left(L,L^{*},f\right)$ be directed augmented complexes. We
recall that $K\otimes L$ is the chain complex with 
\[
\left(K\otimes L\right)_{n}=\bigoplus_{i+j=n}K_{i}\otimes L_{j}
\]
and differentials
\[
d_{n}=\bigoplus_{i+j=n}\left(d_{i}^{K}\otimes\id_{L}+\left(-1\right)^{i}\left(\id_{K}\otimes d_{j}^{L}\right)\right)
\]
We extend this to augmented chain complexes by picking the map $e\otimes f:K_{0}\otimes L_{0}\longrightarrow\Z$
as the augmentation.
\end{defn}

The promised result regarding the lax Gray tensor product follows.
\begin{thm}
(Steiner - See A.13 and A.14 of \cite{AraMaltsiniotis}) For strong
Steiner complexes $K$ and $L$ and $\omega$-categories $\sf X$
we have isomorphisms
\[
\StrCat\left(\nu\left(K\right),\mathsf{\mathsf{Oplax^{\omega}}}\left(\nu\left(L\right),\sf X\right)\right)\losi\StrCat\left(\nu\left(K\otimes L\right),\sf X\right)
\]
and 
\[
\StrCat\left(\nu\left(K\otimes L\right),\sf X\right)\liso\StrCat\left(\nu\left(L\right),\mathsf{Lax}^{\omega}\left(\nu\left(K\right),\sf X\right)\right)
\]
\end{thm}

We leave the definitions of Steiner complexes and strong Steiner complexes
to Appendix \ref{sec:Steiner-Complexes-and-Steiner =00005Comega-categories}
as they are rather long and involved - a Steiner complex is a directed
augmented complex which admits a unital loop-free basis, and a strong
Steiner complex is a Steiner complex whose unital loop-free basis
is strongly so. Fortunate for us however is the fact that we require
the theory only to prove that our formula for $\left[1\right]\LG\overline{n}$
is correct and, as it happens, all the objects of $\Theta$ are strong
Steiner $\omega$-categories. Indeed, in \cite{Steiner2} we find
the following far stronger claim, which relates Berger's wreath product
to Steiner's treatment of $\omega$-categories.
\begin{thm}
\label{thm:(Wreath-Product-of-Steiner}(Theorem 5.6 of \cite{Steiner2})
The functor 
\[
\t\int\CDA\longrightarrow\CDA
\]
 restricts to a full-and-faithful functor 
\[
\t\int\sf{StrSt}\CDA\longrightarrow\sf{StrSt}\CDA
\]
Equivalently 
\[
\t\int\sf{StrSt}\StrCat\longrightarrow\sf{StrSt\StrCat}
\]
is full-and-faithful.
\end{thm}

\begin{rem}
\label{rem:Theta-is-strong-steiner}To clarify precisely how this
result provides the strong Steiner-ness of all of the objects in $\Theta$
, consider that $\overline{0}$ is strong Steiner, so $\left[n\right];\left(\overline{0},\overline{0},\dots,\overline{0}\right)$
is strong Steiner by the theorem above for every $n\in\N$, so all
of $\t$ is strong Steiner. Iterating this it follows that all of
$\Theta$ is strong Steiner.

\pagebreak{}
\end{rem}

\subsection{The lax shuffle decomposition computes the Gray cylinder}
\begin{prop}
For all $n\geq0$, the diagram 
\[
\vcenter{\vbox{\xyR{1.5pc}\xyC{1.5pc}\xymatrix{ & \left[2\right];\left(\overline{n},\overline{0}\right)\ar[ddr]\\
\left[1\right];\left(\overline{n}\right)\ar@{-->}[ur]\ar@{-->}[dr]\\
 & \left[1\right];\left(\left[1\right]\otimes\overline{n}\right)\ar[r] & \nu\left(\lambda\left(\overline{1}\right)\otimes\lambda\left(\overline{n+1}\right)\right)\\
\left[1\right];\left(\overline{n}\right)\ar@{-->}[ur]\ar@{-->}[dr]\\
 & \left[2\right];\left(\overline{0},\overline{n}\right)\ar[uur]
}
}}
\]
is a gluing diagram, meaning:
\begin{itemize}
\item the solid morphisms into $\lambda\left(\overline{1}\right)\otimes\lambda\left(\overline{n+1}\right)$
are monomorphisms;
\item the solid morphisms cover $\lambda\left(\overline{1}\right)\otimes\lambda\left(\overline{n+1}\right)$;
and
\item the two squares are pullbacks.
\end{itemize}
Moreover, as a consequence, the induced map
\[
\clim\left\{ \vcenter{\vbox{\xyR{.5pc}\xyC{.5pc}\xymatrix{ & \left[2\right];\left(\overline{n},\overline{0}\right)\\
\left[1\right];\left(\overline{n}\right)\ar[ur]\ar[dr]\\
 & \left[1\right];\left(\left[1\right]\otimes\overline{n}\right)\\
\left[1\right];\left(\overline{n}\right)\ar[ur]\ar[dr]\\
 & \left[2\right];\left(\overline{0},\overline{n}\right)
}
}}\right\} \liso\nu\left(\lambda\left(\overline{1}\right)\otimes\lambda\left(\overline{n+1}\right)\right)
\]
is an isomorphism (as indicated).
\end{prop}

\begin{proof}
The case for $n\geq3$ is fully general, the cases for $n=0,1,2$
require minimal modifications. As such, we will show only that 
\[
\vcenter{\vbox{\xyR{1.5pc}\xyC{1.5pc}\xymatrix{ & \left[2\right];\left(\overline{n},\overline{0}\right)\ar[ddr]\\
\left[1\right];\left(\overline{n}\right)\ar@{-->}[ur]\ar@{-->}[dr]\\
 & \left[1\right];\left(\left[1\right]\otimes\overline{n}\right)\ar[r] & \nu\left(\lambda\left(\overline{1}\right)\otimes\lambda\left(\overline{n+1}\right)\right)\\
\left[1\right];\left(\overline{n}\right)\ar@{-->}[ur]\ar@{-->}[dr]\\
 & \left[2\right];\left(\overline{0},\overline{n}\right)\ar[uur]
}
}}
\]
is a gluing diagram for all $n\geq3$.

Observe first that $\left[2\right];\left(\overline{n},0\right)$,
$\left[1\right];\left(\left[1\right]\otimes\overline{n}\right)$,
and $\left[2\right];\left(0,\overline{n}\right)$ are all strong Steiner
$\omega$-categories; the first and last as $\Theta$ is a full subcategory
of the category of strong Steiner categories (see Remark \ref{rem:Theta-is-strong-steiner})
and the middle one since $\left[1\right];\left(\_\right)$ preserves
the strong Steiner-ness of $\omega$-categories (as with Remark \ref{rem:Theta-is-strong-steiner},
this a consequence of Theorem \ref{thm:(Wreath-Product-of-Steiner}).
Thus to prove the diagram above to be a gluing diagram is to prove
that
\[
\vcenter{\vbox{\xyR{1.5pc}\xyC{1.5pc}\xymatrix{ & \nu\left(\lambda\left(\left[2\right];\left(\overline{n},\overline{0}\right)\right)\right)\ar[ddr]\\
\nu\left(\lambda\left(\left[1\right];\left(\overline{n}\right)\right)\right)\ar@{-->}[ur]\ar@{-->}[dr]\\
 & \nu\left(\lambda\left(\left[1\right];\left(\left[1\right]\otimes\overline{n}\right)\right)\right)\ar[r] & \nu\left(\lambda\left(\overline{1}\right)\otimes\lambda\left(\overline{n+1}\right)\right)\\
\nu\left(\lambda\left(\left[1\right];\left(\overline{n}\right)\right)\right)\ar@{-->}[ur]\ar@{-->}[dr]\\
 & \nu\left(\lambda\left(\left[2\right];\left(\overline{0},\overline{n}\right)\right)\right)\ar[uur]
}
}}
\]
is a gluing diagram. What's more, since the adjunction $\nu\dashv\lambda$
restricts to an adjoint equivalence of categories on the strong Steiner
subcategories, the diagram above is a gluing diagram if and only if
\[
\vcenter{\vbox{\xyR{1.5pc}\xyC{1.5pc}\xymatrix{ & \lambda\left(\left[2\right];\left(\overline{n},\overline{0}\right)\right)\ar[ddr]\\
\lambda\left(\left[1\right];\left(\overline{n}\right)\right)\ar@{-->}[ur]\ar@{-->}[dr]\\
 & \lambda\left(\left[1\right];\left(\left[1\right]\otimes\overline{n}\right)\right)\ar[r] & \lambda\left(\overline{1}\right)\otimes\lambda\left(\overline{n+1}\right)\\
\lambda\left(\left[1\right];\left(\overline{n}\right)\right)\ar@{-->}[ur]\ar@{-->}[dr]\\
 & \lambda\left(\left[2\right];\left(\overline{0},\overline{n}\right)\right)\ar[uur]
}
}}
\]
is gluing diagram. It is this claim which we will now prove.

Let the copy of $\lambda\left(\overline{1}\right)$ in $\lambda\left(\overline{1}\right)\otimes\lambda\left(\overline{n+1}\right)$
be the obvious directed augmented complex on the chain complex
\[
\vcenter{\vbox{\xyR{0pc}\xyC{1.5pc}\xymatrix{\left\langle h\right\rangle \ar[r] & \left\langle \ell\right\rangle \oplus\left\langle r\right\rangle \\
h\ar@{|->}[r] & -\ell+r
}
}}
\]
and let the copy of $\lambda\left(\overline{n+1}\right)=\lambda\left(\left[1\right];\left(\overline{n}\right)\right)$
be the obvious directed augmented complex on the chain complex
\[
\vcenter{\vbox{\xyR{0pc}\xyC{1.5pc}\xymatrix{\left\langle v_{n+1}\right\rangle \ar[r] & \left\langle b_{n}\right\rangle \oplus\left\langle t_{n}\right\rangle \ar[r] & \left\langle b_{n-1}\right\rangle \oplus\left\langle t_{n-1}\right\rangle \ar[r] & \cdots\ar[r] & \left\langle b_{1}\right\rangle \oplus\left\langle t_{1}\right\rangle \ar[r] & \left\langle b_{0}\right\rangle \oplus\left\langle t_{0}\right\rangle \\
v_{n+1}\ar@{|->}[r] & -b_{n}+t_{n}\\
 & b_{n},t_{n}\ar@{|->}[r] & -b_{n-1}+t_{n-1}\\
 &  &  & \ddots\\
 &  &  &  & b_{1},t_{1}\ar@{|->}[r] & -b_{0}+t_{0}
}
}}
\]
Then we find that $\lambda\left(\overline{1}\right)\otimes\lambda\left(\left[1\right];\overline{n}\right)$
is the obvious directed augmented complex on the chain complex
\[
\xyR{0pc}\xyC{0pc}\xymatrix{{\scriptstyle n+2} &  & \left\langle h\otimes v_{n+1}\right\rangle \ar[dd]\save"1,3"."1,3"*[F]\frm{}\restore\\
 &  & \ \\
 & \save"3,2"."5,4"*[F]\frm{}\restore\phantom{\left\langle \ell\otimes v_{n+1}\right\rangle } & \left\langle h\otimes t_{n}\right\rangle \\
{\scriptstyle n+1} & \left\langle \ell\otimes v_{n+1}\right\rangle \ar@{}[rr]|-{\oplus} &  & \left\langle r\otimes v_{n+1}\right\rangle \\
 &  & \left\langle h\otimes b_{n}\right\rangle \ar[dd] & \phantom{\left\langle \ell\otimes v_{n+1}\right\rangle }\\
 & \save"7,2"."8,4"*[F]\frm{}\restore & \ \\
\ar@{}[d]|-{n} & \left\langle \ell\otimes t_{n}\right\rangle \ar@{}[r]|-{\oplus}\ar@{}[d]|-{\oplus} & \left\langle h\otimes t_{n-1}\right\rangle \ar@{}[r]|-{\oplus}\ar@{}[d]|-{\oplus} & \left\langle r\otimes t_{n}\right\rangle \ar@{}[d]|-{\oplus}\\
 & \left\langle \ell\otimes b_{n}\right\rangle \ar@{}[r]|-{\oplus} & \left\langle h\otimes b_{n-1}\right\rangle \ar@{}[r]|-{\oplus}\ar[dd] & \left\langle r\otimes b_{n}\right\rangle \\
 & \save"10,2"."11,4"*[F]\frm{}\restore & \ \\
\ar@{}[d]|-{n-1} & \left\langle \ell\otimes t_{n-1}\right\rangle \ar@{}[r]|-{\oplus}\ar@{}[d]|-{\oplus} & \left\langle h\otimes t_{n-2}\right\rangle \ar@{}[r]|-{\oplus}\ar@{}[d]|-{\oplus} & \left\langle r\otimes t_{n-1}\right\rangle \ar@{}[d]|-{\oplus}\\
 & \left\langle \ell\otimes b_{n-1}\right\rangle \ar@{}[r]|-{\oplus} & \left\langle h\otimes b_{n-2}\right\rangle \ar@{}[r]|-{\oplus}\ar[ddd]|-{\vdots} & \left\langle r\otimes b_{n-1}\right\rangle \\
 &  & \ \\
 & \save"14,2"."15,4"*[F]\frm{}\restore & \ \\
\ar@{}[d]|-{2} & \left\langle \ell\otimes t_{2}\right\rangle \ar@{}[r]|-{\oplus}\ar@{}[d]|-{\oplus} & \left\langle h\otimes t_{1}\right\rangle \ar@{}[r]|-{\oplus}\ar@{}[d]|-{\oplus} & \left\langle r\otimes t_{2}\right\rangle \ar@{}[d]|-{\oplus}\\
 & \left\langle \ell\otimes b_{2}\right\rangle \ar@{}[r]|-{\oplus} & \left\langle h\otimes b_{1}\right\rangle \ar@{}[r]|-{\oplus}\ar[dd]\ar[ddd]|-{\vdots} & \left\langle r\otimes b_{2}\right\rangle \\
 &  & \ \\
 & \save"18,2"."19,4"*[F]\frm{}\restore & \ \\
\ar@{}[d]|-{1} & \left\langle \ell\otimes t_{1}\right\rangle \ar@{}[r]|-{\oplus}\ar@{}[d]|-{\oplus} & \left\langle h\otimes t_{0}\right\rangle \ar@{}[r]|-{\oplus}\ar@{}[d]|-{\oplus} & \left\langle r\otimes t_{1}\right\rangle \ar@{}[d]|-{\oplus}\\
 & \left\langle \ell\otimes b_{1}\right\rangle \ar@{}[r]|-{\oplus} & \left\langle h\otimes b_{0}\right\rangle \ar@{}[r]|-{\oplus}\ar[dd] & \left\langle r\otimes b_{1}\right\rangle \\
 & \save"21,2"."22,4"*[F]\frm{}\restore & \ \\
\ar@{}[d]|-{0} & \left\langle \ell\otimes t_{0}\right\rangle \ar@{}[rr]|-{\oplus}\ar@{}[d]|-{\oplus} & \  & \left\langle r\otimes t_{0}\right\rangle \ar@{}[d]|-{\oplus}\\
 & \left\langle \ell\otimes b_{0}\right\rangle \ar@{}[rr]|-{\oplus} &  & \left\langle r\otimes b_{0}\right\rangle 
}
\]

Using this description we find that $\lambda\left(\left[1\right];\left(\left[1\right]\otimes\overline{n}\right)\right)$
is the obvious directed augmented complex on 
\[
\xyR{0pc}\xyC{0pc}\xymatrix{{\scriptstyle m+2} & \save"1,3"."1,3"*[F]\frm{}\restore & \left\langle h\otimes v_{m}^{\p}\right\rangle \ar[dd]\\
 &  & \ \\
 & \save"3,2"."5,4"*[F]\frm{}\restore\phantom{\left\langle \ell\otimes v_{n+1}\right\rangle } & \left\langle h^{\p}\otimes t_{m-1}^{\p}\right\rangle \\
{\scriptstyle m+1} & \left\langle \ell^{\p}\otimes v_{m}^{\p}\right\rangle \ar@{}[rr]|-{\oplus} &  & \left\langle r^{\p}\otimes v_{m}^{\p}\right\rangle \\
 &  & \left\langle h^{\p}\otimes b_{m-1}^{\p}\right\rangle \ar[dd] & \phantom{\left\langle \ell\otimes v_{n+1}\right\rangle }\\
 & \save"7,2"."8,4"*[F]\frm{}\restore & \ \\
\ar@{}[d]|-{m} & \left\langle \ell^{\p}\otimes t_{m-1}^{\p}\right\rangle \ar@{}[r]|-{\oplus}\ar@{}[d]|-{\oplus} & \left\langle h^{\p}\otimes t_{m-2}^{\p}\right\rangle \ar@{}[r]|-{\oplus}\ar@{}[d]|-{\oplus} & \left\langle r^{\p}\otimes t_{m-1}^{\p}\right\rangle \ar@{}[d]|-{\oplus}\\
 & \left\langle \ell^{\p}\otimes b_{m-1}\right\rangle \ar@{}[r]|-{\oplus} & \left\langle h^{\p}\otimes b_{m-2}^{\p}\right\rangle \ar@{}[r]|-{\oplus}\ar[dd] & \left\langle r^{\p}\otimes b_{m-1}^{\p}\right\rangle \\
 & \save"10,2"."11,4"*[F]\frm{}\restore & \ \\
\ar@{}[d]|-{m-1} & \left\langle \ell^{\p}\otimes t_{m-2}^{\p}\right\rangle \ar@{}[r]|-{\oplus}\ar@{}[d]|-{\oplus} & \left\langle h^{\p}\otimes t_{m-3}^{\p}\right\rangle \ar@{}[r]|-{\oplus}\ar@{}[d]|-{\oplus} & \left\langle r^{\p}\otimes t_{m-2}^{\p}\right\rangle \ar@{}[d]|-{\oplus}\\
 & \left\langle \ell^{\p}\otimes b_{m-2}^{\p}\right\rangle \ar@{}[r]|-{\oplus} & \left\langle h^{\p}\otimes b_{m-3}^{\p}\right\rangle \ar@{}[r]|-{\oplus}\ar[ddd]|-{\vdots} & \left\langle r^{\p}\otimes b_{m-2}^{\p}\right\rangle \\
 &  & \ \\
 & \save"14,2"."15,4"*[F]\frm{}\restore & \ \\
\ar@{}[d]|-{1} & \left\langle \ell^{\p}\otimes t_{0}^{\p}\right\rangle \ar@{}[rr]|-{\oplus}\ar@{}[d]|-{\oplus} &  & \left\langle r^{\p}\otimes t_{0}^{\p}\right\rangle \ar@{}[d]|-{\oplus}\\
 & \left\langle \ell^{\p}\otimes b_{0}^{\p}\right\rangle \ar@{}[rr]|-{\oplus} & \ \ar[dd] & \left\langle r^{\p}\otimes b_{0}^{\p}\right\rangle \\
 & \save"17,2"."17,4"*[F]\frm{}\restore & \ \\
{\scriptstyle 0} & \left\langle e\right\rangle \ar@{}[rr]|-{\oplus} & \  & \left\langle s\right\rangle 
}
\]
and we identify the map $\lambda\left(\left[1\right];\left(\left[1\right]\otimes\overline{n}\right)\right)\longrightarrow\lambda\left(\overline{1}\right)\otimes\lambda\left(\left[1\right];\overline{n}\right)$
with the obvious directed augmented sub-complex on
\[
\xyR{0pc}\xyC{0pc}\xymatrix{{\scriptstyle n+2} &  & \left\langle h\otimes v_{n+1}\right\rangle \save"1,3"."1,3"*[F]\frm{}\restore\ar[dd]\\
 &  & \ \\
 & \save"3,2"."5,4"*[F]\frm{}\restore\phantom{\left\langle \ell\otimes v_{n+1}\right\rangle } & \left\langle h\otimes t_{n}\right\rangle \\
{\scriptstyle n+1} & \left\langle \ell\otimes v_{n+1}\right\rangle \ar@{}[rr]|-{\oplus} &  & \left\langle r\otimes v_{n+1}\right\rangle \\
 &  & \left\langle h\otimes b_{n}\right\rangle \ar[dd] & \phantom{\left\langle \ell\otimes v_{n+1}\right\rangle }\\
 & \save"7,2"."8,4"*[F]\frm{}\restore & \ \\
\ar@{}[d]|-{n} & \left\langle \ell\otimes t_{n}\right\rangle \ar@{}[r]|-{\oplus}\ar@{}[d]|-{\oplus} & \left\langle h\otimes t_{n-1}\right\rangle \ar@{}[r]|-{\oplus}\ar@{}[d]|-{\oplus} & \left\langle r\otimes t_{n}\right\rangle \ar@{}[d]|-{\oplus}\\
 & \left\langle \ell\otimes b_{n}\right\rangle \ar@{}[r]|-{\oplus} & \left\langle h\otimes b_{n-1}\right\rangle \ar@{}[r]|-{\oplus}\ar[dd] & \left\langle r\otimes b_{n}\right\rangle \\
 & \save"10,2"."11,4"*[F]\frm{}\restore & \ \\
\ar@{}[d]|-{n-1} & \left\langle \ell\otimes t_{n-1}\right\rangle \ar@{}[r]|-{\oplus}\ar@{}[d]|-{\oplus} & \left\langle h\otimes t_{n-2}\right\rangle \ar@{}[r]|-{\oplus}\ar@{}[d]|-{\oplus} & \left\langle r\otimes t_{n-1}\right\rangle \ar@{}[d]|-{\oplus}\\
 & \left\langle \ell\otimes b_{n-1}\right\rangle \ar@{}[r]|-{\oplus} & \left\langle h\otimes b_{n-2}\right\rangle \ar@{}[r]|-{\oplus}\ar[ddd]|-{\vdots} & \left\langle r\otimes b_{n-1}\right\rangle \\
 &  & \ \\
 & \save"14,2"."15,4"*[F]\frm{}\restore & \ \\
\ar@{}[d]|-{2} & \left\langle \ell\otimes t_{2}\right\rangle \ar@{}[r]|-{\oplus}\ar@{}[d]|-{\oplus} & \left\langle h\otimes t_{1}\right\rangle \ar@{}[r]|-{\oplus}\ar@{}[d]|-{\oplus} & \left\langle r\otimes t_{2}\right\rangle \ar@{}[d]|-{\oplus}\\
 & \left\langle \ell\otimes b_{2}\right\rangle \ar@{}[r]|-{\oplus} & \left\langle h\otimes b_{1}\right\rangle \ar@{}[r]|-{\oplus}\ar[dd]\ar[dd] & \left\langle r\otimes b_{2}\right\rangle \\
 & \save"17,2"."18,4"*[F]\frm{}\restore & \ \\
\ar@{}[d]|-{1} & \left\langle \ell\otimes t_{1}+h\otimes t_{0}\right\rangle \ar@{}[rr]|-{\oplus}\ar@{}[d]|-{\oplus} & \  & \left\langle h\otimes b_{0}+r\otimes t_{1}\right\rangle \ar@{}[d]|-{\oplus}\\
 & \left\langle \ell\otimes b_{1}+h\otimes t_{0}\right\rangle \ar@{}[rr]|-{\oplus} & \ \ar[ddd] & \left\langle h\otimes b_{0}+r\otimes b_{1}\right\rangle \\
 &  & \ \\
 & \save"21,2"."21,4"*[F]\frm{}\restore & \ \\
0 & \left\langle \ell\otimes b_{0}\right\rangle \ar@{}[rr]|-{\oplus} & \  & \left\langle r\otimes t_{0}\right\rangle 
}
\]
We may similarly identify the maps $\lambda\left(\left[2\right];\left(\overline{n},0\right)\right)\longrightarrow\lambda\left(\overline{1}\right)\otimes\lambda\left(\overline{n+1}\right)$
and $\lambda\left(\left[2\right];\left(0,\overline{n}\right)\right)\longrightarrow\lambda\left(\overline{1}\right)\otimes\lambda\left(\overline{n+1}\right)$
with the obvious directed augmented sub-complexes
\[
\xyR{0pc}\xyC{0pc}\xymatrix{{\scriptstyle n+2} & \save"1,3"."1,3"*[F]\frm{}\restore & \lightgray{\left\langle h\otimes v_{n+1}\right\rangle }\ar[dd] &  &  & \save"1,7"."1,7"*[F]\frm{}\restore & \lightgray{\left\langle h\otimes v_{n+1}\right\rangle }\ar[dd]\\
 &  & \  &  &  &  & \ \\
 & \save"3,2"."5,4"*[F]\frm{}\restore\phantom{\left\langle \ell\otimes v_{n+1}\right\rangle } & \lightgray{\left\langle h\otimes t_{n}\right\rangle } &  &  & \save"3,6"."5,8"*[F]\frm{}\restore\phantom{\left\langle \ell\otimes v_{n+1}\right\rangle } & \lightgray{\left\langle h\otimes t_{n}\right\rangle }\\
{\scriptstyle n+1} & \left\langle \ell\otimes v_{n+1}\right\rangle \ar@{}[rr]|-{\oplus} &  & \lightgray{\left\langle r\otimes v_{n+1}\right\rangle } &  & \lightgray{\left\langle \ell\otimes v_{n+1}\right\rangle }\ar@{}[rr]|-{\oplus} &  & \left\langle r\otimes v_{n+1}\right\rangle \\
 &  & \lightgray{\left\langle h\otimes b_{n}\right\rangle }\ar[dd] & \phantom{\left\langle \ell\otimes v_{n+1}\right\rangle } &  &  & \lightgray{\left\langle h\otimes b_{n}\right\rangle }\ar[dd] & \phantom{\left\langle \ell\otimes v_{n+1}\right\rangle }\\
 & \save"7,2"."8,4"*[F]\frm{}\restore & \  &  &  & \save"7,6"."8,8"*[F]\frm{}\restore & \ \\
\ar@{}[d]|-{n} & \left\langle \ell\otimes t_{n}\right\rangle \ar@{}[r]|-{\oplus}\ar@{}[d]|-{\oplus} & \lightgray{\left\langle h\otimes t_{n-1}\right\rangle }\ar@{}[r]|-{\oplus}\ar@{}[d]|-{\oplus} & \lightgray{\left\langle r\otimes t_{n}\right\rangle }\ar@{}[d]|-{\oplus} &  & \lightgray{\left\langle \ell\otimes t_{n}\right\rangle }\ar@{}[r]|-{\oplus}\ar@{}[d]|-{\oplus} & \lightgray{\left\langle h\otimes t_{n-1}\right\rangle }\ar@{}[r]|-{\oplus}\ar@{}[d]|-{\oplus} & \left\langle r\otimes t_{n}\right\rangle \ar@{}[d]|-{\oplus}\\
 & \left\langle \ell\otimes b_{n}\right\rangle \ar@{}[r]|-{\oplus} & \lightgray{\left\langle h\otimes b_{n-1}\right\rangle }\ar@{}[r]|-{\oplus}\ar[dd] & \lightgray{\left\langle r\otimes b_{n}\right\rangle } &  & \lightgray{\left\langle \ell\otimes b_{n}\right\rangle }\ar@{}[r]|-{\oplus} & \lightgray{\left\langle h\otimes b_{n-1}\right\rangle }\ar@{}[r]|-{\oplus}\ar[dd] & \left\langle r\otimes b_{n}\right\rangle \\
 & \save"10,2"."11,4"*[F]\frm{}\restore & \  &  &  & \save"10,6"."11,8"*[F]\frm{}\restore & \ \\
\ar@{}[d]|-{n-1} & \left\langle \ell\otimes t_{n-1}\right\rangle \ar@{}[r]|-{\oplus}\ar@{}[d]|-{\oplus} & \lightgray{\left\langle h\otimes t_{n-2}\right\rangle }\ar@{}[r]|-{\oplus}\ar@{}[d]|-{\oplus} & \lightgray{\left\langle r\otimes t_{n-1}\right\rangle }\ar@{}[d]|-{\oplus} &  & \lightgray{\left\langle \ell\otimes t_{n-1}\right\rangle }\ar@{}[r]|-{\oplus}\ar@{}[d]|-{\oplus} & \lightgray{\left\langle h\otimes t_{n-2}\right\rangle }\ar@{}[r]|-{\oplus}\ar@{}[d]|-{\oplus} & \left\langle r\otimes t_{n-1}\right\rangle \ar@{}[d]|-{\oplus}\\
 & \left\langle \ell\otimes b_{n-1}\right\rangle \ar@{}[r]|-{\oplus} & \lightgray{\left\langle h\otimes b_{n-2}\right\rangle }\ar@{}[r]|-{\oplus}\ar[ddd]|-{\vdots} & \lightgray{\left\langle r\otimes b_{n-1}\right\rangle } &  & \lightgray{\left\langle \ell\otimes b_{n-1}\right\rangle }\ar@{}[r]|-{\oplus} & \lightgray{\left\langle h\otimes b_{n-2}\right\rangle }\ar@{}[r]|-{\oplus}\ar[ddd]|-{\vdots} & \left\langle r\otimes b_{n-1}\right\rangle \\
 &  & \  &  &  &  & \ \\
 & \save"14,2"."15,4"*[F]\frm{}\restore & \  &  &  & \save"14,6"."15,8"*[F]\frm{}\restore & \ \\
\ar@{}[d]|-{2} & \left\langle \ell\otimes t_{2}\right\rangle \ar@{}[r]|-{\oplus}\ar@{}[d]|-{\oplus} & \lightgray{\left\langle h\otimes t_{1}\right\rangle }\ar@{}[r]|-{\oplus}\ar@{}[d]|-{\oplus} & \lightgray{\left\langle r\otimes t_{2}\right\rangle }\ar@{}[d]|-{\oplus} &  & \lightgray{\left\langle \ell\otimes t_{2}\right\rangle }\ar@{}[r]|-{\oplus}\ar@{}[d]|-{\oplus} & \lightgray{\left\langle h\otimes t_{1}\right\rangle }\ar@{}[r]|-{\oplus}\ar@{}[d]|-{\oplus} & \left\langle r\otimes t_{2}\right\rangle \ar@{}[d]|-{\oplus}\\
 & \left\langle \ell\otimes b_{2}\right\rangle \ar@{}[r]|-{\oplus} & \lightgray{\left\langle h\otimes b_{1}\right\rangle }\ar@{}[r]|-{\oplus}\ar[dd]\ar[dd] & \lightgray{\left\langle r\otimes b_{2}\right\rangle } &  & \lightgray{\left\langle \ell\otimes b_{2}\right\rangle }\ar@{}[r]|-{\oplus} & \lightgray{\left\langle h\otimes b_{1}\right\rangle }\ar@{}[r]|-{\oplus}\ar[dd]\ar[dd] & \left\langle r\otimes b_{2}\right\rangle \\
 & \save"17,2"."18,4"*[F]\frm{}\restore & \  &  &  & \save"17,6"."18,8"*[F]\frm{}\restore & \ \\
\ar@{}[d]|-{1} & \left\langle \ell\otimes t_{1}\right\rangle \ar@{}[r]|-{\oplus}\ar@{}[d]|-{\oplus} & \left\langle h\otimes t_{0}\right\rangle \ar@{}[r]|-{\oplus}\ar@{}[d]|-{\oplus} & \lightgray{\left\langle r\otimes t_{1}\right\rangle }\ar@{}[d]|-{\oplus} &  & \lightgray{\left\langle \ell\otimes t_{1}\right\rangle }\ar@{}[r]|-{\oplus}\ar@{}[d]|-{\oplus} & \lightgray{\left\langle h\otimes t_{0}\right\rangle }\ar@{}[r]|-{\oplus}\ar@{}[d]|-{\oplus}\  & \left\langle r\otimes t_{1}\right\rangle \ar@{}[d]|-{\oplus}\\
 & \left\langle \ell\otimes b_{1}\right\rangle \ar@{}[r]|-{\oplus} & \lightgray{\left\langle h\otimes b_{0}\right\rangle }\ar@{}[r]|-{\oplus}\ar[ddd] & \lightgray{\left\langle r\otimes b_{1}\right\rangle } &  & \lightgray{\left\langle \ell\otimes b_{1}\right\rangle }\ar@{}[r]|-{\oplus} & \left\langle h\otimes b_{0}\right\rangle \ar@{}[r]|-{\oplus}\ar[ddd] & \left\langle r\otimes b_{1}\right\rangle \\
 &  & \ \\
 & \save"21,2"."22,4"*[F]\frm{}\restore & \  &  &  & \save"21,6"."22,8"*[F]\frm{}\restore & \ \\
\ar@{}[d]|-{0} & \left\langle \ell\otimes t_{0}\right\rangle \ar@{}[rr]|-{\oplus}\ar@{}[d]|-{\oplus} & \  & \left\langle r\otimes t_{0}\right\rangle \ar@{}[d]|-{\oplus} &  & \lightgray{\left\langle \ell\otimes t_{0}\right\rangle }\ar@{}[rr]|-{\oplus}\ar@{}[d]|-{\oplus} & \  & \left\langle r\otimes t_{0}\right\rangle \ar@{}[d]|-{\oplus}\\
 & \left\langle \ell\otimes b_{0}\right\rangle \ar@{}[rr]|-{\oplus} &  & \lightgray{\left\langle r\otimes b_{0}\right\rangle } &  & \left\langle \ell\otimes b_{0}\right\rangle \ar@{}[rr]|-{\oplus} &  & \left\langle r\otimes b_{0}\right\rangle 
}
\]
Now, we have belabored their description so as to make it evident
that:
\begin{itemize}
\item the maps are monomorphisms;
\item the maps cover the target - obvious in degrees strictly greater than
$1$, requiring a subtraction in degree $1$, and obvious again in
degree $0$; and
\item that the requisite squares are indeed pullbacks - this is not hard
to see: in terms of generators we simply add $h\otimes t_{0}$ (resp.
$h\otimes b_{0}$) to the other generators in degree $1$ and remove
$\left\langle \ell\otimes t_{0}\right\rangle $ (resp. $\left\langle r\otimes b_{0}\right\rangle $)
in degree 0).
\end{itemize}
which concludes the proof.
\end{proof}
Then, synthesizing the proposition above, and our work regarding the
globular sum preservation, we prove the following.
\begin{thm}
For all cells $T$ of $\Theta$, the canonical morphism $\left[1\right]\LG T\liso\left[1\right]\otimes T$
is an isomorphism, as indicated.
\end{thm}

\begin{proof}
Since $\left[1\right]\LG\left(\_\right)$ preserves globular sums
(see Section \ref{subsec:The-assignment-LG-preserves-globular-sums}),
it suffices to show that $\left[1\right]\LG\overline{n}\liso\left[1\right]\otimes\overline{n}$
and this is implied by the proposition above.
\end{proof}

\subsection{\label{sec:Shifted-Product-Rule-=00005Comega-categories}Shifted
Product Rule and a $\protect\StrCat$-enriched categorical treatment
of the lax shuffle decomposition}

Its not hard to see that the objects of the $\omega$-category $\left[1\right]\otimes\left[n\right];\left(S_{i}\right)$,
the Gray cylinder on $\left[n\right];\left(S_{i}\right)$, are in
bijection with the objects of $\left[1\right]\times\left[n\right];\left(S_{i}\right)$,
the cartesian cylinder on $\left[n\right];\left(S_{i}\right)$. Indeed
both are in bijection with the set $\left\{ 0,1\right\} \times\left\{ 0,1,\cdots,n\right\} $.
This observation in turn gives us easy access to a description of
the cellular sets $\left[1\right]\otimes\left[n\right];\left(S_{i}\right)$
as $\omega$-category enriched categories.
\begin{defn}
For any $n\in\N$ and objects $S_{1},\dots,S_{n}\in\Ob\left(\Theta\right)$,
define $\mathsf{P.R.}\left(S_{1},\dots,S_{n}\right)$ to be the colimit
\[
\colim\left\{ \vcenter{\vbox{\xyR{1.5pc}\xyC{3pc}\xymatrix{\left(\left[1\right]\otimes S_{1}\right)\times S_{2}\times\cdots\times S_{n}\\
S_{1}\times S_{2}\times\cdots\times S_{n}\ar[u]^{\left\{ 1\right\} \otimes S_{1}\times\id\times\cdots\times\id}\ar[d]\\
\vdots\\
S_{1}\times\cdots\times S_{n-1}\times S_{n}\ar[d]_{\id\times\cdots\times\id\times\left\{ 0\right\} \otimes S_{n}}\ar[u]\\
S_{1}\times\cdots\times S_{n-1}\times\left(\left[1\right]\otimes S_{n}\right)
}
}}\right\} 
\]
taken in $\omega$-categories.
\end{defn}

\begin{rem}
The reader should observe that we could well have written the following
expression.
\[
\mathsf{P.R.}\left(\left(S_{i}\right)_{i\in\left\langle n\right\rangle }\right)=\left(\left(\left[1\right]\otimes S_{1}\right)\times\cdots\times S_{n}\right)\underset{\prod S_{i}}{\bigoplus}\cdots\underset{\prod S_{i}}{\bigoplus}\left(S_{1}\times\cdots\times\left(\left[1\right]\otimes S_{n}\right)\right)
\]
\end{rem}

\begin{lem}
\label{lem:PR-is-PR-isPR}For any $A_{1},A_{2},\dots,A_{n}$ and $B_{1},B_{2},\dots,B_{m}$
of $\Theta$ there is a canonical isomorphism from the colimit
\[
\colim\left\{ \vcenter{\vbox{\xyR{1.5pc}\xyC{3pc}\xymatrix{\PR\left(A_{1},A_{2},\dots,A_{n}\right)\times B_{1}\times B_{2}\times\cdots\times B_{m}\\
A_{1}\times A_{2}\times\cdots\times A_{n}\times B_{1}\times B_{2}\times\cdots\times B_{m}\ar[u]\ar[d]\\
A_{1}\times A_{2}\times\cdots\times A_{n}\times\PR\left(B_{1}\times B_{2}\times\cdots\times B_{m}\right)
}
}}\right\} 
\]
to the object $\PR\left(A_{1},A_{2},\dots,A_{n},B_{1},B_{2},\dots,B_{m}\right)$.
\end{lem}

\begin{lem}
\label{cor:=00005Comega-cat-descrption-of-lax-tensor}Let $n\in\N$
and let 
\begin{eqnarray*}
S_{1} & = & \left[s_{1}\right];\left(R_{1}^{1},R_{2}^{1},\dots,R_{s_{1}}^{1}\right)\\
S_{2} & = & \left[s_{2}\right];\left(R_{1}^{2},R_{2}^{2},\dots,R_{s_{2}}^{2}\right)\\
\vdots & \vdots & \vdots\\
S_{n-1} & = & \left[s_{n-1}\right];\left(R_{1}^{n-1},R_{2}^{n-1},\dots,R_{s_{n-1}}^{n-1}\right)\\
S_{n} & = & \left[s_{n}\right];\left(R_{1}^{n},R_{2}^{n},\dots,R_{s_{n}}^{n}\right)
\end{eqnarray*}
 be cells of $\Theta.$ Then, for each $1\leq\ell\leq k\leq n$ the
$\omega$-categories
\[
\mathsf{P.R.}\left(\left(S_{i}\right)_{i\in\left\langle \ell,k\right\rangle }\right)=\mathsf{P.R.}\left(\left(\left[s_{i}\right];\left(R_{j}^{i}\right)_{j\in\left\langle s_{i}\right\rangle }\right)_{i\in\left\langle \ell,k\right\rangle }\right)
\]
 admit the following description as categories enriched in $\omega$-categories.
\begin{itemize}
\item on objects we find 
\begin{eqnarray*}
\Ob\left(\mathsf{P.R.}\left(\left(S_{i}\right)_{i\in\left\langle \ell,k\right\rangle }\right)\right) & = & \Ob\left(\PR\left(\left(\left[s_{i}\right]\right)_{i\in\left\langle \ell,k\right\rangle }\right)\right)\\
 & \iso & \left\{ \ell,\ell+1,\dots,k\right\} \times\left(\underset{i\in\left\langle \ell,k\right\rangle }{\prod}\left\{ 1,2,\dots,s_{i}\right\} \right)
\end{eqnarray*}
\item and the $\Hom$-$\omega$-categories 
\[
\Hom_{\mathsf{P.R.}\left(\left(S_{i}\right)_{i\in\left\langle \ell,k\right\rangle }\right)}\left(\left(x,\left(z_{i}\right)_{i\in\left\langle \ell,k\right\rangle }\right),\left(y,\left(w_{i}\right)_{i\in\left\langle \ell,k\right\rangle }\right)\right)
\]
 are given by the following formula(e):
\begin{itemize}
\item if $\ell\leq k$ then 
\[
\Hom_{\mathsf{P.R.}\left(\left(S_{i}\right)_{i\in\left\langle \ell,k\right\rangle }\right)}\left(\left(x,\left(z_{i}\right)_{i\in\left\langle \ell,k\right\rangle }\right),\left(y,\left(w_{i}\right)_{i\in\left\langle \ell,k\right\rangle }\right)\right)
\]
 is the $\omega$-category
\[
\left(\prod_{a\in\left\langle \ell,x\right\rangle }\prod_{q\in\left\langle z_{a},w_{a}\right\rangle }R_{q}^{a}\right)\times\PR\left(\left(\left(R_{q}^{b}\right)_{q\in\left\langle z_{b},w_{b}\right\rangle }\right)_{b\in\left\langle x,y\right\rangle }\right)\times\left(\prod_{c\in\left\langle y,k\right\rangle }\prod_{q\in\left\langle z_{c},w_{c}\right\rangle }R_{q}^{c}\right)
\]
\item if $\ell\not\leq k$ then 
\[
\Hom_{\mathsf{P.R.}\left(\left(S_{i}\right)_{i\in\left\langle \ell,k\right\rangle }\right)}\left(\left(x,\left(z_{i}\right)_{i\in\left\langle \ell,k\right\rangle }\right),\left(y,\left(w_{i}\right)_{i\in\left\langle \ell,k\right\rangle }\right)\right)=\varnothing
\]
\end{itemize}
\end{itemize}
\end{lem}

\begin{rem}
See that in the case $n=1$ this formula computes $\left[1\right]\otimes S_{1}$
as $\left[1\right]\otimes S=\mathsf{P.R.}\left(S\right)$. We note
too that, since we consider only cells of finite height, the recursion
in the definition above terminates.
\end{rem}

\begin{proof}
The proof of the first claim comes from consideration of the pasting
diagram for $\left[1\right]\otimes\left(\left[n\right];\left(T_{1},T_{2},\dots,T_{n}\right)\right)$
given in Figure \ref{fig:justifying the shuffle}. The second is more
complicated, but is 'essentially' the same.
\end{proof}
What's more, this shifted product rule description allows for a completely
explicit description of the $\omega$-category enriched functors $\left[1\right]\otimes f$
for each $f:A\longrightarrow B$ of $\Theta$.

\subsubsection{The $\omega$-functors $\left[1\right]\otimes f:\left[1\right]\otimes T\protect\longrightarrow\left[1\right]\otimes S$
as $\omega$-category enriched functors}

As we'll see, to characterize the action of $\left[1\right]\otimes\left(\_\right)$
on the morphisms of $\Theta$, it will suffice to recurse to full
subcategory $\t\int\t\subset\Theta$.

Suppose

\[
f;\mathbf{g}:\left[n\right];\left(T_{i}\right)_{i\in\left\langle n\right\rangle }\longrightarrow\left[m\right];\left(S_{j}\right)_{j\in\left\langle m\right\rangle }
\]
where 
\[
\mathbf{g}=\left(\left(g_{i\rightarrow j}\right)_{j\in F\left(f\right)\left(i\right)}\right)_{i\in\left\langle n\right\rangle }
\]
 with 
\[
g_{i\rightarrow j}:T_{i}\longrightarrow S_{j},
\]
be a morphism of $\Theta\liso\t\int\Theta$. Since
\[
\Ob\left(\left[n\right];\left(T_{i}\right)_{i\in\left\langle n\right\rangle }\right)\iso\Ob\left(\left[1\right]\otimes\left[n\right]\right)\iso\Ob\left(\left[1\right]\times\left[n\right]\right)\iso\Ob\left(\left[1\right]\right)\times\Ob\left(\left[n\right]\right)
\]
and the two functors $\left[1\right]\otimes\left(\_\right)$ and $\left[1\right]\times\left(\_\right)$
coincide on objects, we characterize the $\omega$-category enriched
functor 
\[
\left[1\right]\otimes f;\mathbf{g}:\left[1\right]\otimes\left[n\right];\left(T_{i}\right)_{i\in\left\langle n\right\rangle }\longrightarrow\left[1\right]\otimes\left[m\right];\left(S_{j}\right)_{j\in\left\langle m\right\rangle }
\]
 as follows.

We observe that $\Ob\left(\left[1\right]\otimes f;\mathbf{g}\right)$
is given
\[
\vcenter{\vbox{\xyR{.25pc}\xyC{6pc}\xymatrix{\Ob\left(\left[1\right]\otimes\left[n\right];\left(T_{i}\right)\right)\ar[r]\ar@{=}[d] & \Ob\left(\left[1\right]\otimes\left[m\right];\left(S_{j}\right)\right)\ar@{=}[d]\\
\left[1\right]\times\left[n\right]\ar[r]^{\id\times f} & \left[1\right]\times\left[m\right]
}
}}
\]
The types of the constituent $\omega$-functors
\[
\vcenter{\vbox{\xyR{.25pc}\xyC{6pc}\xymatrix{\Hom_{\left[1\right]\otimes\left[n\right];\left(T_{i}\right)}\left(\left(y,x\right),\left(w,z\right)\right)\ar[r] & \Hom_{\left[1\right]\otimes\left[m\right];\left(S_{j}\right)}\left(\left(y,f\left(x\right)\right),\left(w,f\left(z\right)\right)\right)}
}}
\]
are then given case-wise by Corollary \ref{cor:=00005Comega-cat-descrption-of-lax-tensor},
depending on $\left\langle y,w\right\rangle $:
\begin{itemize}
\item $1\in\left\langle y,w\right\rangle $, 
\[
\vcenter{\vbox{\xyR{.25pc}\xyC{6pc}\xymatrix{\PR\left(\left(T_{i}\right)_{i\in\left\langle x,z\right\rangle }\right)\ar[r]^{\PR\left(f;\mathbf{g},\left\langle x,z\right\rangle \right)} & \underset{i\in\left\langle x,z\right\rangle }{\prod}\PR\left(\left(S_{j}\right)_{j\in F\left(f\right)\left(i\right)}\right)}
}}
\]
\item $1\notin\left\langle y,w\right\rangle =\varnothing$
\[
\vcenter{\vbox{\xyR{.25pc}\xyC{6pc}\xymatrix{\underset{i\in\left\langle x,z\right\rangle }{\prod}\left(T_{i}\right)\ar[r]^{\underset{i\in\left\langle x,z\right\rangle }{\prod}\left(g_{i\rightarrow j}\right)_{j\in F\left(f\right)\left(i\right)}} & \underset{i\in\left\langle x,z\right\rangle }{\prod}\underset{j\in F\left(f\right)\left(i\right)}{\prod}S_{j}}
}}
\]
\item $y\not\leq w$
\[
\vcenter{\vbox{\xyR{.25pc}\xyC{6pc}\xymatrix{\varnothing\ar[r] & \varnothing}
}}
\]
\end{itemize}
with the meaning of $\PR\left(f;\mathbf{g},\left\langle x,z\right\rangle \right)$
given by recursion to the following case.
\begin{defn}
Suppose 
\[
f;\mathbf{g}:\left[n\right];\left(\left[t_{i}\right]\right)_{i\in\left\langle n\right\rangle }\longrightarrow\left[m\right];\left(\left[s_{j}\right]\right)_{j\in\left\langle m\right\rangle }
\]
to be a morphism of $\t\int\t$. Then set 
\[
\vcenter{\vbox{\xyR{.25pc}\xyC{6pc}\xymatrix{\PR\left(\left(\left[t_{i}\right]\right)_{i\in\left\langle x,z\right\rangle }\right)\ar[r]^{\PR\left(f;\mathbf{g},\left\langle x,z\right\rangle \right)} & \underset{i\in\left\langle x,z\right\rangle }{\prod}\PR\left(\left(\left[s_{j}\right]\right)_{j\in F\left(f\right)\left(i\right)}\right)}
}}
\]
to be the functor given as follows. We set

\[
\vcenter{\vbox{\xyR{.25pc}\xyC{6pc}\xymatrix{\Ob\left(\PR\left(\left(\left[t_{i}\right]\right)_{i\in\left\langle x,z\right\rangle }\right)\right)\ar@{=}[d]\ar[r]^{\Ob\left(\PR\left(f;\mathbf{g},\left\langle x,z\right\rangle \right)\right)} & \underset{i\in\left\langle x,z\right\rangle }{\prod}\Ob\left(\PR\left(\left(\left[s_{j}\right]\right)_{j\in F\left(f\right)\left(i\right)}\right)\right)\ar@{=}[d]\\
\left[x,z\right]\times\left(\underset{i\in\left\langle x,z\right\rangle }{\prod}\left[t_{i}\right]\right)\ar[r] & \prod_{i\in\left\langle x,z\right\rangle }\left[f\left(i-1\right),f\left(i\right)\right]\times\prod_{j\in F\left(f\right)\left(i\right)}\left[s_{j}\right]
}
}}
\]
to be the product of:
\begin{itemize}
\item the composite multimorphism 
\[
\vcenter{\vbox{\xyR{0pc}\xyC{6pc}\xymatrix{\left[x,z\right]\ar[r] & \underset{i\in\left\langle x,z\right\rangle }{\prod}\left[i-1,i\right]\ar[r] & \underset{i\in\left\langle x,z\right\rangle }{\prod}\left[f\left(i-1\right),f\left(i\right)\right]\\
p<i\ar@{|->}[r] & i-1\\
p\geq i\ar@{|->}[r] & i\\
 & \left(w_{i}\right)_{i\in\left\langle x,z\right\rangle }\ar@{|->}[r] & \left(f\vert_{{\scriptscriptstyle \left[i-1,i\right]}}\left(w_{i}\right)\right)_{i\in\left\langle x,z\right\rangle }
}
}}
\]
which locates a $0-$globular sum inside a product; and
\item the product of multimorphisms
\[
\vcenter{\vbox{\xyR{.25pc}\xyC{6pc}\xymatrix{\underset{i\in\left\langle x,z\right\rangle }{\prod}\left(g_{i\rightarrow j}\right)_{j\in F\left(j\right)\left(i\right)}:\underset{i\in\left\langle x,z\right\rangle }{\prod}\left[t_{i}\right]\ar[r] & \underset{i\in\left\langle x,z\right\rangle }{\prod}\left(\prod_{j\in F\left(f\right)\left(i\right)}\left[s_{j}\right]\right)}
}}
\]
\end{itemize}
and on morphisms we define the constituent functors
\[
\xyR{1pc}\xyC{1pc}\xymatrix{\Hom\left(\left(b,\left(a_{i}\right)_{i\in\left\langle x,z\right\rangle }\right),\left(d,\left(c_{i}\right)_{i\in\left\langle x,z\right\rangle }\right)\right)\ar[d]\\
\underset{i\in\left\langle x,z\right\rangle }{\prod}\Hom\left(\left(f\vert_{{\scriptscriptstyle \left[i-1,i\right]}}\left(b\right),\left(g_{i\rightarrow j}\left(a_{i}\right)\right)\right),\left(f\vert_{{\scriptscriptstyle \left[i-1,i\right]}}\left(d\right),\left(g_{i\rightarrow j}\left(c_{i}\right)\right)\right)\right)
}
\]
by way of the following reductions.

The source of the $2$-functor computes to the product
\[
\underset{i\in\left\langle x,z\right\rangle }{\prod}\begin{cases}
\PR\left(\left(\left[0\right]\right)_{j\in\left\langle a_{i},c_{i}\right\rangle }\right) & i\in\left\langle b,d\right\rangle \neq\varnothing\\
\prod_{i\in\left\langle a_{i},b_{i}\right\rangle }\left[0\right] & i\notin\left\langle b,d\right\rangle \neq\varnothing\\
\varnothing & i\notin\left\langle b,d\right\rangle =\varnothing
\end{cases}
\]
which reduces to 
\[
\underset{i\in\left\langle x,z\right\rangle }{\prod}\begin{cases}
\left[a_{i},c_{i}\right] & i\in\left\langle b,d\right\rangle \neq\varnothing\\
\left[0\right] & i\notin\left\langle b,d\right\rangle \neq\varnothing\\
\varnothing & i\notin\left\langle b,d\right\rangle =\varnothing
\end{cases}
\]
Similarly, the target computes to the product 
\[
\underset{i\in\left\langle x,z\right\rangle }{\prod}\underset{j\in F\left(f\right)\left(i\right)}{\prod}\begin{cases}
\PR\left(\left(\left[0\right]\right)_{k\in\left\langle g_{i\rightarrow j}\left(a_{i}\right),g_{i\rightarrow j}\left(c_{i}\right)\right\rangle }\right) & j\in\left\langle f\vert_{{\scriptscriptstyle \left[i-1,i\right]}}\left(b\right),f\vert_{{\scriptscriptstyle \left[i-1,i\right]}}\left(d\right)\right\rangle \neq\varnothing\\
\prod_{k\in\left\langle g_{i\rightarrow j}\left(a_{i}\right),g_{i\rightarrow j}\left(c_{i}\right)\right\rangle }\left[0\right] & j\not\in\left\langle f\vert_{{\scriptscriptstyle \left[i-1,i\right]}}\left(b\right),f\vert_{{\scriptscriptstyle \left[i-1,i\right]}}\left(d\right)\right\rangle \neq\varnothing\\
\varnothing & j\notin\left\langle f\vert_{{\scriptscriptstyle \left[i-1,i\right]}}\left(b\right),f\vert_{{\scriptscriptstyle \left[i-1,i\right]}}\left(d\right)\right\rangle =\varnothing
\end{cases}
\]
which reduces to 
\[
\underset{i\in\left\langle x,z\right\rangle }{\prod}\underset{j\in F\left(f\right)\left(i\right)}{\prod}\begin{cases}
\left[g_{i\rightarrow j}\left(a_{i}\right),g_{i\rightarrow j}\left(c_{i}\right)\right] & j\in\left\langle f\vert_{{\scriptscriptstyle \left[i-1,i\right]}}\left(b\right),f\vert_{{\scriptscriptstyle \left[i-1,i\right]}}\left(d\right)\right\rangle \neq\varnothing\\
\left[0\right] & j\not\in\left\langle f\vert_{{\scriptscriptstyle \left[i-1,i\right]}}\left(b\right),f\vert_{{\scriptscriptstyle \left[i-1,i\right]}}\left(d\right)\right\rangle \neq\varnothing\\
\varnothing & j\notin\left\langle f\vert_{{\scriptscriptstyle \left[i-1,i\right]}}\left(b\right),f\vert_{{\scriptscriptstyle \left[i-1,i\right]}}\left(d\right)\right\rangle =\varnothing
\end{cases}
\]
Thus, we define 
\[
\xyR{1pc}\xyC{1pc}\xymatrix{\Hom\left(\left(b,\left(a_{i}\right)_{i\in\left\langle x,z\right\rangle }\right),\left(d,\left(c_{i}\right)_{i\in\left\langle x,z\right\rangle }\right)\right)\ar[d]\\
\underset{i\in\left\langle x,z\right\rangle }{\prod}\Hom\left(\left(f\vert_{{\scriptscriptstyle \left[i-1,i\right]}}\left(b\right),\left(g_{i\rightarrow j}\left(a_{i}\right)\right)\right),\left(f\vert_{{\scriptscriptstyle \left[i-1,i\right]}}\left(d\right),\left(g_{i\rightarrow j}\left(c_{i}\right)\right)\right)\right)
}
\]
to be the obvious map which is constituted of, component-wise, restrictions
of maps $g_{i\rightarrow j}$ and canonical maps $\left[0\right]\rightarrow\left[0\right]$
or $\varnothing\rightarrow\varnothing$.
\end{defn}

\begin{example}
While the formulae above are straightforward, it is not necessarily
easy to see exactly what they are doing. As such it may be hard to
convince one's self that these formulae are: (1) correct and/or (2)
interesting. As an example then, which should both clarify and justify
the formulae, we'll consider the map
\[
d^{1};\left(s^{1},s^{0}\right)=\left\{ 0,2\right\} ;\left(\left\{ 0,1,1\right\} ,\left\{ 0,0,1\right\} \right):\left[1\right];\left(\left[2\right]\right)\longrightarrow\left[2\right];\left(\left[1\right],\left[1\right]\right)
\]
and describe the map
\[
\mathsf{P.R.}\left(d^{1};\left(s^{1},s^{0}\right),\left\langle 0,1\right\rangle \right):\mathsf{P.R.}\left(\left[2\right]\right)\longrightarrow\mathsf{P.R.}\left(\left[1\right],\left[1\right]\right).
\]
 In Figure \ref{fig:2-lg-1} we sketch $\mathsf{P.R.}\left(\left[2\right]\right)$
and in Figure \ref{fig:1-times-1-lg-1} we sketch a pasting diagram
for $\PR\left(\left[1\right],\left[1\right]\right)$\footnote{We hope that the reader will forgive us for only providing a picture
of the pasting diagram in this second case; including the copy of
$\left[1\right]\times\left[1\right]$ as the $\Hom$-category between
$\left(0,0,0\right)$ and $\left(1,1,2\right)$ is difficult.}. Then, by our formula, the functor 
\[
\mathsf{P.R.}\left(d^{1};\left(s^{1},s^{0}\right),\left\langle 0,1\right\rangle \right):\PR\left(\left[2\right]\right)\longrightarrow\PR\left(\left[1\right],\left[1\right]\right)
\]
 can be given explicitly by
\[
\vcenter{\vbox{\xyR{.25pc}\xyC{6pc}\xymatrix{\Ob\left(\PR\left(\left[2\right]\right)\right)\ar[r]\ar@{=}[d] & \Ob\left(\mathsf{P.R.}\left(\left[1\right],\left[1\right]\right)\right)\ar@{=}[d]\\
\left[1\right]\times\left[2\right]\ar[r]^{d^{1}\times\left(s^{1},s^{0}\right)} & \left[2\right]\times\left[1\right]\times\left[1\right]
}
}}
\]
We first compute the functors between $\Hom$-categories of the first
sort in Lemma \ref{cor:=00005Comega-cat-descrption-of-lax-tensor}.
We compute
\[
\vcenter{\vbox{\xyR{.25pc}\xyC{6pc}\xymatrix{\Hom_{\PR\left(\left[2\right]\right)}\left(\left(0,0\right),\left(1,1\right)\right)\ar[r]\ar@{=}[d] & \Hom_{\PR\left(\left[0\right]_{1},\left[0\right]_{2}\right)}\left(\left(0,0,0\right),\left(2,1,0\right)\right)\ar@{=}[d]\\
\PR\left(\left[0\right]_{1}\right)\ar@{=}[d] & \underbrace{\PR\left(\overbrace{\left[0\right]_{1}}^{0<1}\right)}_{p=1}\times\underbrace{\PR\left(\overbrace{}^{0\not<0}\right)}_{p=2}\ar@{=}[d]\\
\left[0,1\right]\ar[r]^{\left\{ 0,1\right\} ,\left\{ 0,0\right\} } & \underbrace{\left[0,1\right]}_{p=1}\times\underbrace{\left[0,0\right]}_{p=2}
}
}}
\]
where we note that $\left\{ 0,1\right\} =\left\{ 0,1,1\right\} \restriction_{\left[0,1\right]}$
and $\left\{ 0,0\right\} =\left\{ 0,0,1\right\} \restriction_{\left[0,1\right]}$.
We compute
\[
\vcenter{\vbox{\xyR{.25pc}\xyC{6pc}\xymatrix{\Hom\left(\left(1,0\right),\left(2,1\right)\right)\ar[r]\ar@{=}[d] & \Hom\left(\left(0,1,0\right),\left(2,1,1\right)\right)\ar@{=}[d]\\
\PR\left(\left[0\right]_{2}\right)\ar@{=}[d] & \underbrace{\PR\left(\overbrace{}^{1\not<1}\right)}_{p=1}\times\underbrace{\PR\left(\overbrace{\left[0\right]_{1}}^{0<1}\right)}_{p=2}\ar@{=}[d]\\
\left[1,2\right]\ar[r]^{\left\{ 1,1\right\} ,\left\{ 0,1\right\} } & \left[1,1\right]\times\left[0,1\right]
}
}}
\]
where we note that $\left\{ 1,1\right\} =\left\{ 0,1,1\right\} \restriction_{\left[1,2\right]}$
and $\left\{ 0,1\right\} =\left\{ 0,0,1\right\} \restriction_{\left[1,2\right]}$.
Lastly, we compute the following.
\[
\vcenter{\vbox{\xyR{.25pc}\xyC{6pc}\xymatrix{\Hom\left(\left(0,0\right),\left(1,2\right)\right)\ar[r]\ar@{=}[d] & \Hom\left(\left(0,0,0\right),\left(2,1,1\right)\right)\ar@{=}[d]\\
\PR\left(\left[0\right]_{1},\left[0\right]_{2}\right)\ar@{=}[d] & \underbrace{\PR\left(\overbrace{\left[0\right]_{1}}^{0<1}\right)}_{p=1}\times\underbrace{\PR\left(\overbrace{\left[0\right]_{1}}^{0<1}\right)}_{p=2}\ar@{=}[d]\\
\left[0,2\right]\ar[r]^{\left(\left\{ 0,0,1\right\} ,\left\{ 0,1,1\right\} \right)} & \left[0,1\right]\times\left[0,1\right]
}
}}
\]
We next compute the functors between $\Hom$-categories of the second
sort from Lemma \ref{cor:=00005Comega-cat-descrption-of-lax-tensor}.
We compute
\[
\vcenter{\vbox{\xyR{.25pc}\xyC{6pc}\xymatrix{\Hom_{\PR\left(\left[2\right]\right)}\left(\left(0,0\right),\left(0,1\right)\right)\ar[r]\ar@{=}[d] & \Hom_{\PR\left(\left[0\right]_{1},\left[0\right]_{2}\right)}\left(\left(0,0,0\right),\left(0,1,0\right)\right)\ar@{=}[d]\\
\left[0\right]_{1}\ar[r] & \text{\ensuremath{\underbrace{\left[0\right]_{1}}_{p=1}}}\times\text{\ensuremath{\underbrace{\bullet}_{p=2}}},
}
}}
\]
we compute

\[
\vcenter{\vbox{\xyR{.25pc}\xyC{6pc}\xymatrix{\Hom_{\PR\left(\left[2\right]\right)}\left(\left(0,1\right),\left(0,2\right)\right)\ar[r]\ar@{=}[d] & \Hom_{\PR\left(\left[0\right]_{1},\left[0\right]_{2}\right)}\left(\left(0,1,0\right),\left(0,1,1\right)\right)\ar@{=}[d]\\
\left[0\right]_{2}\ar[r] & \text{\ensuremath{\underbrace{\bullet}_{p=1}}}\times\text{\ensuremath{\underbrace{\left[0\right]_{1}}_{p=2}}}
}
}}
\]
and we compute
\[
\vcenter{\vbox{\xyR{.25pc}\xyC{6pc}\xymatrix{\Hom_{\PR\left(\left[2\right]\right)}\left(\left(0,0\right),\left(0,2\right)\right)\ar[r]\ar@{=}[d] & \Hom_{\PR\left(\left[0\right]_{1},\left[0\right]_{2}\right)}\left(\left(0,0,0\right),\left(0,1,1\right)\right)\ar@{=}[d]\\
\left[0\right]_{1}\times\left[0\right]_{2}\ar[r] & \text{\ensuremath{\underbrace{\left[0\right]_{1}}_{p=1}}}\times\text{\ensuremath{\underbrace{\left[0\right]_{1}}_{p=2}}}
}
}}
\]
Changing the second index repeats the pattern and in all other cases
the source $\Hom$-category is the initial object.
\end{example}

\begin{rem}
That these formulae correct express $\left[1\right]\otimes f;\mathbf{g}$
is, while complicated, completely formal and thus left to the reader.

\pagebreak{}
\end{rem}

\begin{figure}[H]
\noindent\fbox{\begin{minipage}[t]{1\columnwidth - 2\fboxsep - 2\fboxrule}%
\bigskip{}
\adjustbox{scale=.75,center}{
$$
\begin{tikzcd}[ampersand replacement=\&]
	\&
		\&
			\&
				\&
					{(1,0)}
					\arrow[rrrrd]
					%to (2,0)
					\arrow[ddd]
				    %to (1,1)
					\arrow[
						rrrrdddd,
						bend left=15,
						""{name=SFT, below, pos=.45}
					]
					%to (2,1)
					\arrow[
						rrrrdddd,
						bend right=15,
						""{name=SFB, above}
					]
					%to (2,1)
					\arrow[
						from = SFT,
						to = SFB,
						Rightarrow
					]
					\&
						\&
							\&
								\&
									\\
%---------------------------------------------
{(0,0)}
\arrow[rrrru]
%to (1,0)
\arrow[ddd]
%to (0,1)
\arrow[
	rrrrdd,
	bend left=20,
	""{name=FFT, below}
]
%to (1,1)
\arrow[
	rrrrdd,
	bend right=20,
	""{name=FFB, above}
]
\arrow[
	from = FFT,
	to = FFB,
	Rightarrow
]
\arrow[
	rrrrrrrrddd,
	bend left=20,
	crossing over,
	""{name=LFT, below},
	%shorten <=10pt,
	%shorten >=10pt
]
\arrow[rrrrrrrr]
%to (2,0)
	\&
		\&
			\&
				\&
					\&
						\&
							\&
								\&
									{(2,0)}
									\arrow[ddd]
									\\
%---------------------------------------------
	\&
		\&
			\&
				\&
					\&
						\&
							\&
								\&
									\\
%---------------------------------------------
	\&
		\&
			\&
				\&
				{(1,1)}
				\arrow[rrrrd]
					\&
						\&
							\&
								\&
									\\
%---------------------------------------------
{(0,1)}
\arrow[rrrru]
	\&
		\&
			\&
				\&
					\&
						\&
							\&
								\&
								{(2,1)}
								%last arrows for 'over'-ness
								%from (0,0) to (2,1)
								%moved the LFT to the (0,0)-location
			 				   \arrow[
									from=lllllllluuu,
									crossing over,
									""{name=LFMU, above, pos=0.4},
									""{name=LFMD, below, pos=0.4},
								]
								\arrow[
									from=lllllllluuu,
									bend right=20,
									crossing over,
									""{name=LFB, above}
								]
								\arrow[
									from=LFT,
									to=LFMU,
									Rightarrow,
									bend right=20,
									crossing over
								]
								\arrow[
									from=LFMD,
									to=LFB,
									Rightarrow,
									bend right=20,
									crossing over
								]
								\arrow[
									Rightarrow,
									from=LFT,
									to=LFB,
									crossing over,
									bend left=10
								]
								\arrow[from=llllllll]
\end{tikzcd}
$$
}

\caption{\label{fig:2-lg-1}The $\omega$-category $\left[1\right]\protect\LG\left[2\right]$}
\bigskip{}
\end{minipage}}
\end{figure}

\begin{figure}[H]
\noindent\fbox{\begin{minipage}[t]{1\columnwidth - 2\fboxsep - 2\fboxrule}%
\bigskip{}
\adjustbox{scale=.75,center}{
$$
\begin{tikzcd}[ampersand replacement=\&]
	\&
		\&
			\&
			{(1,0,0)}
			\arrow[rrrrd]
			\arrow[dddd]
			\arrow[rrrrddddd]
				\&
					\&
						\&
							\&
								\\
%----------------------------------------
	\&
		\&
			\&
				\&
					\&
						\&
							\&  							{(1,1,0)}  							\arrow[dddd] 								\\ %----------------------------------------
{(0,0,0)}
\arrow[dddd]
\arrow[rrruu]
\arrow[
	rrrdd,
	bend left,
	""{name=C,below, pos=0.4}
]
\arrow[
	rrrdd,
	bend right,
	""{name=D, above, pos=0.6}
]
\arrow[
	Rightarrow,
	from=C,
	to=D
]
	\&
		\&
			\&
				\&
					\&
						\&
							\&
								\\
%----------------------------------------
	\&
		\&
			\&
				\&
				{(0,1,0)}
					\&
						\&
							\&
								\\
%----------------------------------------
	\&
		\&
			\&
			{(1,0,1)}
			\arrow[rrrrd]
			\arrow[dddd]
			\arrow[
				rrrrddddd,
				bend left,
				""{name=G, below, pos=0.4}
			]
			\arrow[
				rrrrddddd,
				bend right,
				""{name=H, above, pos=0.6}
			]
			\arrow[
				Rightarrow,
				from=G,
				to=H
			]
				\&
					\&
						\&
							\&
								\\
%----------------------------------------
	\&
		\&
			\&
				\&
					\&
						\&
							\&
							{(1,1,1)}
							\arrow[dddd]
								\\
%----------------------------------------
{(0,0,1)}
\arrow[dddd]
\arrow[rrruu]
\arrow[rrrdd]
	\&
		\&
			\&
				\&
					\&
						\&
							\&
								\\
%----------------------------------------
	\&
		\&
			\&
				\&
				{(0,1,1)}
					\&
						\&
							\&
								\\
%----------------------------------------
	\&
		\&
			\&
			{(1,0,2)}
			\arrow[rrrrd]
				\&
					\&
						\&
							\&
								\\
%----------------------------------------
	\&
		\&
			\&
				\&
					\&
						\&
							\&
							{(1,1,2)}
								\\
%----------------------------------------
{(0,0,2)}
\arrow[rrrrd]
\arrow[rrruu]
\arrow[rrrrrrru]
	\&
		\&
			\&
				\&
					\&
						\&
							\&
								\\
%----------------------------------------
	\&
		\&
			\&
				\&
				{(0,1,2)}
				\arrow[rrruu]
					\&
						\&
							\&
								\\
%------------------------------------------
%arrows sheet two
%0,0,0 to 1,1,1
\arrow[
	from=uuuuuuuuuu,
	to=uuuuuuurrrrrrr,
	bend left,
	crossing over,
	""{name=A, below}
]
%0,0,0 to 1,1,1
\arrow[
	from=uuuuuuuuuu,
	to=uuuuuuurrrrrrr,
	bend right,
	crossing over,
	""{name=B, above}
]
\arrow[
	from=A,
	to=B,
	Rightarrow,
	crossing over
]
%0,0,0 to 1,1,0
\arrow[
	from=uuuuuuuuuu,
	to=rrrrrrruuuuuuuuuuu,
	crossing over
]
%0,0,1 to 1,1,1
\arrow[
	from=uuuuuu,
	to=uuuuuuurrrrrrr,
	crossing over
]
%0,0,1 to 1,1,2
\arrow[
	from=uuuuuu,
	to=uuurrrrrrr,
	bend left,
	crossing over,
	""{name=I, below}
]
\arrow[
	from=uuuuuu,
	to=uuurrrrrrr,
	bend right,
	crossing over,
	""{name=J, above}
]
\arrow[
	Rightarrow,
	from=I,
	to=J,
	crossing over
]
%
%
%------------------------------------------
%arrows sheet three
%0,0,0 to 0,1,0
\arrow[
	from=uuuuuuuuuu,
	to=uuuuuuuuurrrr,
	crossing over
]
%0,0,0 to 0,1,1
\arrow[
	from=uuuuuuuuuu,
	to=uuuuurrrr,
	crossing over
]
%0,0,1 to 0,1,1
\arrow[
	from=uuuuuu,
	to=uuuuurrrr,
	crossing over
]
%move these to target
%0,0,1 to 0,1,2
\arrow[
	from=uuuuuu,
	to=urrrr,
	crossing over,
	bend left,
	""{name=X, below, pos=0.4}
]
\arrow[
	from=uuuuuu,
	to=urrrr,
	crossing over,
	bend right,
	""{name=Y, above, pos=0.6}
]
\arrow[
	Rightarrow,
	from=X,
	to=Y,
	crossing over
]
%0,1,0 to 1,1,0
\arrow[
	from=uuuuuuuuurrrr,
	to=uuuuuuuuuuurrrrrrr,
	crossing over
]
%0,1,0 to 0,1,1
\arrow[
	from=uuuuuuuuurrrr,
	to=uuuuurrrr,
	crossing over
]
%0,1,0 to 1,1,1
\arrow[
	from=uuuuuuuuurrrr,
	to=uuuuuuurrrrrrr,
	bend left,
	crossing over,
	""{name=E, below, pos=0.4}
]
%0,1,0 to 1,1,1
\arrow[
	from=uuuuuuuuurrrr,
	to=uuuuuuurrrrrrr,
	bend right,
	crossing over,
	""{name=F,above, pos=0.6}
]
%the 2-cell
\arrow[
	Rightarrow,
	from=E,
	to=F,
	crossing over
]
%0,1,1 to 1,1,1
\arrow[
	from=uuuuurrrr,
	to=uuuuuuurrrrrrr,
	crossing over
]
%0,1,1 to 0,1,2
\arrow[
	from=uuuuurrrr,
	to=urrrr,
	crossing over
]
%0,1,1 to 1,1,2
%\arrow[rrrdd, crossing over]
%move this to target
%0,1,0 to 0,1,1
%\arrow[
%	dddd,
%	crossing over
%]
%
%
%
\end{tikzcd}
$$
}

\caption{\label{fig:1-times-1-lg-1}A pasting diagram for the strict-$\omega$-category
$\left[1\right]\protect\LG\left(\left[1\right]\times\left[1\right]\right)$}
\bigskip{}
\end{minipage}}
\end{figure}

\pagebreak{}

\subsection{An even more explicit decomposition for the lax-Gray Cylinder}

Now, the lax-shuffle decomposition for $\left[1\right]\otimes\left(\_\right)$
enjoys only some of the useful properties that the shuffle decomposition
does for $\text{\ensuremath{\left[1\right]}\ensuremath{\ensuremath{\times\left(\_\right)}} }$.
The formula requires taking colimits in the category of $\omega$-categories
and not in the larger category of cellular sets precisely because
not every cell of $\left[1\right]\otimes\left(\_\right)$ factors
through one of objects in that diagram. In light of the lemma of the
previous section however, we can cover $\left[1\right]\otimes T$
with cellular sets in the image of $\t\int\StrCat$. 
\begin{cor}
For any $\left[n\right];\left(A_{i}\right)$ of $\Theta$, $\left[1\right]\otimes\left[n\right];\left(A_{i}\right)$,
as a cellular set, is the colimit of diagram in Figure \ref{fig:Complete-wide-pushout-for-Gray-cylinder}
in cellular sets.
\begin{figure}
\noindent\fbox{\begin{minipage}[t]{1\columnwidth - 2\fboxsep - 2\fboxrule}%
\bigskip{}
\centering
\rotatebox{90}{
\resizebox{.9\textheight}{!}{
\begin{tikzpicture}[xscale=.75, yscale=.75,yslant=.25]
	\begin{pgfonlayer}{nodelayer}
		\node [style=Slanted Text] (22) at (-63, 40) {};
		\node [style=Slanted Text] (23) at (-45, 40) {};
		\node [style=Slanted Text] (24) at (-63, 10) {};
		\node [style=Slanted Text] (25) at (-45, 10) {};
		\node [style=Slanted Text] (26) at (-55, 38) {$\scriptstyle{[n]; \left( A_{\leq n} \right)}$};
		\node [style=Slanted Text] (33) at (-55, 12) {$\scriptstyle{[n]; \left( A_{\leq n} \right)}$};
		\node [style=Slanted Text] (37) at (-55, 22) {$\scriptstyle{[n]; \left( A_{\leq n} \right)}$};
		\node [style=Slanted Text] (38) at (-55, 28) {$\scriptstyle{[n]; \left( A_{\leq n} \right)}$};
		\node [style=Slanted Text] (39) at (-55, 20) {$\scriptstyle{\vdots}$};
		\node [style=Slanted Text] (40) at (-55, 26) {$\scriptstyle{[n]; \left( A_{\leq n} \right)}$};
		\node [style=Slanted Text] (41) at (-55, 32) {$\scriptstyle{\vdots}$};
		\node [style=Slanted Text] (42) at (-55, 36) {$\scriptstyle{[n]; \left( A_{\leq n} \right)}$};
		\node [style=Slanted Text] (43) at (-55, 34) {$\scriptstyle{[n]; \left( A_{\leq n} \right)]}$};
		\node [style=Slanted Text] (44) at (-55, 14) {$\scriptstyle{[n]; \left( A_{\leq n} \right)}$};
		\node [style=Slanted Text] (45) at (-55, 16) {$\scriptstyle{[n]; \left( A_{\leq n} \right)}$};
		\node [style=Slanted Text] (46) at (-55, 24) {$\scriptstyle{[n]; \left( A_{\leq n} \right)}$};
		\node [style=Slanted Text] (47) at (-55, 18) {$\scriptstyle{\vdots }$};
		\node [style=Slanted Text] (48) at (-55, 30) {$ \scriptstyle{\vdots }$};
		\node [style=Slanted Text] (101) at (-55, 25) {};
		\node [style=Slanted Text] (4) at (-40, 32) {$ \scriptstyle{[n+1]; \left( [0], A_{\leq n} \right)}$};
		\node [style=Slanted Text] (5) at (-40, 30) {$ \scriptstyle{[n]; \left( [1] \otimes A_1,  A_{\geq 2} \right)}$};
		\node [style=Slanted Text] (6) at (-40, 28) {$ \scriptstyle{[n+1]; \left( A_1,  [0] ,  A_{\geq 2} \right)}$};
		\node [style=Slanted Text] (7) at (-40, 26) {$ \scriptstyle{[n]; \left( A_1,  [1] \otimes A_2,  A_{\geq 3} \right)}$};
		\node [style=Slanted Text] (8) at (-40, 20) {$ \scriptstyle{[n+1]; \left( A_{\leq k-1},  [0],  A_{\geq k} \right)}$};
		\node [style=Slanted Text] (9) at (-40, 18) {};
		\node [style=Slanted Text] (10) at (-40, 16) {$\scriptstyle{[n+1]; \left( A_{\leq k}, [0],  A_{\geq k+1} \right)}$};
		\node [style=Slanted Text] (11) at (-40, 4) {$\scriptstyle{[n+1]; \left( A_{\leq n} ,  [0]  \right)}$};
		\node [style=Slanted Text] (12) at (-40, 6) {$\scriptstyle{[n]; \left( A_{\leq n-1},  [1] \otimes A_n \right)}$};
		\node [style=Slanted Text] (13) at (-40, 8) {$\scriptstyle{[n+1]; \left( A_{\leq n-1}, [0], A_{n} \right)}$};
		\node [style=Slanted Text] (14) at (-40, 10) {$\scriptstyle{[n]; \left( A_{\leq n},[1]\otimes A_{n-1}, A_{\geq n} \right)}$};
		\node [style=Slanted Text] (15) at (-40, 14) {$\scriptstyle{[n]; \left( A_{\leq k},  [1] \otimes A_{k+1},  A_{\geq k+2} \right)}$};
		\node [style=Slanted Text] (17) at (-40, 22) {$\scriptstyle{[n+1]; \left(A_{\leq k-2},  [1] \otimes A_{k-1},  A_{\geq k} \right)}$};
		\node [style=Slanted Text] (18) at (-40, 12) {$\scriptstyle{\vdots }$};
		\node [style=Slanted Text] (20) at (-40, 18) {$ \scriptstyle{[n]; \left( A_{\leq k-1},  [1] \otimes A_k,  A_{\geq k+1} \right) }$};
		\node [style=Slanted Text] (21) at (-40, 24) {$ \scriptstyle{\vdots }$};
		\node [style=Slanted Text] (0) at (-50, 33) {};
		\node [style=Slanted Text] (1) at (-30, 33) {};
		\node [style=Slanted Text] (2) at (-50, 3) {};
		\node [style=Slanted Text] (3) at (-30, 3) {};
		\node [style=Slanted Text] (116) at (-45, 40) {};
		\node [style=Slanted Text] (118) at (-55, 38) {};
		\node [style=Slanted Text] (119) at (-55, 12) {};
		\node [style=Slanted Text] (120) at (-55, 22) {};
		\node [style=Slanted Text] (121) at (-55, 28) {};
		\node [style=Slanted Text] (122) at (-55, 20) {};
		\node [style=Slanted Text] (123) at (-55, 26) {};
		\node [style=Slanted Text] (124) at (-55, 32) {};
		\node [style=Slanted Text] (125) at (-55, 36) {};
		\node [style=Slanted Text] (126) at (-55, 34) {};
		\node [style=Slanted Text] (127) at (-55, 14) {};
		\node [style=Slanted Text] (128) at (-55, 16) {};
		\node [style=Slanted Text] (129) at (-55, 24) {};
		\node [style=Slanted Text] (130) at (-55, 18) {};
		\node [style=Slanted Text] (131) at (-55, 30) {};
		\node [style=Slanted Text] (134) at (-40, 32) {};
		\node [style=Slanted Text] (135) at (-40, 30) {};
		\node [style=Slanted Text] (136) at (-40, 28) {};
		\node [style=Slanted Text] (137) at (-40, 26) {};
		\node [style=Slanted Text] (138) at (-40, 20) {};
		\node [style=Slanted Text] (139) at (-40, 18) {};
		\node [style=Slanted Text] (140) at (-40, 16) {};
		\node [style=Slanted Text] (141) at (-40, 4) {};
		\node [style=Slanted Text] (142) at (-40, 6) {};
		\node [style=Slanted Text] (143) at (-40, 8) {};
		\node [style=Slanted Text] (144) at (-40, 10) {};
		\node [style=Slanted Text] (145) at (-40, 14) {};
		\node [style=Slanted Text] (146) at (-40, 22) {};
		\node [style=Slanted Text] (147) at (-40, 12) {};
		\node [style=Slanted Text] (148) at (-40, 18) {};
		\node [style=Slanted Text] (149) at (-40, 24) {};
		\node [style=Slanted Text] (95) at (-15, 26) {};
		\node [style=Slanted Text] (56) at (-30, 19) {$ \scriptstyle{[n-1]; \left( A_1 \times [1] \otimes A_2,  A_{\geq 3} \right)}$};
		\node [style=Slanted Text] (61) at (-30, -1) {$ \scriptstyle{[n-1]; \left( A_{\leq n-2}, A_{n-1} \times [1] \otimes A_n \right)}$};
		\node [style=Slanted Text] (63) at (-30, 3) {$ \scriptstyle{[n-1]; \left( A_{n-1}, A_{n-2}\times[1]\otimes A_{n-1}, A_n \right)}$};
		\node [style=Slanted Text] (64) at (-30, 7) {$ \scriptstyle{[n-1]; \left( A_{\leq k-1},  A_{k} \times [1] \otimes A_{k+1},  A_{\geq k+2}\right)}$};
		\node [style=Slanted Text] (65) at (-30, 15) {$ \scriptstyle{[n-1]; \left( A_{\leq k-3},  A_{k-2} \times [1] \otimes A_{k-1},  A_{\geq k} \right)}$};
		\node [style=Slanted Text] (66) at (-30, 5) {$ \scriptstyle{ \vdots }$};
		\node [style=Slanted Text] (67) at (-30, 11) {$ \scriptstyle{ [n-1]; \left( A_{\leq k-2},  A_{k-1} \times [1] \otimes A_k,  A_{\geq k+1} \right) }$};
		\node [style=Slanted Text] (68) at (-30, 17) {$ \scriptstyle{ \vdots }$};
		\node [style=Slanted Text] (70) at (-20, 23) {$ \scriptstyle{ [n-1]; \left( [1] \otimes A_1 \times A_2,  A_{\geq 3} \right)}$};
		\node [style=Slanted Text] (72) at (-20, 19) {$ \scriptstyle{[n-1]; \left( A_1,  [1] \otimes A_2 \times A_3,  , A_{\geq 4} \right) }$};
		\node [style=Slanted Text] (74) at (-20, 11) {};
		\node [style=Slanted Text] (79) at (-20, 3) {$\scriptstyle{ [n-1]; \left( A_{n-2},[1]\otimes A_{n-1} \times A_n \right)}$};
		\node [style=Slanted Text] (80) at (-20, 7) {$\scriptstyle{[n-1]; \left( A_{\leq k},  [1] \otimes A_{k+1} \times A_{k+2},  A_{\geq k+3} \right)}$};
		\node [style=Slanted Text] (81) at (-20, 15) {$\scriptstyle{[n-1]; \left( A_{\leq k-2},  [1] \otimes A_{k-1},\times  A_{k},   A_{\geq k+1} \right)}$};
		\node [style=Slanted Text] (82) at (-20, 5) {$\scriptstyle{ \vdots }$};
		\node [style=Slanted Text] (83) at (-20, 11) {$\scriptstyle{ [n-1]; \left( A_{\leq k-1},  [1] \otimes A_k \times A_{k+1},   A_{\geq k+2 } \right) }$};
		\node [style=Slanted Text] (84) at (-20, 17) {$\scriptstyle{ \vdots }$};
		\node [style=Slanted Text] (58) at (-30, 11) {};
		\node [style=Slanted Text] (49) at (-35, 26) {};
		\node [style=Slanted Text] (50) at (-15, 26) {};
		\node [style=Slanted Text] (51) at (-35, -4) {};
		\node [style=Slanted Text] (52) at (-15, -4) {};
		\node [style=Slanted Text] (151) at (0, 19) {};
		\node [style=Slanted Text] (174) at (-10, 14) {$ \scriptstyle{[n-1];\left( \mathsf{P.R.}\left( A_{1},A_{2} \right),A_{\geq 3} \right)} $};
		\node [style=Slanted Text] (175) at (-10, 10) {$\ddots$};
		\node [style=Slanted Text] (176) at (-10, 6) {$ \scriptstyle{[n-1];\left( A_{\leq k-2}, \mathsf{P.R.}\left( A_{k-1},A_{k} \right),A_{\geq k+1} \right)} $};
		\node [style=Slanted Text] (177) at (-10, 2) {$ \scriptstyle{[n-1];\left( A_{\leq k-1}, \mathsf{P.R.}\left( A_{k},A_{k+1} \right),A_{\geq k+2} \right)} $};
		\node [style=Slanted Text] (178) at (-10, -2) {$\ddots$};
		\node [style=Slanted Text] (179) at (-10, -6) {$ \scriptstyle{[n-1];\left( A_{\leq n-2}, \mathsf{P.R.}\left( A_{n-1},A_{n} \right) \right)} $};
		\node [style=Slanted Text] (180) at (-15, 26) {};
		\node [style=Slanted Text] (182) at (0, 19) {};
		\node [style=Slanted Text] (170) at (-20, 19) {};
		\node [style=Slanted Text] (171) at (0, 19) {};
		\node [style=Slanted Text] (172) at (-20, -11) {};
		\node [style=Slanted Text] (173) at (0, -11) {};
		\node [style=Slanted Text] (183) at (0, 19) {};
		\node [style=Slanted Text] (184) at (0, 19) {};
		\node [style=Slanted Text] (186) at (0, 19) {};
		\node [style=Slanted Text] (187) at (15, 12) {};
		\node [style=Slanted Text] (190) at (15, 12) {};
		\node [style=Slanted Text] (213) at (15, 12) {};
		\node [style=Slanted Text] (221) at (10, 7) {};
		\node [style=Slanted Text] (225) at (10, 7) {$ \scriptstyle{[n-2]; \left( \mathsf{P.R.} \left( A_1, A_2 \right) \times A_3, A_{\geq 4} \right) }$};
		\node [style=Slanted Text] (226) at (0, -1) {$ \scriptstyle{[n-2];\left( A_{\leq k-3}, A_{k-2} \times \mathsf{P.R.}\left( A_{k-1},A_{k} \right),A_{\geq k+1} \right)} $};
		\node [style=Slanted Text] (227) at (10, -1) {$ \scriptstyle{[n-2];\left( A_{\leq k-2}, \mathsf{P.R.}\left( A_{k-1},A_{k} \right)\times A_{k+1},A_{\geq k+2} \right)} $};
		\node [style=Slanted Text] (228) at (0, -5) {$ \scriptstyle{[n-2];\left( A_{\leq k-1}, \mathsf{P.R.}\left( A_{k},A_{k+1} \right),A_{\geq k+2},A_{\geq k+3} \right)} $};
		\node [style=Slanted Text] (229) at (10, -5) {$ \scriptstyle{[n-2];\left( A_{\leq k-1}, \mathsf{P.R.}\left( A_{k},A_{k+1} \right) \times A_{2},A_{\geq k+3} \right)} $};
		\node [style=Slanted Text] (232) at (0, -13) {$ \scriptstyle{[n-2];\left( A_{\leq n-3},A_{n-2} \times \mathsf{P.R.}\left( A_{n-1},A_{n} \right) \right)} $};
		\node [style=Slanted Text] (233) at (10, 3) {$ \cdots $};
		\node [style=Slanted Text] (234) at (0, 3) {$ \cdots $};
		\node [style=Slanted Text] (235) at (10, -9) {$ \cdots $};
		\node [style=Slanted Text] (236) at (0, -9) {$ \cdots $};
		\node [style=Slanted Text] (209) at (-5, 12) {};
		\node [style=Slanted Text] (210) at (15, 12) {};
		\node [style=Slanted Text] (211) at (-5, -18) {};
		\node [style=Slanted Text] (212) at (15, -18) {};
		\node [style=Slanted Text] (237) at (30, 5) {};
		\node [style=Slanted Text] (238) at (30, 5) {};
		\node [style=Slanted Text] (239) at (30, 5) {};
		\node [style=Slanted Text] (242) at (15, -8) {};
		\node [style=Slanted Text] (244) at (15, -12) {};
		\node [style=Slanted Text] (246) at (15, -20) {};
		\node [style=Slanted Text] (248) at (15, -4) {$ \cdots $};
		\node [style=Slanted Text] (250) at (15, -16) {$ \cdots $};
		\node [style=Slanted Text] (255) at (20, -2) {$ \scriptstyle{[n-2]; \left( A_1,  [1] \otimes A_2 \times A_3,  , A_{\geq 4} \right) }$};
		\node [style=Slanted Text] (256) at (20, -6) {$ \ddots $};
		\node [style=Slanted Text] (257) at (20, -10) {$ \scriptstyle{[n-2];\left( A_{\leq k-1}, \mathsf{P.R.}\left(A_{k-1}, A_{k},A_{k+1} \right),A_{\geq k+2},A_{\geq k+3} \right)} $};
		\node [style=Slanted Text] (258) at (20, -14) {$ \ddots $};
		\node [style=Slanted Text] (259) at (20, -18) {$ \scriptstyle{[n-2];\left( A_{\leq n-3},  \mathsf{P.R.}\left(A_{n-2}, A_{n-1},A_{n} \right) \right)} $};
		\node [style=Slanted Text] (251) at (10, 5) {};
		\node [style=Slanted Text] (252) at (30, 5) {};
		\node [style=Slanted Text] (253) at (10, -25) {};
		\node [style=Slanted Text] (254) at (30, -25) {};
		\node [style=Slanted Text] (260) at (15, 12) {};
		\node [style=Slanted Text] (261) at (15, 12) {};
		\node [style=Slanted Text] (262) at (15, 12) {};
		\node [style=Slanted Text] (263) at (15, 12) {};
		\node [style=Slanted Text] (264) at (15, 12) {};
		\node [style=Slanted Text] (265) at (15, 12) {};
		\node [style=Slanted Text] (266) at (30, 5) {};
		\node [style=Slanted Text] (267) at (30, 5) {};
		\node [style=Slanted Text] (268) at (30, 5) {};
		\node [style=Slanted Text] (269) at (30, 5) {};
		\node [style=Slanted Text] (274) at (30, -25) {};
		\node [style=Slanted Text] (275) at (45, -2) {};
		\node [style=Slanted Text] (276) at (45, -2) {};
		\node [style=Slanted Text] (277) at (45, -2) {};
		\node [style=Slanted Text] (278) at (40, -9) {$\ddots$};
		\node [style=Slanted Text] (279) at (30, -17) {$\ddots$};
		\node [style=Slanted Text] (280) at (40, -17) {$\ddots$};
		\node [style=Slanted Text] (281) at (30, -19) {$ \ddots $};
		\node [style=Slanted Text] (282) at (40, -15) {$ \ddots $};
		\node [style=Slanted Text] (283) at (30, -27) {};
		\node [style=Slanted Text] (284) at (40, -13) {$ \ddots $};
		\node [style=Slanted Text] (285) at (30, -13) {$ \ddots $};
		\node [style=Slanted Text] (288) at (35, -9) {};
		\node [style=Slanted Text] (289) at (35, -13) {$ \ddots $};
		\node [style=Slanted Text] (290) at (35, -17) {$ \ddots $};
		\node [style=Slanted Text] (291) at (35, -21) {$ \ddots $};
		\node [style=Slanted Text] (292) at (35, -25) {};
		\node [style=Slanted Text] (297) at (45, -2) {};
		\node [style=Slanted Text] (298) at (45, -2) {};
		\node [style=Slanted Text] (299) at (45, -2) {};
		\node [style=Slanted Text] (300) at (45, -2) {};
		\node [style=Slanted Text] (301) at (30, -9) {$ \ddots $};
		\node [style=Slanted Text] (302) at (30, -21) {$\ddots$};
		\node [style=Slanted Text] (303) at (40, -21) {$\ddots$};
		\node [style=Slanted Text] (304) at (30, -25) {$\ddots$};
		\node [style=Slanted Text] (305) at (40, -25) {$\ddots$};
		\node [style=Slanted Text] (293) at (25, -2) {};
		\node [style=Slanted Text] (294) at (45, -2) {};
		\node [style=Slanted Text] (295) at (25, -32) {};
		\node [style=Slanted Text] (296) at (45, -32) {};
		\node [style=Slanted Text] (306) at (45, -2) {};
		\node [style=Slanted Text] (307) at (60, -9) {};
		\node [style=Slanted Text] (308) at (45, -32) {};
		\node [style=Slanted Text] (309) at (45, -32) {};
		\node [style=Slanted Text] (310) at (60, -9) {};
		\node [style=Slanted Text] (311) at (60, -9) {};
		\node [style=Slanted Text] (312) at (60, -9) {};
		\node [style=Slanted Text] (313) at (55, -16) {$\ddots$};
		\node [style=Slanted Text] (314) at (45, -24) {$\ddots$};
		\node [style=Slanted Text] (315) at (55, -24) {$\ddots$};
		\node [style=Slanted Text] (316) at (45, -26) {};
		\node [style=Slanted Text] (317) at (55, -22) {};
		\node [style=Slanted Text] (318) at (45, -34) {};
		\node [style=Slanted Text] (319) at (55, -20) {$ \ddots $};
		\node [style=Slanted Text] (320) at (45, -20) {$ \ddots $};
		\node [style=Slanted Text] (321) at (50, -16) {};
		\node [style=Slanted Text] (322) at (50, -20) {$ \ddots $};
		\node [style=Slanted Text] (323) at (50, -24) {$ [1];\left( \mathsf{P.R.} \left( A_{\leq n} \right) \right) $};
		\node [style=Slanted Text] (324) at (50, -28) {$ \ddots $};
		\node [style=Slanted Text] (325) at (50, -32) {};
		\node [style=Slanted Text] (326) at (60, -9) {};
		\node [style=Slanted Text] (327) at (60, -9) {};
		\node [style=Slanted Text] (328) at (60, -9) {};
		\node [style=Slanted Text] (329) at (60, -9) {};
		\node [style=Slanted Text] (330) at (45, -16) {$ \ddots $};
		\node [style=Slanted Text] (331) at (45, -28) {$\ddots$};
		\node [style=Slanted Text] (332) at (55, -28) {$\ddots$};
		\node [style=Slanted Text] (333) at (45, -32) {$\ddots$};
		\node [style=Slanted Text] (334) at (55, -32) {$\ddots$};
		\node [style=Slanted Text] (339) at (60, -9) {};
		\node [style=Slanted Text] (335) at (40, -9) {};
		\node [style=Slanted Text] (336) at (60, -9) {};
		\node [style=Slanted Text] (337) at (40, -39) {};
		\node [style=Slanted Text] (338) at (60, -39) {};
	\end{pgfonlayer}
	\begin{pgfonlayer}{edgelayer}
		\draw [style=Fronts] (22.center)
			 to (24.center)
			 to (25.center)
			 to (23.center)
			 to cycle;
		\draw [style=Solid arrow] (26) to (5);
		\draw [style=Solid arrow] (26) to (4);
		\draw [style=Solid arrow] (42) to (5);
		\draw [style=Solid arrow] (42) to (6);
		\draw [style=Solid arrow] (43) to (6);
		\draw [style=Solid arrow] (43) to (7);
		\draw [style=Dashed Edges No Fill] (41) to (7);
		\draw [style=Dashed Edges No Fill] (41) to (21);
		\draw [style=Dashed Edges No Fill] (48) to (21);
		\draw [style=Dashed Edges No Fill] (48) to (17);
		\draw [style=Solid arrow] (38) to (17);
		\draw [style=Solid arrow] (38) to (8);
		\draw [style=Solid arrow] (40) to (8);
		\draw [style=Solid arrow] (40) to (20);
		\draw [style=Solid arrow] (46) to (20);
		\draw [style=Solid arrow] (46) to (10);
		\draw [style=Solid arrow] (37) to (10);
		\draw [style=Solid arrow] (37) to (15);
		\draw [style=Dashed Edges No Fill] (39) to (15);
		\draw [style=Dashed Edges No Fill] (39) to (18);
		\draw [style=Dashed Edges No Fill] (47) to (18);
		\draw [style=Dashed Edges No Fill] (47) to (14);
		\draw [style=Solid arrow] (45) to (14);
		\draw [style=Solid arrow] (45) to (13);
		\draw [style=Solid arrow] (44) to (13);
		\draw [style=Solid arrow] (44) to (12);
		\draw [style=Solid arrow] (33) to (12);
		\draw [style=Solid arrow] (33) to (11);
		\draw [style=Fronts] (0.center)
			 to (2.center)
			 to (3.center)
			 to (1.center)
			 to cycle;
		\draw [style=Dotted arrow] (118) to (135);
		\draw [style=Dotted arrow] (118) to (134);
		\draw [style=Dotted arrow] (125) to (135);
		\draw [style=Dotted arrow] (125) to (136);
		\draw [style=Dotted arrow] (126) to (136);
		\draw [style=Dotted arrow] (126) to (137);
		\draw [style=Dotted arrow] (124) to (137);
		\draw [style=Dotted arrow] (124) to (149);
		\draw [style=Dotted arrow] (131) to (149);
		\draw [style=Dotted arrow] (131) to (146);
		\draw [style=Dotted arrow] (121) to (146);
		\draw [style=Dotted arrow] (121) to (138);
		\draw [style=Dotted arrow] (123) to (138);
		\draw [style=Dotted arrow] (123) to (148);
		\draw [style=Dotted arrow] (129) to (148);
		\draw [style=Dotted arrow] (129) to (140);
		\draw [style=Dotted arrow] (120) to (140);
		\draw [style=Dotted arrow] (120) to (145);
		\draw [style=Dotted arrow] (122) to (145);
		\draw [style=Dotted arrow] (122) to (147);
		\draw [style=Dotted arrow] (130) to (147);
		\draw [style=Dotted arrow] (130) to (144);
		\draw [style=Dotted arrow] (128) to (144);
		\draw [style=Dotted arrow] (128) to (143);
		\draw [style=Dotted arrow] (127) to (143);
		\draw [style=Dotted arrow] (127) to (142);
		\draw [style=Dotted arrow] (119) to (142);
		\draw [style=Dotted arrow] (119) to (141);
		\draw [style=Solid arrow] (70) to (135);
		\draw [style=Solid arrow] (72) to (137);
		\draw [style=Solid arrow] (56) to (137);
		\draw [style=Solid arrow] (81) to (146);
		\draw [style=Solid arrow] (65) to (146);
		\draw [style=Solid arrow] (83) to (148);
		\draw [style=Solid arrow] (58) to (148);
		\draw [style=Solid arrow] (64) to (145);
		\draw [style=Solid arrow] (80) to (145);
		\draw [style=Solid arrow] (61) to (142);
		\draw [style=Solid arrow] (63) to (144);
		\draw [style=Solid arrow] (79) to (144);
		\draw [style=Fronts] (49.center)
			 to (51.center)
			 to (52.center)
			 to [in=270, out=90] (50.center)
			 to cycle;
		\draw [style=Solid arrow] (70) to (174);
		\draw [style=Solid arrow] (56) to (174);
		\draw [style=Solid arrow] (72) to (175);
		\draw [style=Solid arrow] (65) to (175);
		\draw [style=Solid arrow] (58) to (176);
		\draw [style=Solid arrow] (65) to (176);
		\draw [style=Solid arrow] (64) to (177);
		\draw [style=Solid arrow] (83) to (177);
		\draw [style=Solid arrow] (63) to (178);
		\draw [style=Solid arrow] (80) to (178);
		\draw [style=Solid arrow] (79) to (179);
		\draw [style=Solid arrow] (61) to (179);
		\draw [style=Fronts] (170.center)
			 to (172.center)
			 to (173.center)
			 to [in=270, out=90] (171.center)
			 to cycle;
		\draw [style=Solid arrow] (221) to (174);
		\draw [style=Solid arrow] (226) to (176);
		\draw [style=Solid arrow] (227) to (176);
		\draw [style=Solid arrow] (228) to (177);
		\draw [style=Solid arrow] (229) to (177);
		\draw [style=Solid arrow] (232) to (179);
		\draw [style=Solid arrow] (233) to (175);
		\draw [style=Solid arrow] (234) to (175);
		\draw [style=Solid arrow] (236) to (178);
		\draw [style=Solid arrow] (235) to (178);
		\draw [style=Fronts] (209.center)
			 to (211.center)
			 to (212.center)
			 to [in=270, out=90] (210.center)
			 to cycle;
		\draw [style=Solid arrow] (225) to (255);
		\draw [style=Solid arrow] (234) to (255);
		\draw [style=Solid arrow] (226) to (256);
		\draw [style=Solid arrow] (233) to (256);
		\draw [style=Solid arrow] (227) to (257);
		\draw [style=Solid arrow] (228) to (257);
		\draw [style=Solid arrow] (229) to (258);
		\draw [style=Solid arrow] (236) to (258);
		\draw [style=Solid arrow] (235) to (259);
		\draw [style=Solid arrow] (232) to (259);
		\draw [style=Fronts] (251.center)
			 to (253.center)
			 to (254.center)
			 to [in=270, out=90] (252.center)
			 to cycle;
		\draw [style=Solid arrow] (301) to (255);
		\draw [style=Solid arrow] (278) to (255);
		\draw [style=Solid arrow] (285) to (256);
		\draw [style=Solid arrow] (284) to (256);
		\draw [style=Solid arrow] (280) to (257);
		\draw [style=Solid arrow] (279) to (257);
		\draw [style=Solid arrow] (303) to (258);
		\draw [style=Solid arrow] (302) to (258);
		\draw [style=Solid arrow] (305) to (259);
		\draw [style=Solid arrow] (304) to (259);
		\draw [style=Fronts] (294.center)
			 to (293.center)
			 to (295.center)
			 to (296.center)
			 to [in=270, out=90] cycle;
		\draw [style=Solid arrow] (281) to (323);
		\draw [style=Solid arrow] (282) to (323);
		\draw [style=Fronts] (336.center)
			 to (335.center)
			 to (337.center)
			 to (338.center)
			 to [in=270, out=90] cycle;
	\end{pgfonlayer}
\end{tikzpicture}
}
}

\caption{\label{fig:Complete-wide-pushout-for-Gray-cylinder}Complete wide-pushout
for the Gray cylinder $\left[1\right]\otimes\left[n\right];\left(A_{1},A_{2},\dots,A_{n}\right)$}
\bigskip{}
\end{minipage}}
\end{figure}
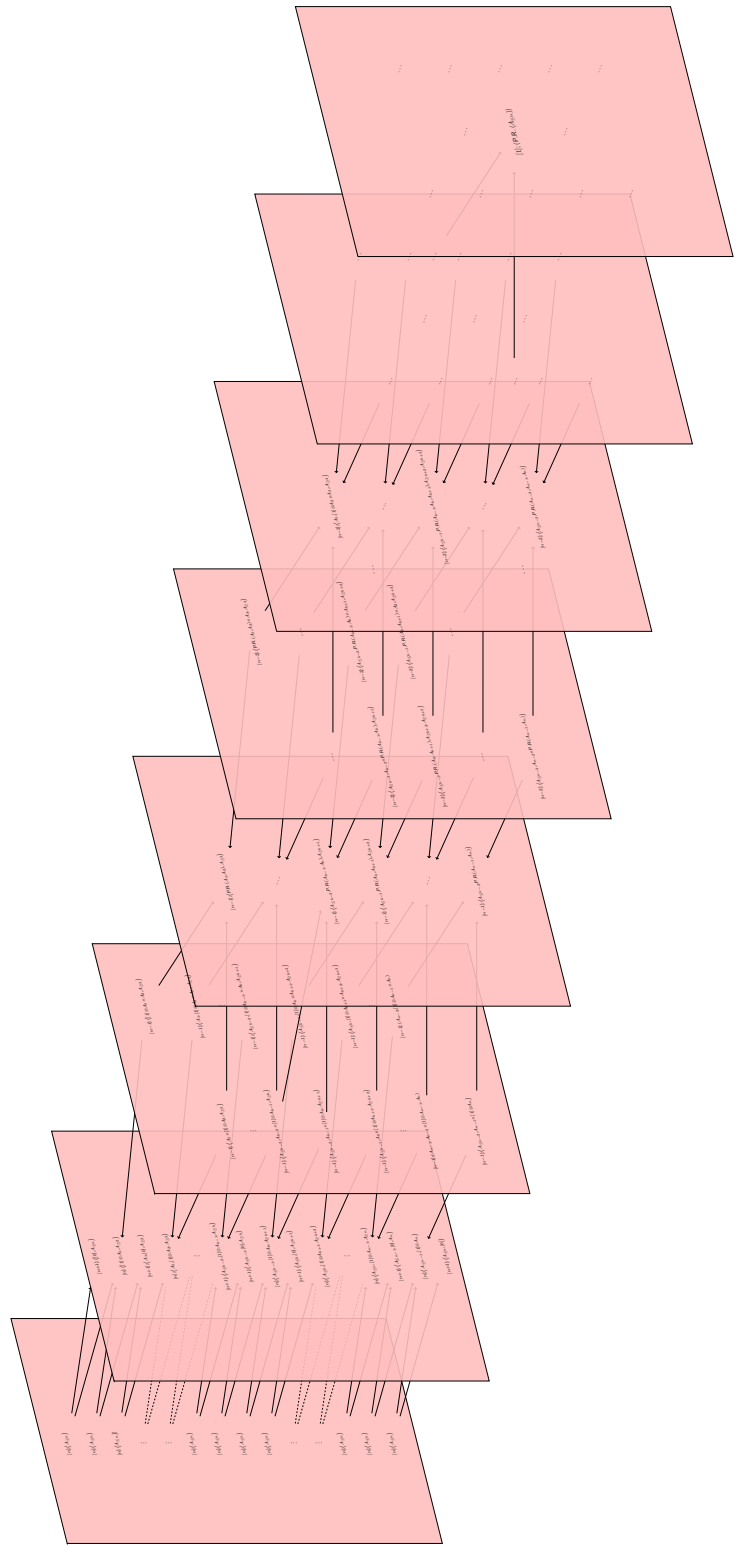
\end{cor}

\begin{proof}
Follows from the prior Lemma characterizing the $\Hom$ $\omega$-categories
of $\left[1\right]\otimes\left[n\right];\left(A_{i}\right)$.
\end{proof}

\subsection{The Gray cylinder of hyperfaces}

The hyperfaces of cells $T$ of $\Theta$ hold a particularly important
place in the Cisinski model categories which $\wh{\Theta}$ admits.
In this section we use the formulae we've developed to describe how
the Gray cylinder $\left[1\right]\otimes\left(\_\right)$ acts on
hyperfaces. Unsurprisingly, this action can be defined recursively
on the height of the hyperface's target.

\subsubsection{Vertical hyperfaces}

The action of $\left[1\right]\otimes\left(\_\right)$ on the vertical
hyperfaces 
\[
\left[n\right];\left(A_{<k},A_{k}^{\p},A_{>k}\right)\longrightarrow\left[n\right];\left(A_{<k},A_{k},A_{>k}\right)
\]
 is easily described in terms of the lax shuffle decomposition.
\begin{cor}
Given a vertical hyperface 
\[
\id;\left(\id_{<k},\nu,\id_{>k}\right):\left[n\right];\left(A_{<k},A_{k}^{\p},A_{>k}\right)\longrightarrow\left[n\right];\left(A_{<k},A_{k},A_{>k}\right)
\]
 the map 
\[
\left[1\right]\otimes\left[n\right];\left(A_{<k},A_{k}^{\p},A_{>k}\right)\longrightarrow\left[1\right]\otimes\left[n\right];\left(A_{<k},A_{k},A_{>k}\right)
\]
is induced by the diagram
\[
\xyR{3pc}\xyC{4.5pc}\xymatrix{\vdots\ar[d]\ar[rr] &  & \vdots\ar[d]\\
\left[n+1\right];\left(A_{<k},\underline{0},A_{k}^{\p},A_{>k}\right)\ar[rr]|-{\id;\left(\id_{<k},\id_{\underline{0}},\nu,\id_{>k}\right)} &  & \left[n+1\right];\left(A_{<k},\underline{0},A_{k},A_{>k}\right)\\
\left[n\right];\left(A_{<k},A_{k}^{\p},A_{>k}\right)\ar[u]|-{d^{k};\left(\id_{<k},\left(!,\id_{k}\right),\id_{>k}\right)}\ar[d]|-{\id;\left(\id_{<k},\left\{ 1\right\} \otimes A_{k}^{\p},\id_{>k}\right)}\ar[rr]|-{\id;\left(\id_{<k},\nu,\id_{>k}\right)} &  & \left[n\right];\left(A_{<k},A_{k},A_{>k}\right)\ar[u]|-{d^{k};\left(\id_{<k},\left(!,\id_{k}\right),\id_{>k}\right)}\ar[d]|-{\id;\left(\id_{<k},\left\{ 1\right\} \otimes A_{k},\id_{>k}\right)}\\
\left[n\right]\left(A_{<k},\left[1\right]\otimes A_{k}^{\p},A_{>K}\right)\ar[rr]|-{\id;\left(\id_{<k},\left[1\right]\otimes\nu,\id_{>k}\right)} &  & \left[n\right];\left(A_{<k},\left[1\right]\otimes A_{k},A_{>k}\right)\\
\left[n\right];\left(A_{<k},A_{k}^{\p},A_{>k}\right)\ar[d]|-{d^{k};\left(\id_{<k},\left(\id_{k},!\right),\id_{>k}\right)}\ar[u]|-{\id;\left(\id_{<k},\left\{ 0\right\} \otimes A_{k}^{\p},\id_{>k}\right)}\ar[rr]|-{\id;\left(\id_{<k},\nu,\id_{>k}\right)} &  & \left[n\right];\left(A_{<k},A_{k}^{\p},A_{>k}\right)\ar[u]|-{\id;\left(\id_{<k},\left\{ 0\right\} \otimes A_{k}^{\p},\id_{>k}\right)}\ar[d]|-{d^{k};\left(\id_{<k},\left(\id_{k},!\right),\id_{>k}\right)}\\
\left[n+1\right];\left(A_{<k},A_{k}^{\p},\underline{0},A_{>k}\right)\ar[rr]|-{\id;\left(\id_{<k},\nu,\id_{\underline{0}},\id_{>k}\right)} &  & \left[n+1\right];\left(A_{<k},A_{k},\underline{0},A_{>k}\right)\\
\vdots\ar[u]\ar[rr] &  & \vdots\ar[u]
}
\]
\end{cor}

\subsubsection{Horizontal hyperfaces}

Above we used the fact that inner and outer \emph{vertical} hyperfaces
admit the same description in terms of the lax shuffle decomposition.
This however is not the case for inner and outer horizontal hyperfaces
- the formulae are subtly different.
\begin{cor}
For outer hyperfaces 
\[
d^{0};\left(\left(!,\id_{1}\right),\id_{>1}\right):\left[n-1\right];\left(A_{0<i\leq n-1}\right)\longrightarrow\left[n\right];\left(\underline{0},A_{0<i\leq n-1}\right)
\]
 and 
\[
d^{n};\left(\id_{<n-1},\left(\id_{n-1},!\right)\right):\left[n-1\right];\left(A_{0<i\leq n-1}\right)\longrightarrow\left[n\right];\left(A_{0<i\leq n-1},\underline{0}\right)
\]
the maps 
\[
\left[1\right]\otimes d^{0};\left(\left(!,\id_{1}\right),\id_{>1}\right):\left[n-1\right];\left(A_{0<i\leq n-1}\right)\longrightarrow\left[n\right];\left(\underline{0},A_{0<i\leq n-1}\right)
\]
and
\[
\left[1\right]\otimes d^{n};\left(\id_{<n-1},\left(\id_{n-1},!\right)\right):\left[n-1\right];\left(A_{0<i\leq n-1}\right)\longrightarrow\left[n\right];\left(A_{0<i\leq n-1},\underline{0}\right)
\]
are induced by the diagrams below, with $\left[1\right]\otimes d^{0};\left(\left(!,\id_{1}\right),\id_{>1}\right)$
induced by the diagram on the left in Figure \ref{fig:Diagrams-inducing-outer horizontal hyperfaces}
and $\left[1\right]\otimes d^{n};\left(\id_{<n-1},\left(\id_{n-1},!\right)\right)$
induced by the diagram on the right in Figure \ref{fig:Diagrams-inducing-outer horizontal hyperfaces}.

\begin{figure}
\noindent\fbox{\begin{minipage}[t]{1\columnwidth - 2\fboxsep - 2\fboxrule}%
\bigskip{}
\adjustbox{scale=.5, center}{
	$$
	\xymatrix{
		& \scriptstyle{\left[n+1\right];\left(\underline{0},\left[0\right],A_{0<i\leq n-1}\right)}
			& \scriptstyle{ \left[n\right];\left(\underline{0},A_{0<i\leq n-1}\right)} 			\ar[r]|{d^{n+1};\left(\mathrm{id}\right)}
				& \scriptstyle{\left[n+1\right];\left(\underline{0},A_{0<i\leq n-1},\left[0\right]\right)} 					\\
		& \scriptstyle{\left[n\right];\left(\left[0\right],A_{0<i\leq n-1}\right)}
		\ar[u]\ar[d]
			& \scriptstyle{\left[n-1\right];\left(A_{0<i\leq n-1}\right)}
			\ar[u]\ar[d]
			\ar[r]|{d^{n};\left(\mathrm{id}\right)}
				& \scriptstyle{\left[n\right];\left(A_{0<i\leq n-1},\left[0\right]\right)}
				\ar[u]\ar[d]
					\\
		& \scriptstyle{\left[n\right];\left(\left[1\right],A_{0<i\leq n-1}\right)}
			& \scriptstyle{\left[n-1\right];\left(\left[1\right]\otimes A_{1},A_{1<i\leq n-1}\right)} 			\ar[r]|{d^{n};\left(\mathrm{id}\right)}
				& \scriptstyle{\left[n\right];\left(\left[1\right],A_{0<i\leq n-1},\left[0\right]\right)} 					\\
		& \scriptstyle{ \left[n\right];\left(\left[0\right],A_{0<i\leq n-1}\right) }
		\ar[u]\ar[d]
			& \scriptstyle{ \left[n-1\right];\left(A_{0<i\leq n-1}\right)}
			\ar[u]\ar[d]
			\ar[r]|{d^{n};\left(\mathrm{id}\right)}
				& \scriptstyle{\left[n\right];\left(A_{0<i\leq n-1},\left[0\right]\right)}
				\ar[u]\ar[d]
					\\
		\scriptstyle{ \left[n\right];\left(\underline{0},A_{0<i\leq n-1}\right)} 		\ar[r]|{d^{0};\left(\mathrm{id}\right)}
			& \scriptstyle{ \left[n+1\right];\left(\left[0\right],\underline{0},A_{0<i\leq n-1}\right)} 				& \scriptstyle{ \left[n\right];\left(A_{1},\underline{0},A_{1<i\leq n-1}\right)}
				\ar[r]|{d^{n+1};\left(\mathrm{id}\right)}
					& \scriptstyle{ \left[n+1\right];\left(A_{1},\underline{0},A_{1<i\leq n-1},\left[0\right]\right)}
						\\
		\scriptstyle{ \left[n-1\right];\left(A_{0<i\leq n-1}\right)}
		\ar[u]\ar[d]
		\ar[r]|{d^{0};\left(\mathrm{id}\right)}
			& \scriptstyle{\left[n\right];\left(\left[0\right],A_{1},A_{2<i\leq n-1}\right)}
			\ar[u]\ar[d]
				& \vdots
				\ar[u]\ar[d]\ar[r]
					& \scriptstyle{\vdots}
					\ar[u]\ar[d]
						\\
		\scriptstyle{\left[n-1\right];\left(\left[1\right]\otimes A_{1},A_{1<i\leq n-1}\right)} 		\ar[r]|{d^{0};\left(\mathrm{id}\right)}
			& \scriptstyle{\left[n\right];\left(\left[0\right],\left[1\right]\otimes A_{1},A_{2<i\leq n-1}\right)} 				& \scriptstyle{\left[n\right];\left(A_{0<i<n-1},\underline{0},A_{n-1}\right)}
				\ar[r]|{d^{n+1};\left(\mathrm{id}\right)}
					& \scriptstyle{\left[n+1\right];\left(A_{0<i<n-1},\underline{0},A_{n-1},\left[0\right]\right)} 						\\
		\scriptstyle{\left[n-1\right];\left(A_{0<i\leq n-1}\right)}
		\ar[u]\ar[d]
		\ar[r]|{d^{0};\left(\mathrm{id}\right)}
			& \scriptstyle{\left[n\right];\left(\left[0\right],A_{1},A_{2<i\leq n-1}\right)}
			\ar[u]\ar[d]
				& \scriptstyle{\left[n-1\right];\left(A_{0<i\leq n-1}\right)}
				\ar[u]\ar[d]
				\ar[r]|{d^{n};\left(\mathrm{id}\right)}
					& \scriptstyle{\left[n\right];\left(A_{0<i\leq n-1},\left[0\right]\right)} 					\ar[u]\ar[d] 						\\ 		\scriptstyle{\left[n\right];\left(A_{1},\underline{0},A_{1<i\leq n-1}\right)} 		\ar[r]|{d^{0};\left(\mathrm{id}\right)}
			& \scriptstyle{\left[n\right];\left(\left[0\right],A_{1},\underline{0},A_{1<i\leq n-1}\right)}	 				& \scriptstyle{ \left[n-1\right];\left(A_{0<i<n-1},\left[1\right]\otimes A_{n-1}\right)} 				\ar[r]|{d^{n};\left(\mathrm{id}\right)}
					& \scriptstyle{\left[n\right];\left(A_{1<i<n-1},\left[1\right]\otimes A_{n-1},\left[0\right]\right)}
						\\
		\scriptstyle{\vdots}
		\ar[u]\ar[d]\ar[r]
			& \scriptstyle{\vdots}
			\ar[u]\ar[d]
				& \scriptstyle{\left[n-1\right];\left(A_{0<i\leq n-1}\right)}
				\ar[u]\ar[d]
				\ar[r]|{d^{n};\left(\mathrm{id}\right)}
					& \scriptstyle{\left[n\right];\left(A_{0<i\leq n-1},\left[0\right]\right)} 					\ar[u]\ar[d] 						\\ 		\scriptstyle{\left[n\right];\left(A_{0<i<n-1},\underline{0},A_{n-1}\right)} 		\ar[r]|{d^{0};\left(\mathrm{id}\right)}
			& \scriptstyle{\left[n+1\right];\left(\left[0\right],A_{0<i<n-1},\underline{0},A_{n-1}\right)} 				& \scriptstyle{\left[n\right];\left(A_{0<i\leq n-1},\underline{0}\right)} 				\ar[r]|{d^{n+1};\left(\mathrm{id}\right)}
					& \scriptstyle{\left[n+1\right];\left(\left[0\right],A_{0<i<n-1},\underline{0},A_{n-1}\right)} 						\\
		\scriptstyle{\left[n-1\right];\left(A_{0<i\leq n-1}\right)}
		\ar[u]\ar[d]
		\ar[r]|{d^{0};\left(\mathrm{id}\right)}
			& \scriptstyle{\left[n\right];\left(\left[0\right],A_{0<i\leq n-1}\right)}
			\ar[u]\ar[d]
				&
					& \scriptstyle{\left[n\right];\left(\left[0\right],A_{0<i\leq n-1}\right)} 					\ar[u]\ar[d]
						\\
		\scriptstyle{\left[n-1\right];\left(A_{0<i<n-1},\left[1\right]\otimes A_{n-1}\right)} 		\ar[r]|{d^{0};\left(\mathrm{id}\right)}
			& \scriptstyle{\left[n\right];\left(\left[0\right],A_{0<i<n-1},\left[1\right]\otimes A_{n-1}\right)} 				&
 					& \scriptstyle{\left[n\right];\left(\left[0\right],A_{0<i<n-1},\left[1\right]\otimes A_{n-1}\right)} 						\\
		\scriptstyle{\left[n-1\right];\left(A_{0<i\leq n-1}\right)}
		\ar[u]\ar[d]
		\ar[r]|{d^{0};\left(\mathrm{id}\right)}
			& \scriptstyle{\left[n\right];\left(\left[0\right],A_{0<i\leq n-1}\right)}
			\ar[u]\ar[d]
				&
					& \scriptstyle{\left[n\right];\left(\left[0\right],A_{0<i\leq n-1}\right)} 					\ar[u]\ar[d]
						\\
		\scriptstyle{\left[n\right];\left(A_{0<i\leq n-1},\underline{0}\right)} 		\ar[r]|{d^{0};\left(\mathrm{id}\right)}
			& \scriptstyle{\left[n+1\right];\left(\left[0\right],A_{0<i\leq n-1},\underline{0}\right)} 				&
 					& \scriptstyle{\left[n+1\right];\left(\left[0\right],A_{0<i\leq n-1},\underline{0}\right)} 		}
	$$
}\bigskip{}

\caption{\label{fig:Diagrams-inducing-outer horizontal hyperfaces}Diagrams
inducing the Gray cylinders of outer horizontal hyperfaces}
\bigskip{}
\end{minipage}}
\end{figure}

\end{cor}

Lastly we may attend to inner horizontal hyperfaces.
\begin{cor}
Let 
\[
d^{k};\left(\id_{0<i<k},\left(\id,!\right),\id_{k<i<n}\right):\left[n-1\right];\left(A_{0<i<n}\right)\longrightarrow\left[n\right];\left(A_{0<i<k},A_{k},\left[0\right],A_{k<i<n}\right)
\]
be an inner hyperface. Then the map
\[
\left[1\right]\otimes d^{k};\left(\id_{0<i<k},\left(\id,!\right),\id_{k<i<n}\right):\left[1\right]\otimes:\left[n-1\right];\left(A_{0<i<n}\right)\longrightarrow\left[1\right]\otimes\left[n\right];\left(A_{0<i<k},A_{k},\left[0\right],A_{k<i<n}\right)
\]
 is induced by the diagram in Figure \ref{fig:Diagram-inducing-the-gray-cylinder-inner-hyperfaces}
(in which we leave the denotation of all but the simplicial aspects
of maps implicit).

\begin{figure}
\noindent\fbox{\begin{minipage}[t]{1\columnwidth - 2\fboxsep - 2\fboxrule}%
\bigskip{}

\adjustbox{scale=.5,center}{
	$$
	\xymatrix{
		{\scriptstyle \left[n\right];\left(\underline{0},A_{0<i<n}\right)}
		\ar[rr]|{d^{k+1}}
			&
				& {\scriptstyle \left[n+1\right];\left(\underline{0},A_{0<i\leq k},\left[0\right],A_{k<i<n}\right)}
					\\
		{\scriptstyle \left[n-1\right];\left(A_{0<i<n}\right)}
		\ar[u]
		\ar[d]
			&
				& {\scriptstyle \left[n\right];\left(A_{0<i\leq k},\left[0\right],A_{k<i<n}\right)}
				\ar[u]
				\ar[d]
					\\
		{\scriptstyle \left[n-1\right];\left(\left[1\right]\otimes A_{1},A_{1<i<n}\right)}
		\ar[rr]|{d^{k}}
			&
				& {\scriptstyle \left[n\right];\left(\left[1\right]\otimes A_{1},A_{1<i\leq k},\left[0\right],A_{k<i<n}\right)}
					\\
		{\scriptstyle \left[n-1\right];\left(A_{0<i<n}\right)}
		\ar[u] 		\ar[d] 		& 			& {\scriptstyle \left[n\right];\left(A_{0<i\leq k},\left[0\right],A_{k<i<n}\right)} 			\ar[u]
			\ar[d]
				\\
		{\scriptstyle \left[n\right];\left(A_{1},\underline{0},A_{1<i<n}\right)}
		\ar[rr]|{d^{k+1}}
			&
				& {\scriptstyle \left[n+1\right];\left(A_{1},\underline{0},A_{1<i\leq n-1},\left[0\right]\right)}
					\\
		%placeholder
			&
				& {\scriptstyle \vdots}
				\ar[u]
				\ar[d]
					\\
	{\scriptstyle \vdots}
	\ar[uu]
	\ar[dd]
		&
			& {\scriptstyle \left[n+1\right];\left(A_{0<i<k},\underline{0},A_{k},\left[0\right],A_{k<i<n}\right)}
				\\
	%placeholder
		&
			& {\scriptstyle \left[n\right];\left(A_{0<i<k},A_{k},\left[0\right],A_{k<i<n}\right)}
			\ar[u]
			\ar[d]
				\\ 	{\scriptstyle \left[n\right];\left(A_{0<i<k},\underline{0},A_{k},A_{k<i<n}\right)}
	\ar[uurr]|{d^{k+1}}
		&
			& {\scriptstyle \left[n\right];\left(A_{0<i<k},\left[1\right]\otimes A_{k},\left[0\right],A_{k<i<n}\right)}
				\\
	{\scriptstyle \left[n-1\right];\left(A_{0<i<n}\right)}
	\ar[u]
	\ar[d] 		& {\scriptstyle \left[n-1\right];\left(A_{0<i<k},\left[1\right]\otimes A_{k}\times\left[0\right],A_{k<i<n}\right)} 		\ar[d] 		\ar[ur]|{d^{k}}
			& {\scriptstyle \left[n\right];\left(A_{0<i<k},A_{k},\left[0\right],A_{k<i<n}\right)}
			\ar[u]
			\ar[d]
				\\
	{\scriptstyle \left[n-1\right];\left(A_{0<i<k},\left[1\right]\otimes A_{k},A_{k<i<n}\right)}
	\ar[r]
		& {\scriptstyle \left[n-1\right];\left(A_{0<i<k},\mathsf{P.R.}\left(A_{k},\left[0\right]\right),A_{k<i<n}\right)}
			& {\scriptstyle \left[n+1\right];\left(A_{0<i<k},A_{k},\underline{0},\left[0\right],A_{k<i<n}\right)}
				\\
	{\scriptstyle \left[n-1\right];\left(A_{0<i<n}\right)}
	\ar[u]
	\ar[d]
		& {\scriptstyle \left[n-1\right];\left(A_{0<i<k},A_{k}\times\left[1\right]\otimes\left[0\right],A_{k<i<n}\right)}
		\ar[u]
		\ar[dr]|{d^{k}}
			& {\scriptstyle \left[n\right];\left(A_{0<i<k},A_{k},\left[0\right],A_{k<i<n}\right)}
			\ar[u]
			\ar[d]
			\\
	{\scriptstyle \left[n\right];\left(A_{0<i<k},A_{k},\underline{0},A_{k<i<n}\right)}
	\ar[ddrr]|{d^{k}}
		&
			& {\scriptstyle \left[n\right];\left(A_{0<i<k},A_{k},\left[1\right]\otimes\left[0\right],A_{k<i<n}\right)}
				\\
	%placeholder
		&
			& {\scriptstyle \left[n\right];\left(A_{0<i<k},A_{k},\left[0\right],A_{k<i<n}\right)}
			\ar[u]
			\ar[d]
				\\
	{\scriptstyle \vdots}
	\ar[dd]
	\ar[uu]
		&
			& {\scriptstyle \left[n+1\right];\left(A_{0<i<k},A_{k},\left[0\right],\underline{0},A_{k<i<n}\right)}
			\\
	%palceholder
		&
			& {\scriptstyle \vdots}
			\ar[u]
			\ar[d]
			\\
	{\scriptstyle \left[n\right];\left(A_{0<i<n-1},\underline{0},A_{n}\right)}
	\ar[rr]|{d^{k}}
		&
			& {\scriptstyle \left[n+1\right];\left(A_{0<i<k},A_{k},\left[0\right],A_{k<i<n-1},\underline{0},A_{n-1}\right)}
			\\
	{\scriptstyle \left[n-1\right];\left(A_{0<i<n}\right)}
	\ar[u]
	\ar[d]
		&
			& {\scriptstyle \left[n\right];\left(A_{0<i<k},A_{k},\left[0\right],A_{k<i<n}\right)}
			\ar[u]
			\ar[d]
				\\
	{\scriptstyle \left[n-1\right];\left(A_{0<i<n-1},\left[1\right]\otimes A_{n-1}\right)}
	\ar[rr]|{d^{k}}
		&
			&
				{\scriptstyle \left[n\right];\left(\left[0\right],A_{0<i<n-1},\left[1\right]\otimes A_{n-1}\right)}
					\\
	{\scriptstyle \left[n-1\right];\left(A_{0<i<n}\right)}
	\ar[u]
	\ar[d]
		&
			& {\scriptstyle \left[n\right];\left(A_{0<i<k},A_{k},\left[0\right],A_{k<i<n}\right)}
			\ar[u]
			\ar[d]
				\\
	{\scriptstyle \left[n\right];\left(A_{0<i<n},\underline{0}\right)}
	\ar[rr]|{d^{k}}
		&
			& {\scriptstyle \left[n+1\right];\left(A_{0<i<k},A_{k},\left[0\right],A_{k<i<n},\underline{0}\right)}
	}
	$$
}\bigskip{}

\caption{\label{fig:Diagram-inducing-the-gray-cylinder-inner-hyperfaces}Diagram
inducing the Gray cylinder of an inner hyperface}
\bigskip{}
\end{minipage}}
\end{figure}
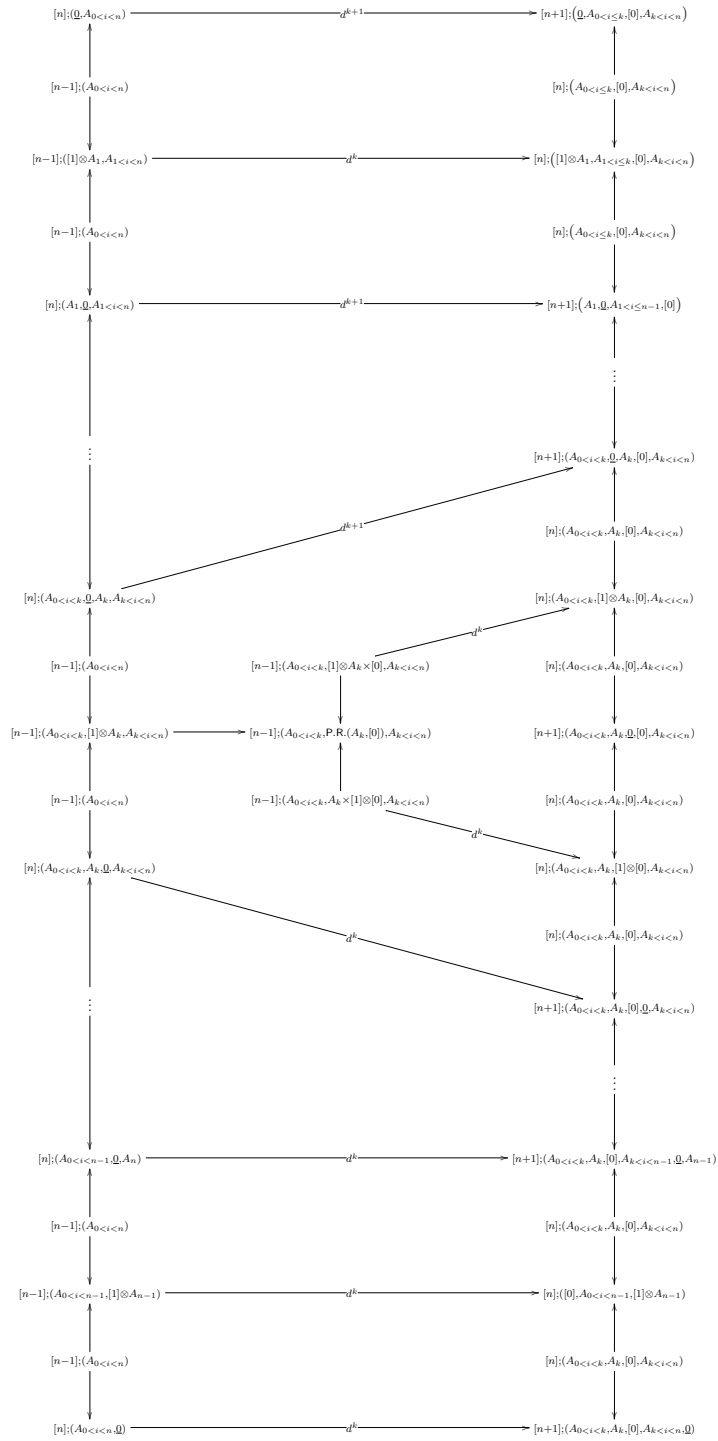

\end{cor}

\begin{rem}
The case where we replace $\left(\id,!\right)$ :$\left(A_{k}\right)\longrightarrow\left(A_{k},\left[0\right]\right)$
with $\left(!,\id\right):\left(A_{k}\right)\longrightarrow\left(\left[0\right],A_{k}\right)$
is nearly the same; we leave it to the reader to make the change themselves
where required.
\end{rem}

\section{The Cartesian-Gray-Shift span and the Gray Cylinder}

We can now define an important span 
\[
\vcenter{\vbox{\xyR{3pc}\xyC{3pc}\xymatrix{ & \left[1\right]\otimes\left(\_\right)\ar[dl]_{\kappa}\ar[dr]^{\sigma}\\
\left[1\right]\times\left(\_\right) &  & \left[1\right];\left(\_\right)
}
}}
\]
of functors $\Theta\longrightarrow\wh{\Theta}$, again by recursion
on height.
\begin{defn}
Let 
\[
\kappa_{\left[0\right]}:\underbrace{\left[1\right]\otimes\left[0\right]}_{\left[1\right]}\longrightarrow\underbrace{\left[1\right]\times\left[0\right]}_{\left[1\right]}
\]
be the identity and let 
\[
\sigma_{\left[0\right]}:\underbrace{\left[1\right]\otimes\left[0\right]}_{\left[1\right]}\longrightarrow\underbrace{\left[1\right];\left[0\right]}_{\left[1\right]}
\]
be the identity as well. Then, by recursion on the height of cells,
we define the maps $\kappa_{T}$ and $\sigma_{T}$ as follows. 

Let $n\in\N$ and $j\in\left[n\right]$ and define the map $\sf{split}_{\left[n\right]}^{j}$
by the following expression.
\[
\vcenter{\vbox{\xyR{0pc}\xyC{3pc}\xymatrix{\mathsf{split}_{\left[n\right]}^{j}:\left[n\right]\ar[r] & \left[1\right]\\
i<j\ar@{|->}[r] & 0\\
i\geq j\ar@{|->}[r] & 1
}
}}
\]
For all $\left[n\right];\left(A_{1},\dots,A_{n}\right)$ of $\Theta$,
the following diagram commutes (proven below as Proposition \ref{prop:General Case GRay Span}).
\begin{equation}
\xyR{3pc}\xyC{3pc}\xymatrix{{\scriptstyle \left[n+1\right];\left(0,A_{1},A_{2},\dots,A_{n}\right)} &  & {\scriptstyle \left[n+1\right];\left(0,A_{1},A_{2},\dots,A_{n}\right)}\ar[ll]|-{\scriptscriptstyle \id}\ar[rr]|-{\scriptscriptstyle \sf{split}_{\left[n+1\right]}^{1};\left(!\right)} &  & {\scriptstyle \left[1\right];\left(0\right)}\ar[dd]|-{\scriptstyle \id;\left(\left\{ 1\right\} \right)}\\
 &  & {\scriptstyle \left[n\right];\left(A_{1},A_{2},\dots,A_{n}\right)}\ar[u]|-{\scriptstyle d^{1};\left(\left(!,\id\right),\id,\dots,\id\right)}\ar[d]|-{\scriptstyle \id;\left(\left\{ 1\right\} \otimes A_{1},\id,\dots,\id\right)}\\
{\scriptstyle \left[n\right];\left(A_{1},A_{2},\dots,A_{n}\right)}\ar[uu]|-{\scriptscriptstyle d^{1};\left(\left(!,\id\right),\id,\dots,\id\right)}\ar[dd]|-{\scriptscriptstyle d^{1};\left(\left(\id,!\right),\id,\dots,\id\right)} &  & {\scriptstyle \left[n\right];\left(\left[1\right]\otimes A_{1},A_{2},\dots,A_{n}\right)}\ar[ll]|-{\scriptstyle \id;\left(\pr_{2}\circ\kappa_{A_{k}},\id,\dots,\id\right)}\ar[rr]|-{\scriptstyle {\scriptscriptstyle \sf{split}_{\left[n+1\right]}^{1};\left(\sigma_{A_{1}}\right)}} &  & {\scriptstyle \left[1\right];\left[1\right];\left(A_{1}\right)}\\
 & \mathrm{} & {\scriptstyle \left[n\right];\left(A_{1},A_{2},\dots,A_{n}\right)}\ar[u]|-{\scriptstyle \id;\left(\left\{ 0\right\} \otimes A_{1},!,\dots,!\right)}\ar[d]|-{\scriptstyle d^{1};\left(\left(\id,!\right),\id,\dots,\id\right)}\\
{\scriptstyle \left[n+1\right];\left(A_{1},0,A_{2},\dots,A_{n}\right)}\ar[d] &  & {\scriptstyle \left[n+1\right];\left(A_{1},0,A_{2},\dots,A_{n}\right)}\ar[ll]|-{\scriptstyle \id}\ar[rr]|-{\scriptstyle {\scriptscriptstyle \sf{split}_{\left[n+1\right]}^{2};\left(!\right)}}\ar[d] &  & {\scriptstyle \left[1\right];\left[0\right]}\ar[uu]|-{\scriptscriptstyle \id;\left(\left\{ 0\right\} \right)}\ar[d]\\
{\scriptstyle \vdots} &  & {\scriptstyle \vdots} &  & \vdots\\
{\scriptstyle \left[n+1\right];\left(A_{1},\dots,A_{n-1},0,A_{n}\right)}\ar[u] &  & {\scriptstyle \left[n+1\right];\left(A_{1},\dots,A_{n-1},0,A_{n}\right)}\ar[ll]|-{\id}\ar[rr]|-{\scriptscriptstyle \sf{split}_{\left[n+1\right]}^{n};\left(!\right)}\ar[u] &  & \ar[u]{\scriptstyle \left[1\right];\left(0\right)}\ar[dd]|-{\scriptscriptstyle \id;\left(\left\{ 1\right\} \right)}\\
 & \mathrm{} & {\scriptstyle \left[n\right];\left(A_{1},\dots,A_{n-1},A_{n}\right)}\ar[u]|-{\scriptstyle d^{n};\left(\id,\dots,\id,\left(!,\id\right)\right)}\ar[d]|-{\scriptstyle \id;\left(\id,\dots,\id,\left\{ 1\right\} \otimes A_{n}\right)}\\
{\scriptstyle \left[n\right];\left(A_{1},\dots,A_{n-1}A_{n}\right)}\ar[uu]|-{\scriptscriptstyle d^{n};\left(\id,\dots,\id,\left(!,\id\right)\right)}\ar[dd]|-{\scriptscriptstyle d^{n};\left(\id,\dots,\id,\left(\id,!\right)\right)} &  & {\scriptstyle \left[n\right];\left(A_{1},\dots,A_{n-1},\left[1\right]\otimes A_{n}\right)}\ar[ll]|-{\scriptstyle \id;\left(\id,\dots,\id,\pr_{2}\circ\kappa_{A_{n}}\right)}\ar[rr]|-{\scriptscriptstyle \sf{split}_{\left[n+1\right]}^{n+1};\left(\sigma_{A_{n}}\right)} &  & {\scriptstyle \left[1\right];\left[1\right];\left(A_{n}\right)}\\
 & \mathrm{} & {\scriptstyle \left[n\right];\left(A_{1},\dots,A_{n}\right)}\ar[u]|-{\scriptscriptstyle \id;\left(\left\{ 0\right\} \otimes A_{1},!,\dots,!\right)}\ar[d]|-{\scriptscriptstyle d^{n};\left(\id,\dots,\id,\left(\id,!\right)\right)}\\
{\scriptstyle \left[n+1\right];\left(A_{1},\dots,A_{n-1},A_{n},0\right)} &  & {\scriptstyle \left[n+1\right];\left(A_{1},\dots,A_{n-1},A_{n},0\right)}\ar[ll]|-{\scriptscriptstyle \id}\ar[rr]|-{\scriptscriptstyle \sf{split}_{\left[n+1\right]}^{n+1};\left(!\right)} &  & {\scriptstyle \left[1\right];\left[0\right]}\ar[uu]|-{\scriptscriptstyle \id;\left(\left\{ 0\right\} \right)}
}
\label{eq: general case gray span diagram-1-1}
\end{equation}

Moreover, since:
\begin{itemize}
\item the colimit of the first column is $\left[1\right]\times\left[n\right];\left(A_{1},A_{2},\dots,A_{n}\right)$;
\item the colimit of the second column is $\left[1\right]\otimes\left[n\right];\left(A_{1},A_{2},\dots,A_{n}\right)$;
and
\item the colimit on the right is the globular sum decomposition for $\left[1\right];\left[n\right];\left(A_{1},A_{2},\dots,A_{n}\right)$;
\end{itemize}
this diagram induces maps $\kappa_{\left[n\right];\left(A_{i}\right)}$
and $\sigma_{\left[n\right];\left(A_{i}\right)}$ (the solid arrows)
for which the diagram (solid and dashed arrows) commutes as follows.
\begin{equation}
\xyR{3pc}\xyC{9pc}\xymatrix{ & {\scriptstyle \left[n\right];\left(A_{1},A_{2},\dots,A_{n}\right)}\ar@{-->}[dl]|-{\scriptstyle \left\{ 1\right\} \times\left[n\right];\left(A_{1},A_{2},\dots,A_{n}\right)}\ar@{-->}[d]|-{\scriptstyle \left\{ 1\right\} \otimes\left[n\right];\left(A_{1},A_{2},\dots,A_{n}\right)}\ar@{-->}[dr]|-{\scriptstyle \left\{ 1\right\} }\\
{\scriptstyle \left[1\right]\times\left[n\right];\left(A_{1},A_{2},\dots,A_{n}\right)} & {\scriptstyle \left[1\right]\otimes\left[n\right];\left(A_{1},A_{2},\dots,A_{n}\right)}\ar[l]|-{\kappa_{\left[n\right];\left(A_{1},A_{2},\dots,A_{n}\right)}}\ar[r]|-{\sigma_{\left[n\right];\left(A_{1},A_{2},\dots,A_{n}\right)}} & {\scriptstyle \left[1\right];\left[n\right];\left(A_{1},A_{2},\dots,A_{n}\right)}\\
 & {\scriptstyle \left[n\right];\left(A_{1},A_{2},\dots,A_{n}\right)}\ar@{-->}[ul]|-{\scriptstyle \left\{ 0\right\} \times\left[n\right];\left(A_{1},A_{2},\dots,A_{n}\right)}\ar@{-->}[u]|-{\scriptstyle \left\{ 0\right\} \otimes\left[n\right];\left(A_{1},A_{2},\dots,A_{n}\right)}\ar@{-->}[ur]|-{\scriptstyle \left\{ 0\right\} }
}
\label{eq: general case folding diamond-1-1}
\end{equation}
\end{defn}

\begin{lem}
\label{lem:Simplicial Case Gray Span}For all $n\in\N$, the diagram
\begin{equation}
\xyR{3pc}\xyC{3pc}\xymatrix{{\scriptstyle \left[n+1\right]} &  & {\scriptstyle \left[n+1\right]}\ar[ll]|-{\id}\ar[rr]|-{\scriptscriptstyle \sf{split}_{\left[n+1\right]}^{1};\left(!\right)} &  & {\scriptstyle \left[1\right]}\ar[dd]|-{\scriptscriptstyle \id;\left(\left\{ 1\right\} \right)}\\
 &  & {\scriptstyle \left[n\right]}\ar[u]|-{d^{1}}\ar[d]|-{\scriptscriptstyle \id;\left(\left\{ 1\right\} \otimes0,!,\dots,!\right)}\\
{\scriptstyle \left[n\right]}\ar[uu]|-{\scriptscriptstyle d^{1};\left(\left(!,\id\right),\id,\dots,\id\right)}\ar[dd]|-{\scriptscriptstyle d^{1};\left(\left(\id,!\right),\id,\dots,\id\right)} &  & {\scriptstyle \left[n\right];\left(\left[1\right]\otimes0,0,\dots,0\right)}\ar[ll]|-{\scriptscriptstyle \id;\left(\pr_{2}\circ\kappa_{0},\id,\dots,\id\right)}\ar[rr]|-{\scriptscriptstyle \sf{split}_{\left[n+1\right]}^{1};\left(\sigma_{\left[0\right]}\right)} &  & {\scriptstyle \left[1\right];\left[1\right]}\\
 &  & {\scriptstyle \left[n\right]}\ar[u]|-{\scriptscriptstyle \id;\left(\left\{ 0\right\} \otimes0,!,\dots,!\right)}\ar[d]|-{\scriptscriptstyle d^{1}}\\
{\scriptstyle \left[n+1\right]}\ar[d] &  & {\scriptstyle \left[n+1\right]}\ar[ll]|-{\id}\ar[rr]|-{\scriptscriptstyle \sf{split}_{\left[n+1\right]}^{2};\left(!\right)}\ar[d] &  & {\scriptstyle \left[1\right]}\ar[uu]|-{\scriptscriptstyle \id;\left(\left\{ 0\right\} \right)}\ar[d]\\
{\scriptstyle \vdots} &  & {\scriptstyle \vdots} &  & \vdots\\
{\scriptstyle \left[n+1\right]}\ar[u] &  & {\scriptstyle \left[n+1\right]}\ar[ll]|-{\id}\ar[rr]|-{\scriptscriptstyle \sf{split}_{\left[n+1\right]}^{n};\left(!\right)}\ar[u] &  & {\scriptstyle \left[1\right]}\ar[dd]|-{\scriptscriptstyle \id;\left(\left\{ 1\right\} \right)}\ar[u]\\
 &  & {\scriptstyle \left[n\right]}\ar[d]|-{\scriptscriptstyle \id;\left(!,\dots,!,\left\{ 1\right\} \otimes0\right)}\ar[u]|-{d^{n}}\\
{\scriptstyle \left[n\right]}\ar[uu]|-{\scriptscriptstyle d^{n};\left(\id,\dots,\id,\left(!,\id\right)\right)}\ar[dd]|-{\scriptscriptstyle d^{n}:\left(\id,\dots,\id,\left(\id,!\right)\right)} &  & {\scriptstyle \left[n\right];\left(0,\dots,0,\left[1\right]\otimes0\right)}\ar[ll]|-{\scriptscriptstyle \id;\left(\id,\dots,\id,\pr_{2}\circ\kappa_{0}\right)}\ar[rr]|-{\scriptscriptstyle \sf{split}_{\left[n+1\right]}^{n+1};\left(\sigma_{A_{n}}\right)} &  & {\scriptstyle \left[1\right];\left[1\right]}\\
 &  & {\scriptstyle \left[n\right]}\ar[d]|-{\scriptscriptstyle d^{n}}\ar[u]|-{\scriptscriptstyle \id;\left(\id,\dots,\id,\left\{ 0\right\} \otimes0\right)}\\
{\scriptstyle \left[n+1\right]} &  & {\scriptstyle \left[n+1\right]}\ar[rr]|-{\scriptscriptstyle \sf{split}_{\left[n+1\right]}^{n+1};\left(!\right)}\ar[ll]|-{\scriptscriptstyle \id} &  & {\scriptstyle \left[1\right]}\ar[uu]|-{\scriptscriptstyle \id;\left(\left\{ 0\right\} \right)}
}
\label{eq: simplicial case of gray span}
\end{equation}
commutes. Moreover, since:
\begin{itemize}
\item the colimit of the first column is $\left[1\right]\times\left[n\right]$;
\item the colimit of the second column is $\left[1\right]\otimes\left[n\right]$;
and
\item the colimit on the right is the globular sum decomposition for $\left[1\right];\left[n\right]$;
\end{itemize}
this diagram induces maps (the solid arrows) for which the diagram
(solid and dashed arrows) commutes as follows.
\begin{equation}
\xyR{3pc}\xyC{9pc}\xymatrix{ & {\scriptstyle \left[n\right]}\ar@{-->}[dl]|-{\scriptstyle \left\{ 1\right\} \times\left[n\right]}\ar@{-->}[d]|-{\scriptstyle \left\{ 1\right\} \otimes\left[n\right]}\ar@{-->}[dr]|-{\scriptstyle \left\{ 1\right\} }\\
{\scriptstyle \left[1\right]\times\left[n\right]} & {\scriptstyle \left[1\right]\otimes\left[n\right]}\ar[l]|-{\kappa_{\left[n\right]}}\ar[r]|-{\scriptstyle \sigma_{\left[n\right]}} & {\scriptstyle \left[1\right];\left[n\right]}\\
 & {\scriptstyle \left[n\right]}\ar@{-->}[ul]|-{\scriptstyle \left\{ 0\right\} \times\left[n\right]}\ar@{-->}[u]|-{\scriptstyle \left\{ 0\right\} \otimes\left[n\right]}\ar@{-->}[ur]|-{\scriptstyle \left\{ 0\right\} }
}
\label{eq: simplicial case folding diamond}
\end{equation}
\end{lem}

\begin{proof}
To show that Diagram \ref{eq: simplicial case of gray span} commutes
it suffices to check the commutativity of the squares (for $1\leq k\leq n$)
of four sorts; the two sorts
\[
\xyR{3pc}\xyC{3pc}\xymatrix{{\scriptstyle \left[n+1\right]} &  & {\scriptstyle \left[n+1\right]}\ar[ll]|-{\scriptstyle \id}\\
 & \mathrm{(I)} & {\scriptstyle \left[n\right]}\ar[u]|-{\scriptstyle d^{k}}\ar[d]|-{\scriptstyle \id;\left(!,\dots,!,\left\{ 1\right\} \otimes0,!,\dots,!\right)}\\
{\scriptstyle \left[n\right]}\ar[uu]|-{d^{k}}\ar[dd]|-{d^{k}} &  & {\scriptstyle \left[n\right];\left(0,\dots,0,\left[1\right]\otimes0,0,\dots,0\right)}\ar[ll]|-{\scriptstyle \id}\\
 & \mathrm{(II)} & {\scriptstyle \left[n\right]}\ar[u]|-{\scriptstyle \id;\left(!,\dots,!,\left\{ 0\right\} \otimes0,!,\dots,!\right)}\ar[d]|-{\scriptstyle d^{k}}\\
{\scriptstyle \left[n+1\right]} &  & {\scriptstyle \left[n+1\right]}\ar[ll]|-{\scriptstyle \id}
}
\]
and the two sorts
\[
\xyR{3pc}\xyC{3pc}\xymatrix{{\scriptstyle \left[n+1\right]}\ar[rr]|-{\sf{split}_{\left[n+1\right]}^{k}} &  & {\scriptstyle \left[1\right]}\ar[dd]|-{\id;\left(\left\{ 1\right\} \right)}\\
{\scriptstyle \left[n\right]}\ar[u]|-{d^{k}}\ar[d]|-{\id;\left(!,\dots,!,\left\{ 1\right\} \otimes0,!,\dots,!\right)} & \mathrm{(III)}\\
{\scriptstyle \left[n\right];\left(0,\dots,0,\left[1\right]\otimes0,0,\dots,0\right)}\ar[rr]|-{\sf{split}_{\left[n+1\right]}^{k};\left(\sigma_{0}\right)} &  & {\scriptstyle \left[1\right];\left[1\right]}\\
{\scriptstyle \left[n\right]}\ar[u]|-{\id;\left(!,\dots,!,\left\{ 0\right\} \otimes0,!,\dots,!\right)}\ar[d]|-{d^{k}} & \mathrm{(IV)}\\
{\scriptstyle \left[n+1\right]}\ar[rr]|-{\sf{split}_{\left[n+1\right]}^{k+1}} &  & {\scriptstyle \left[1\right]}\ar[uu]|-{\id;\left(\left\{ 0\right\} \right)}
}
\]
commute.

For sort (I) it suffices to observe that 
\begin{eqnarray*}
\id\circ d^{k} & = & d^{k}\circ\id\circ\id\\
 & \mathrm{and}\\
\left(!,!\right)\circ!\circ\left(\left\{ 1\right\} \otimes0\right) & = & \left(!\times!\right)\circ\left(!,!\right)
\end{eqnarray*}
and for sort (II) see that 
\begin{eqnarray*}
\id\circ d^{k} & = & d^{k}\circ\id\circ\id\\
 & \mathrm{and}\\
\left(!,!\right)\circ!\circ\left(\left\{ 0\right\} \otimes0\right) & = & \left(!\times!\right)\circ\left(!,!\right)
\end{eqnarray*}

For sort (III) it suffices to observe that 
\begin{eqnarray*}
\sf{split}_{\left[n+1\right]}^{k} & = & \sf{split}_{\left[n+2\right]}^{k}\circ d^{k}\\
 & \mathrm{and}\\
\sigma_{0}\circ\left\{ 1\right\} \otimes0 & = & \left\{ 1\right\} 
\end{eqnarray*}
since $\sigma_{0}$ is defined to be $\id_{\left[1\right]}$. Likewise,
for sort (IV) it suffices to observe that 
\begin{eqnarray*}
\sf{split}_{\left[n+1\right]}^{k} & = & \sf{split}_{\left[n+2\right]}^{k+1}\circ d^{k}\\
 & \mathrm{and}\\
\sigma_{0}\circ\left\{ 0\right\} \otimes0 & = & \left\{ 0\right\} 
\end{eqnarray*}
again since $\sigma_{0}=\id_{\left[1\right]}$.

Lastly, to check that Diagram \ref{eq: simplicial case folding diamond}
commutes, it suffices to observe that the diagrams 
\[
\xyR{3pc}\xyC{9pc}\xymatrix{ & {\scriptstyle \left[n\right]}\ar[dl]|-{d^{0}}\ar[d]|-{d^{0}}\ar[drr]|-{\left\{ 1\right\} }\\
{\scriptstyle \left[1+n\right]} & {\scriptstyle \left[1+n\right]}\ar[l]|-{\id}\ar[r]|-{\sf{split}_{\left[n+1\right]}^{1};\left(!\right)} & {\scriptstyle \left[1\right]}\ar[r]|-{\id;\left\{ 1\right\} } & {\scriptstyle \left[1\right];\left[1\right]}
}
\]
 and 
\[
\xyR{3pc}\xyC{9pc}\xymatrix{{\scriptstyle \left[n+1\right]} & {\scriptstyle \left[n+1\right]}\ar[l]|-{\id}\ar[r]|-{\sf{split}_{\left[n+1\right]}^{1};\left(!\right)} & {\scriptstyle \left[1\right]}\ar[r]|-{\id;\left\{ 0\right\} } & {\scriptstyle \left[1\right];\left[1\right]}\\
 & {\scriptstyle \left[n\right]}\ar[ul]|-{d^{n+1}}\ar[u]|-{d^{n+1}}\ar[urr]|-{\left\{ 0\right\} }
}
\]
commute.
\end{proof}
\begin{prop}
\label{prop:General Case GRay Span}For all $\left[n\right];\left(A_{1},\dots,A_{n}\right)$
of $\Theta$, the following diagram commutes.

\begin{equation}
\xyR{3pc}\xyC{3pc}\xymatrix{{\scriptstyle \left[n+1\right];\left(0,A_{1},A_{2},\dots,A_{n}\right)} &  & {\scriptstyle \left[n+1\right];\left(0,A_{1},A_{2},\dots,A_{n}\right)}\ar[ll]|-{\scriptscriptstyle \id}\ar[rr]|-{\scriptscriptstyle \sf{split}_{\left[n+1\right]}^{1};\left(!\right)} &  & {\scriptstyle \left[1\right];\left(0\right)}\ar[dd]|-{\scriptstyle \id;\left(\left\{ 1\right\} \right)}\\
 &  & {\scriptstyle \left[n\right];\left(A_{1},A_{2},\dots,A_{n}\right)}\ar[u]|-{\scriptstyle d^{1};\left(\left(!,\id\right),\id,\dots,\id\right)}\ar[d]|-{\scriptstyle \id;\left(\left\{ 1\right\} \otimes A_{1},\id,\dots,\id\right)}\\
{\scriptstyle \left[n\right];\left(A_{1},A_{2},\dots,A_{n}\right)}\ar[uu]|-{\scriptscriptstyle d^{1};\left(\left(!,\id\right),\id,\dots,\id\right)}\ar[dd]|-{\scriptscriptstyle d^{1};\left(\left(\id,!\right),\id,\dots,\id\right)} &  & {\scriptstyle \left[n\right];\left(\left[1\right]\otimes A_{1},A_{2},\dots,A_{n}\right)}\ar[ll]|-{\scriptstyle \id;\left(\pr_{2}\circ\kappa_{A_{k}},\id,\dots,\id\right)}\ar[rr]|-{\scriptstyle {\scriptscriptstyle \sf{split}_{\left[n+1\right]}^{1};\left(\sigma_{A_{1}}\right)}} &  & {\scriptstyle \left[1\right];\left[1\right];\left(A_{1}\right)}\\
 & \mathrm{} & {\scriptstyle \left[n\right];\left(A_{1},A_{2},\dots,A_{n}\right)}\ar[u]|-{\scriptstyle \id;\left(\left\{ 0\right\} \otimes A_{1},!,\dots,!\right)}\ar[d]|-{\scriptstyle d^{1};\left(\left(\id,!\right),\id,\dots,\id\right)}\\
{\scriptstyle \left[n+1\right];\left(A_{1},0,A_{2},\dots,A_{n}\right)}\ar[d] &  & {\scriptstyle \left[n+1\right];\left(A_{1},0,A_{2},\dots,A_{n}\right)}\ar[ll]|-{\scriptstyle \id}\ar[rr]|-{\scriptstyle {\scriptscriptstyle \sf{split}_{\left[n+1\right]}^{2};\left(!\right)}}\ar[d] &  & {\scriptstyle \left[1\right];\left[0\right]}\ar[uu]|-{\scriptscriptstyle \id;\left(\left\{ 0\right\} \right)}\ar[d]\\
{\scriptstyle \vdots} &  & {\scriptstyle \vdots} &  & \vdots\\
{\scriptstyle \left[n+1\right];\left(A_{1},\dots,A_{n-1},0,A_{n}\right)}\ar[u] &  & {\scriptstyle \left[n+1\right];\left(A_{1},\dots,A_{n-1},0,A_{n}\right)}\ar[ll]|-{\id}\ar[rr]|-{\scriptscriptstyle \sf{split}_{\left[n+1\right]}^{n};\left(!\right)}\ar[u] &  & \ar[u]{\scriptstyle \left[1\right];\left(0\right)}\ar[dd]|-{\scriptscriptstyle \id;\left(\left\{ 1\right\} \right)}\\
 & \mathrm{} & {\scriptstyle \left[n\right];\left(A_{1},\dots,A_{n-1},A_{n}\right)}\ar[u]|-{\scriptstyle d^{n};\left(\id,\dots,\id,\left(!,\id\right)\right)}\ar[d]|-{\scriptstyle \id;\left(\id,\dots,\id,\left\{ 1\right\} \otimes A_{n}\right)}\\
{\scriptstyle \left[n\right];\left(A_{1},\dots,A_{n-1}A_{n}\right)}\ar[uu]|-{\scriptscriptstyle d^{n};\left(\id,\dots,\id,\left(!,\id\right)\right)}\ar[dd]|-{\scriptscriptstyle d^{n};\left(\id,\dots,\id,\left(\id,!\right)\right)} &  & {\scriptstyle \left[n\right];\left(A_{1},\dots,A_{n-1},\left[1\right]\otimes A_{n}\right)}\ar[ll]|-{\scriptstyle \id;\left(\id,\dots,\id,\pr_{2}\circ\kappa_{A_{n}}\right)}\ar[rr]|-{\scriptscriptstyle \sf{split}_{\left[n+1\right]}^{n+1};\left(\sigma_{A_{n}}\right)} &  & {\scriptstyle \left[1\right];\left[1\right];\left(A_{n}\right)}\\
 & \mathrm{} & {\scriptstyle \left[n\right];\left(A_{1},\dots,A_{n}\right)}\ar[u]|-{\scriptscriptstyle \id;\left(\left\{ 0\right\} \otimes A_{1},!,\dots,!\right)}\ar[d]|-{\scriptscriptstyle d^{n};\left(\id,\dots,\id,\left(\id,!\right)\right)}\\
{\scriptstyle \left[n+1\right];\left(A_{1},\dots,A_{n-1},A_{n},0\right)} &  & {\scriptstyle \left[n+1\right];\left(A_{1},\dots,A_{n-1},A_{n},0\right)}\ar[ll]|-{\scriptscriptstyle \id}\ar[rr]|-{\scriptscriptstyle \sf{split}_{\left[n+1\right]}^{n+1};\left(!\right)} &  & {\scriptstyle \left[1\right];\left[0\right]}\ar[uu]|-{\scriptscriptstyle \id;\left(\left\{ 0\right\} \right)}
}
\label{eq: general case gray span diagram}
\end{equation}
Moreover, since:
\begin{itemize}
\item the colimit of the first column is $\left[1\right]\times\left[n\right];\left(A_{1},A_{2},\dots,A_{n}\right)$;
\item the colimit of the second column is $\left[1\right]\otimes\left[n\right];\left(A_{1},A_{2},\dots,A_{n}\right)$;
and
\item the colimit on the right is the globular sum decomposition for $\left[1\right];\left[n\right];\left(A_{1},A_{2},\dots,A_{n}\right)$;
\end{itemize}
this diagram induces maps (the solid arrows) for which the diagram
(solid and dashed arrows) commutes as follows.
\begin{equation}
\xyR{3pc}\xyC{9pc}\xymatrix{ & {\scriptstyle \left[n\right];\left(A_{1},A_{2},\dots,A_{n}\right)}\ar@{-->}[dl]|-{\left\{ 1\right\} \times\left[n\right];\left(A_{1},A_{2},\dots,A_{n}\right)}\ar@{-->}[d]|-{\left\{ 1\right\} \otimes\left[n\right];\left(A_{1},A_{2},\dots,A_{n}\right)}\ar@{-->}[dr]|-{\left\{ 1\right\} }\\
{\scriptstyle \left[1\right]\times\left[n\right];\left(A_{1},A_{2},\dots,A_{n}\right)} & {\scriptstyle \left[1\right]\otimes\left[n\right];\left(A_{1},A_{2},\dots,A_{n}\right)}\ar[l]|-{\kappa_{\left[n\right];\left(A_{1},A_{2},\dots,A_{n}\right)}}\ar[r]|-{\sigma_{\left[n\right];\left(A_{1},A_{2},\dots,A_{n}\right)}} & {\scriptstyle \left[1\right];\left[n\right];\left(A_{1},A_{2},\dots,A_{n}\right)}\\
 & {\scriptstyle \left[n\right];\left(A_{1},A_{2},\dots,A_{n}\right)}\ar@{-->}[ul]|-{\left\{ 0\right\} \times\left[n\right];\left(A_{1},A_{2},\dots,A_{n}\right)}\ar@{-->}[u]|-{\left\{ 0\right\} \otimes\left[n\right];\left(A_{1},A_{2},\dots,A_{n}\right)}\ar@{-->}[ur]|-{\left\{ 0\right\} }
}
\label{eq: general case folding diamond-1}
\end{equation}
\end{prop}

\begin{proof}
To show that Diagram \ref{eq: general case gray span diagram} commutes
it suffices to check the commutativity of four sorts of squares, for
$1\leq k\leq n$: the two sorts
\[
\xyR{3pc}\xyC{3pc}\xymatrix{{\scriptstyle \left[n+1\right];\left(A_{1},\dots,A_{k-1},0,A_{k},A_{k+1},\dots,A_{n}\right)} &  & {\scriptstyle \left[n+1\right];\left(A_{1},\dots,A_{k-1},0,A_{k},A_{k+1},\dots,A_{n}\right)}\ar[ll]|-{\id}\\
 & \mathrm{(I)} & {\scriptstyle \left[n\right];\left(A_{1},\dots,A_{k-1},A_{k},A_{k+1},\dots,A_{n}\right)}\ar[u]|-{d^{1};\left(\id,\dots,\id,\left(!,\id\right),\id,\dots,\id\right)}\ar[d]|-{\id;\left(\id,\dots,\id,\left\{ 1\right\} \otimes A_{1},\id,\dots,\id\right)}\\
{\scriptstyle \left[n\right];\left(A_{1},\dots,A_{k-1},A_{k},A_{k+1},\dots,A_{n}\right)}\ar[uu]|-{d^{k};\left(\id,\dots,\id,\left(!,\id\right),\id,\dots,\id\right)}\ar[dd]|-{d^{1};\left(\id,\dots,\id,\left(\id,!\right),\id,\dots,\id\right)} &  & {\scriptstyle \left[n\right];\left(A_{1},\dots,A_{k-1},\left[1\right]\otimes A_{k},A_{k+1},\dots,A_{n}\right)}\ar[ll]|-{\id;\left(\id,\dots,\id,\kappa_{A_{k}},\id,\dots,\id\right)}\\
 & \mathrm{(II)} & {\scriptstyle \left[n\right];\left(A_{1},\dots,A_{n}\right)}\ar[u]|-{\id;\left(\left\{ 0\right\} \otimes A_{1},!,\dots,!\right)}\ar[d]|-{d^{1};\left(\id,\dots,\id,\left(\id,!\right),\id,\dots,\id\right)}\\
{\scriptstyle \left[n+1\right];\left(A_{1},\dots,A_{k-1},A_{k},0,A_{k+1},\dots,A_{n}\right)} &  & {\scriptstyle \left[n+1\right];\left(A_{1},\dots,A_{k-1},A_{k},0,A_{2},\dots,A_{n}\right)}\ar[ll]|-{\id}
}
\]
and the two sorts
\[
\vcenter{\vbox{\xyR{3pc}\xyC{3pc}\xymatrix{{\scriptstyle \left[n+1\right];\left(A_{1},\dots,A_{k-1},0,A_{k},A_{k+1}\dots,A_{n}\right)}\ar[rr]|-{\sf{split}_{\left[n+1\right]}^{k};\left(!\right)} &  & {\scriptstyle \left[1\right];\left(0\right)}\ar[dd]|-{\id;\left(\left\{ 1\right\} \right)}\\
{\scriptstyle \left[n\right];\left(A_{1},\dots,A_{k-1},A_{k},A_{k+1},\dots,A_{n}\right)}\ar[u]|-{d^{k};\left(\id,\dots,\id,\left(!,\id\right),\id,\dots,\id\right)}\ar[d]|-{\id;\left(\id,\dots,\id,\left\{ 1\right\} \otimes A_{k},\id,\dots,\id\right)} & \mathrm{(III)}\\
{\scriptstyle \left[n\right];\left(A_{1},\dots,A_{k-1},\left[1\right]\otimes A_{k},A_{k+1},\dots,A_{n}\right)}\ar[rr]|-{\sf{split}_{\left[n+1\right]}^{k};\left(\sigma_{A_{k}}\right)} &  & {\scriptstyle \left[1\right];\left[1\right];\left(A_{k}\right)}\\
{\scriptstyle \left[n\right];\left(A_{1},\dots,A_{k-n},A_{k},A_{k+1},\dots,A_{n}\right)}\ar[u]|-{\id;\left(\id,\dots,\id,\left\{ 0\right\} \otimes A_{1},\id,\dots,\id\right)}\ar[d]|-{d^{k};\left(\id,\dots,\id,\left(\id,!\right),\id,\dots,\id\right)} & \mathrm{(IV)}\\
{\scriptstyle \left[n+1\right];\left(A_{1},\dots,A_{k-1},A_{k},0,A_{k+1},\dots,A_{n}\right)}\ar[rr]|-{\sf{split}_{\left[n+1\right]}^{k+1};\left(!\right)} &  & {\scriptstyle \left[1\right];\left[0\right]}\ar[uu]|-{\id;\left(\left\{ 0\right\} \right)}
}
}}
\]
The commutation argument for sorts (I) and (II) are nearly identical.
See that squares of sort (I) commute as (left-hand side - counter-clockwise,
right-hand side - clockwise):

\begin{eqnarray*}
\id_{\left[n+1\right]}\circ d^{k} & = & \id_{\left[n+1\right]}\circ\id_{\left[n+1\right]}\circ d^{k}\\
 & \mathrm{and}\\
\id_{A_{1}}\circ\id_{A_{1}} & = & \id_{A_{1}}\circ\id_{A_{1}}\circ\id_{A_{1}}\\
 & \vdots\\
\id_{A_{k-1}}\circ\id_{A_{k-1}} & = & \id_{A_{k-1}}\circ\id_{A_{k-1}}\circ\id_{A_{k-1}}\\
\left(!\times\id_{A_{k}}\right)\circ\left(!,\id_{A_{k}}\right) & = & \left(!,\id\right)\circ\left(\kappa_{A_{k}}\right)\circ\left(\left\{ 1\right\} \otimes A_{k}\right)\\
\id_{A_{k+1}}\circ\id_{A_{k+1}} & = & \id_{A_{k+1}}\circ\id_{A_{k+1}}\circ\id_{A_{k+1}}\\
 & \vdots\\
\id_{A_{n}}\circ\id_{A_{n}} & = & \id_{A_{n}}\circ\id_{A_{n}}\circ\id_{A_{n}}
\end{eqnarray*}
where the only non-obvious equality will follow by recursion to the
case where $A_{k}$ is of height $1$ - Lemma \ref{lem:Simplicial Case Gray Span}.
Nearly the same computation provides the commutativity of the squares
of sort (II), with 
\[
\left(!\times\id_{A_{k}}\right)\circ\left(!,\id_{A_{k}}\right)=\left(!,\id\right)\circ\left(\kappa_{A_{k}}\right)\circ\left(\left\{ 0\right\} \otimes A_{k}\right)
\]
 in place of 
\[
\left(!\times\id_{A_{k}}\right)\circ\left(!,\id_{A_{k}}\right)=\left(!,\id\right)\circ\left(\kappa_{A_{k}}\right)\circ\left(\left\{ 1\right\} \otimes A_{k}\right)
\]
 is the computation. Again, this follows by recursion to the case
where $A_{k}$ is of height $1$ - Lemma \ref{lem:Simplicial Case Gray Span}.

The commutation for sorts (III) and (IV) are likewise nearly identical.
See that sort (III) commutes as
\begin{eqnarray*}
\id_{\left[1\right]}\circ\sf{split}_{\left[n+1\right]}^{k}\circ d^{k} & = & \sf{split}_{\left[n+1\right]}^{k}\\
 & \mathrm{and}\\
\left\{ 1\right\} \circ!\circ\left(!,\id_{A_{K}}\right) & = & \sigma_{A_{k}}\circ\left\{ 1\right\} \otimes A_{k}
\end{eqnarray*}
 by recursion to the case where $A_{k}$ is of height 1 Lemma \ref{lem:Simplicial Case Gray Span}.
The commutation for squares of sort (IV) is similar, as promised.
Indeed squares of sort (IV) commute as 
\begin{eqnarray*}
\id_{\left[1\right]}\circ\sf{split}_{\left[n+1\right]}^{k+1}\circ d^{k} & = & \sf{split}_{\left[n+1\right]}^{k}\\
 & \mathrm{and}\\
\left\{ 0\right\} \circ!\circ\left(!,\id_{A_{K}}\right) & = & \sigma_{A_{k}}\circ\left\{ 0\right\} \otimes A_{k}
\end{eqnarray*}
which again follows by recursion to $A_{k}$ of height 1 - Lemma \ref{lem:Simplicial Case Gray Span}.

Lastly, to check that Diagram \ref{eq: simplicial case folding diamond}
commutes, it suffices to observe that the diagrams
\[
\xyR{3pc}\xyC{6pc}\xymatrix{ & {\scriptstyle \left[n\right];\left(A_{1},A_{2},\dots,A_{n}\right)}\ar[dl]|-{d^{0};\left(\id,\id,\dots,\id\right)}\ar[d]|-{d^{0};\left(\id,\id,\dots,\id\right)}\ar[drr]|-{\left\{ 1\right\} }\\
{\scriptstyle \left[1+n\right];\left(0,A_{1},A_{2},\dots,A_{n}\right)} & {\scriptstyle \left[1+n\right]\left(0,A_{1},A_{2},\dots,A_{n}\right)}\ar[l]|-{\id}\ar[r]|-{\sf{split}_{\left[n+1\right]}^{1};\left(!\right)} & {\scriptstyle \left[1\right]}\ar[r]|-{\id;\left\{ 1\right\} } & {\scriptstyle \left[1\right];\left[1\right];\left(A_{1}\right)}
}
\]
 and 
\[
\xyR{3pc}\xyC{6pc}\xymatrix{{\scriptstyle \left[n+1\right];\left(A_{1},\dots A_{n-1},A_{n},0\right)} & {\scriptstyle \left[n+1\right];\left(A_{1},\dots A_{n-1},A_{n},0\right)}\ar[l]|-{\id}\ar[r]|-{\sf{split}_{\left[n+1\right]}^{n+1};\left(!\right)} & {\scriptstyle \left[1\right]}\ar[r]|-{\id;\left\{ 0\right\} } & {\scriptstyle \left[1\right];\left[1\right];\left(A_{n}\right)}\\
 & {\scriptstyle \left[n\right];\left(A_{1},\dots A_{n-1},A_{n}\right)}\ar[ul]|-{d^{n+1};\left(\id,\dots,\id,\id\right)}\ar[u]|-{d^{n+1};\left(\id,\dots,\id,\id\right)}\ar[urr]|-{\left\{ 0\right\} }
}
\]
commute.
\end{proof}

\appendix

\section{\label{sec:Steiner-Complexes-and-Steiner =00005Comega-categories}Steiner
Complexes and Steiner $\omega$-categories: Bases, Loop-free Bases,
etc.}

\global\long\def\defiff{\overset{\mathsf{def}}{\Leftrightarrow}}%

This appendix is a translation, done by Yuki Maehara and this author,
from French to English, of \cite{AraMaltsiniotis} Paragraphs 2.6
through 2.10.

\paragraph*{2.6}

A \textbf{basis} for an augmented directed complex $K$ is a graded
set $B=(B_{i})_{i\ge0}$ such that, for all $i\ge0$,
\begin{enumerate}
\item $B_{i}$ is a basis for the $\Z$-module $K_{i}$;
\item $B_{i}$ generates the sub-monoid $K_{i}^{*}$ of $K_{i}$.
\end{enumerate}
We will sometimes identify a basis $B=(B_{i})_{i\ge0}$ with the set
$\coprod_{i\ge0}B_{i}$. 

Let $K$ be an augmented directed complex. We define a preorder $\le$
on $K_{i}$ by setting 
\[
x\le y\quad\defiff\quad y-x\in K_{i}^{*}.
\]

It follows immediately that if $K$ admits a basis, the preorder is
in fact a partial order, and the elements of $B_{i}$ are the minimal
elements of $(K_{i}^{*}\setminus\{0\},\le)$. Thus if $K$ admits
a basis then that basis is unique.

We will say that an augmented directed complex $K$ is \textbf{based}
if it admits a (necessarily unique) basis.

\paragraph*{2.7 }

Fix an abelian group $A$ freely generated by a basis $B$. Let 
\[
x=\sum_{b\in B}x_{b}b
\]
 be an element in $A$. We define the \textbf{support} of $x$ to
be the set 
\[
\sf{supp}(x)=\{b\in B~|~x_{b}\neq0\}.
\]
 Denote by $A^{*}$ the sub-monoid of $A$ generated by $B$. We define
two elements $x_{+}$ and $x_{-}$ in $A^{*}$ by 
\[
x_{+}=\sum_{x_{b}>0}x_{b}b\quad\text{and}\quad x_{-}=-\sum_{x_{b}<0}x_{b}b.
\]
 We have $x=x_{+}-x_{-}$. 

In particular, if $K$ is an augmented directed complex admitting
a basis $B=(B_{i})_{i\ge0}$, for all $x$ in $K_{i}$ with $i\ge0$,
we have, in applying the preceding paragraph to the abelian group
$K_{i}$ given by the basis $B_{i}$, a notion of support of $x$
and elements $x_{+}$ and $x_{-}$ in $K_{i}^{*}$.

\paragraph*{2.8\label{par:2.8 - AraMaltsiniotisBracketForGenerationOfTables}}

Let $K$ be an augmented directed complex with a basis $B=(B_{i})_{i\ge0}$.
For $i\ge0$ and for $x\in K_{i}$, we define the matrix 
\[
\langle x\rangle=\tabld{\langle x\rangle}{i},
\]
 where $\langle x\rangle_{k}^{\epsilon}$ are defined by recursion
on $k$ from $i$ to $0$:
\begin{itemize}
\item $\langle x\rangle_{i}^{0}=x=\langle x\rangle_{i}^{1}$;
\item $\langle x\rangle_{k-1}^{0}=d(\langle x\rangle_{k}^{0})_{-}$ and
$\langle x\rangle_{k-1}^{1}=d(\langle x\rangle_{k}^{1})_{+}$ for
$0<k\le i$.
\end{itemize}
It is easy to see that this matrix is an $i$-arrow in $\nu(K)$ if
and only if, $x$ appears in $K_{i}^{*}$, and we have $e(\langle x\rangle_{0}^{0})=1=e(\langle x\rangle_{0}^{1})$.
We will also set $\langle x\rangle_{i}=x$ and $\langle x\rangle_{k}^{\epsilon}=0$
for $k>i$ and $\epsilon=0,1$. In the case when $\langle x\rangle$
is an $i$-arrows, this is compatible with the convention described
in Paragraph 2.4 of \cite{AraMaltsiniotis}. 

We will say that a basis $B$ for $K$ is \textbf{unital} if, for
all $i\ge0$ and for all $x\in B_{i}$, the matrix $\langle x\rangle$
is an $i$-arrow in $\nu(K)$, which amounts to $e(\langle x\rangle_{0}^{0})=1=e(\langle x\rangle_{0}^{1})$. 

We say that an augmented directed complex is \textbf{unitally based}
if it is based and its unique basis is unital. If an augmented directed
complex $K$ admits a unital basis, then for all elements $x$ ins
the basis for $K$, we call the cell $\langle x\rangle$ in $\nu(K)$
the \textbf{atom} associated to $x$.

\paragraph*{2.9}

Let $K$ be an augmented directed complex with a basis $B$. For $i\ge0$,
denote by $\le_{i}$ the smallest preorder on $B(=\coprod_{j}B_{j})$
satisfying: 
\[
x\le_{i}y\quad\text{if}\quad|x|>i,|y|>i,\text{ and }\sf{supp}(\langle x\rangle_{i}^{1})\cap\supp(\langle y\rangle_{i}^{0})\neq\varnothing.
\]
 We will say that the basis $B$ is \textbf{loop-free} if, for all
$i\ge0$, the preorder $\le_{i}$ is a partial order. 

We say an augmented directed complex \textbf{admits a loop-free basis}
if it admits a basis and its unique basis is loop-free.

\paragraph*{2.10}

We refer to an augmented directed complex with a unital, loop-free
basis as a \textbf{Steiner complex}. 

We call an $\omega$-category a \textbf{Steiner $\omega$-category}
if it is in the essential image of the functor $\nu:\CDA\to\StrCat$
restricted to the category of Steiner complexes. The following theorem
affirms that the functor $\nu$ induces an equivalence between the
categories of Steiner complexes and of Steiner $\omega$-categories.
\begin{thm}
(Steiner, 2.11 in \cite{AraMaltsiniotis}) For all Steiner complexes
$K$, the counit 
\[
\lambda(\nu(K))\to K
\]
 is an isomorphism. In particular, the restriction of the functor
$\nu:\CDA\to\StrCat$ to the full subcategory of Steiner complexes
is full-and-faithful.
\end{thm}

\begin{proof}
As cited in \cite{AraMaltsiniotis}, See Theorem 5.6 of \cite{Steiner1}.
\end{proof}
\begin{thm}
({[}Steiner{]}, 2.12 in \cite{AraMaltsiniotis}) Let $K$ be a Steiner
complex. Then the $\infty$-category $\nu(K)$ is freely generated
in the sense of polygraphs by the atoms $\langle x\rangle$ where
$x$ varies over the basis for $K$.
\end{thm}

\begin{proof}
See Theorem 6.1 of \cite{Steiner1}.
\end{proof}

\paragraph*{2.13}

Let $K$ be an augmented directed complex admitting a basis $B$.
Denote by $\le_{\N}$ the smallest preorder on $B$ satisfying 
\[
x\le_{\N}y\quad\text{if}\quad x\in\supp(d(y)_{-})\text{ or }y\in\supp(d(x)_{+}),
\]
 where, by convention, $d(b)=0$ if $b\in B_{0}$. We will say that
a basis $B$ is \textbf{\emph{strongly loop-free}} if the preorder
$\le_{N}$ is a partial order.
\begin{prop}
({[}Steiner{]}, 2.14 in \cite{AraMaltsiniotis}) Let $K$ be an augmented
directed complex with a basis $B$. If $B$ is strongly loop-free,
then it is also loop-free.
\end{prop}

\begin{proof}
See Proposition 3.7 of {[}Steiner{]}.
\end{proof}

\paragraph*{2.15}

An augmented directed complex with a unital, strongly loop-free basis
will be called a \textbf{strong Steiner complex}. By virtue of the
preceding paragraph, a strong Steiner complex is a Steiner complex.
We denote by $\sf{StrSt\CDA}$ the full subcategory of $\CDA$ spanned
by the strong Steiner complexes. 

By a \textbf{strong Steiner $\omega$-category}, we mean an $\omega$-category
in the essential image of the functor $\nu:\CDA\to\StrCat$ restricted
to the strong Steiner complexes. By virtue of Theorem 2.11, the functor
$\nu$ induces an equivalence between the category of strong Steiner
complexes and the category of strong Steiner $\omega$-categories.

\end{document}